\documentclass[10pt]{article}
\DeclareMathSymbol{\varChi}{\mathord}{letters}{88}

\usepackage[latin1]{inputenc}
\usepackage{amsfonts}
\usepackage{mathrsfs}
\usepackage{amsmath}
\usepackage{amssymb}

\usepackage{amsthm}

\usepackage{tensor}
 \usepackage{color}
\definecolor{blu}{rgb}{0.0,0.0,1.0}

\usepackage[T1]{fontenc}

\newcommand{\ud}{\,\mathrm{d}}
\def\K{\mathbb K}
\def\M{\mathbb M}
\def\R{\mathbb R}
\def\X{\mathbb X}
\def\Y{\mathbb Y}
\def\N{\mathbb N}
\def\W{\mathbb W}
\def\Z{\mathbb Z}
\def\A{\mathbb A}
\def\S{\mathbb S}
\def\V{\mathbb V}
\def\Q{\mathbb Q}
\def\P{\mathbb P}

\def \shd{\mathcal D}
\def \shl{\mathcal L}
\def \shp{\mathcal P}
\newtheorem{Theorem}{Theorem}[section]
\newtheorem{Definition}[Theorem]{Definition}
\newtheorem{Proposition}[Theorem]{Proposition}
\newtheorem{Lemma}[Theorem]{Lemma}
\newtheorem{Corollary}[Theorem]{Corollary}
\newtheorem{Hypothesis}[Theorem]{Hypothesis}
\newtheorem{Remark}[Theorem]{Remark}
\newtheorem{Example}[Theorem]{Example}
\newtheorem{Notation}[Theorem]{Notation}

\selectfont

\author{Giorgio Fabbri\footnote{EPEE, Universit\'e d'Evry-Val-d'Essonne (TEPP, FR-CNRS 3126), D\'epartement d'Economie, 4 bd.
 Fran\c{c}ois Mitterrand, 91025 Evry cedex, France. E-mail: \texttt{giorgio.fabbri@univ-evry.fr}}
 \; and \;  Francesco Russo\footnote{ENSTA ParisTech, Unit\'e de Math\'ematiques appliqu\'ees, 828, boulevard des Mar\'echaux,
F-91120 Palaiseau.
E-mail: francesco.russo@ensta-paristech.fr} }
\title{Infinite dimensional weak Dirichlet processes, stochastic PDEs and optimal control}
\date{February 27th 2014}

\begin{document}

\maketitle

\begin{abstract}
The present paper continues the study of infinite dimensional calculus via
regularization, started by C. Di Girolami and the second named author,
 introducing the notion of \emph{weak Dirichlet process} 
 in this context. Such a process $\X$, taking values in a Hilbert space $H$,
is the sum of a local martingale and a suitable {\it orthogonal} process. \\
The new concept is shown to be useful in several contexts and directions.
On one side,  the mentioned decomposition 
appears to be a  substitute of an It\^o's type formula 
applied to $f(t, \X(t))$ where $f:[0,T] \times H \rightarrow \R$ 
is a $C^{0,1}$ function and, on the other side, 
the idea of weak Dirichlet process fits the widely used 
notion of \emph{mild solution} for  stochastic evolution equations on
infinite dimensional Hilbert spaces, including several classes of
 stochastic partial differential equations (SPDEs).

As a specific application, we provide a verification theorem for stochastic optimal control problems whose state equation is an infinite dimensional stochastic evolution equation.
\end{abstract}

\bigskip

\bigskip


\medskip

{\bf 2010 AMS MATH CLASSIFICATION:}    60H05, 60H07, 60H10, 60H30, 91G80.

\bigskip

\section{Introduction}










Stochastic calculus  constitutes one of the basic tools for stochastic
 optimal control theory;
in particular the classical It\^o formula  enables to relate the solution of a control  problem in closed
 loop form with the  (smooth enough) solutions of the related
 Hamilton-Jacobi-Bellman (HJB) equation via some suitable \emph{verification theorems}.

For the finite dimensional systems the literature presents quite precise and general results, see e.g. \cite{FlemingRishel75, YongZhou99}. If the system is infinite dimensional, for instance if it is driven by a stochastic partial differential (SPDEs) or a stochastic delay differential equation, the situation is more complex 
especially when the value function of the problem is not regular enough.

\smallskip

This paper contributes to the subject providing an efficient (infinite dimensional) stochastic calculus which fits the structure of a mild solution of 
 a stochastic evolution equation in infinite dimension and  gives the possibility to prove a verification theorem for a class of stochastic 
optimal control problems with non-regular value function, refining previous results.


\bigskip

The contributions of the paper can be ascribed to the following three ``labels'': 
infinite dimensional stochastic calculus,  stochastic evolution equations in Hilbert spaces 
 and dynamic programming. In the next subsection we describe the state of the art, while in the following we will concentrate on the new results.

\subsection*{State of the art}

\paragraph*{Stochastic calculus}\hspace{-0.43cm}  .

{\it Stochastic calculus via regularization} for real processes was initiated in \cite{RussoVallois91} and \cite{RussoVallois93}. 
It is an efficient calculus for non-semimartingales whose related literature is  surveyed in \cite{RussoVallois07}.

Let $T > 0$. The processes will be indexed by $[0,T]$  adopting the convention described  Notation   \ref{Not0}. In the whole paper $s$
will be a real in $[0,T]$.
All the considered processes will be considered as measurable
from $[s,T] \times \Omega$ (equipped with the product
of the Borel $\sigma$-field of $[s,T]$ and the $\sigma$-algebra
of all events $\mathscr{F}$) and the Borel $\sigma$-field 
of the value space, for instance $\R$.
 Given an a.s. bounded [resp. continuous] real process $Y$ [resp. $X$], the forward integral of $Y$ with respect to $X$ 
and the covariation between $Y$ and $X$ are defined as follows.
Suppose that, for every $t \in [s,T]$, the limit $I(t)$ [resp. $C(t)$]
 in probability exists:
\begin{eqnarray} \label{DefRealIntCov}
I(t) : &=& \lim_{\epsilon \to 0^+} \int_s^t Y(r) \left ( \frac{X(r+\epsilon) - 
X(r)}{\epsilon} \right ) \ud r, t \in [s,T], \nonumber\\
& &  \\
C(t) :&=& \lim_{\epsilon \to 0^+} \int_{s}^{t} \frac{ \left (X({r+\epsilon})-
X({r}) \right )
\left (Y({r+\epsilon})-Y({r})\right )}{\epsilon}dr,  t \in [s,T].  \nonumber
\end{eqnarray}

If the random function $I$ [resp. $C$] admits a continuous version,  this
will be denoted  by  $\int_s^\cdot Y d^-X$ [resp. $[X,Y]$]. 
It is the {\it  forward integral} of $Y$ with respect to $X$ 
[resp. the {\it covariation} of $X$ and $Y$]. If $X$ is a real continuous semimartingale and $Y$ is a c\`adl\`ag
process which is progressively measurable 
 [resp.  a semimartingale], the integral $\int_s^\cdot Y d^-X$ 
[resp. the covariation $[X,Y]$] is the same as the classical It\^o's integral 
[resp. covariation].

The definition of $[X,Y]$ given above is slightly more general (weaker) than the one in
 \cite{RussoVallois07}. There the authors supposed that the convergence in (\ref{DefRealIntCov}) holds in the ucp (uniformly convergence in probability)
 topology.\footnote{Given a Banach space $B$ and a probability space 
$(\Omega, \mathbb{P})$ a family of processes $\X^\epsilon\colon \Omega
 \times [s, T] \to B$ is said to converge in the ucp sense to
 $\X\colon \Omega \times [s, T] \to B$, when $\epsilon$ goes to zero,
if $
\lim_{\epsilon \to 0} 
 \sup_{t\in [s,T]} |\X^\epsilon_t - \X_t|_B = 0,
$
in probability.}
In this work we use the weak definition for the real case, i.e. when both
 $X$ and $Y$ are real, and the strong definition, via ucp convergence,
 when either $X$ or $Y$ is not one-dimensional. When $X=Y$ the two definitions are equivalent taking into account Lemma 2.1 of \cite{RussoVallois07}.

Real  processes  $X$ for which $[X,X]$ exists are called 
{\it finite quadratic variation  processes}.
A rich class of finite quadratic variation processes
is provided by Dirichlet processes.  Let $(\mathscr{F}_t, t \in [0,T])$
be a fixed filtration, fulfilling the usual conditions. A real process $X$ is said to be 
{\it Dirichlet} (or {\it F\"ollmer-Dirichlet}) if it is the sum of a local martingale $M$ and a 
{\it zero quadratic variation process} $A$, i.e. such that $[A,A] = 0$.
Those processes were defined by H. F\"ollmer \cite{FoDir} 
using  limits of discrete sums. 
A significant generalization, due to \cite{errami, GozziRusso06}, is the notion
of {\it weak Dirichlet process}, extended to the case of
jump processes in \cite{coquet}.

\begin{Definition}
\label{def:Dirweak}
A real process $X \colon [s,T]\times \Omega \to \mathbb{R}$ is called weak Dirichlet process if it can be written as
\begin{equation}
\label{eq:dec2}
X=M+A,
\end{equation}
where
\begin{itemize}
\item[(i)]  $M$ is a local martingale,
\item[(ii)]  $A$ is a process such that $\left[ A,N\right] =0$ for every 
continuous local martingale $N$ and $A(s)=0$.
\end{itemize}
\end{Definition}
Obviously a semimartingale  is a weak Dirichlet process.
In Remark 3.5 of \cite{GozziRusso06}, appears the following 
important statement.
\begin{Proposition}
\label{rm:ex418}
The decomposition described in Definition \ref{def:Dirweak} is unique.
\end{Proposition}

Elements of calculus via regularization were extended to Banach space valued
processes in a series of papers, see e.g. \cite{DiGirolamiRusso11, 
DiGirolamiRusso09, DGRNote, DiGirolamiRusso11Fukushima}.
We go on introducing two classical notions of 
stochastic calculus in Banach spaces, which appear in
\cite{MetivierPellaumail80} and
 \cite{Dinculeanu00}: the scalar  and tensor quadratic variations.
 We propose here a regularization
approach, even though, originally they appeared 
in a discretization framework.
The two monographs above use the  term {\it real}
instead of {\it scalar}; we have decided to change it 
to avoid confusion with the quadratic variation of real processes.


\begin{Definition}
\label{def:per-nota}
  Consider a separable Banach spaces $B$.
We say that a process $\X\colon [s,T]\times \Omega \to B$,
a.s. square integrable,
 admits a {\bf scalar quadratic variation} if, for any $t\in [s,T]$,
 the limit, for $\epsilon\searrow 0$ of 
\[
[\X,\X]^{\epsilon, \mathbb{R}}(t):= \int_{s}^{t} \frac{\left 
|\X({r+\epsilon})-\X({r}) \right |^2_B}{\epsilon} dr
\]
exists in probability and it admits a continuous version. The limit process 
is called scalar quadratic variation of $\X$ 
and it is denoted  by  $[\X,\X]^{\mathbb{R}}$. 
\end{Definition}
\begin{Remark}
\label{R15}
The definition  above is equivalent to the one contained in
 \cite{DiGirolamiRusso11}. 
In fact, previous convergence in probability implies the ucp convergence, 
since the $\epsilon$-approximation processes are increasing and
so Lemma 3.1 in \cite{RussoVallois00} can be applied.
\end{Remark}

\begin{Proposition} \label{RBVQV}
Let $B$ be a separable Banach space.
A continuous $B$-valued process with bounded variation
admits a zero scalar quadratic variation.
In particular, a process $\X$ of the type $\X(t) = \int_s^t b(r)dr$, where 
$b$ is a $B$-valued 
measurable process
has a zero scalar quadratic variation. 
\end{Proposition}
\begin{proof}
The proof is very similar to the one related to the case when $B = \R$,
which was the object of  Proposition 1 
part 7-b, see \cite{RussoVallois07}.
\end{proof}


From \cite{DiGirolamiRusso11} we borrow the following definition.
\begin{Definition}
\label{def:tensorcovariation}
Consider two separable Banach spaces $B_1$ and $B_2$.
Suppose that either $B_1$ or $B_2$ is different from $\mathbb{R}$.
Let  $\X\colon [s,T]\times \Omega \to B_1$ and 
$\Y\colon [s,T]\times \Omega \to B_2$  be two
a.s. square integrable processes.
We say that $(\X,\Y)$ admits a {\bf tensor covariation} if the limit,
 for $\epsilon\searrow 0$ of the $B_1\hat\otimes_\pi B_2$-valued processes
\[
[\X,\Y]^{\otimes, \epsilon}:= \int_{s}^{\cdot} \frac{ \left (\X({r+\epsilon})-\X({r}) \right ) \otimes
\left (\Y({r+\epsilon})-\Y({r})\right )}{\epsilon}dr,
\]
exists ucp. The limit process is called {\bf tensor covariation} of $(\X,\Y)$ 
and is denoted  by  $[\X,\Y]^\otimes$. The tensor covariation $[\X,\X]^\otimes$ is called {\bf tensor quadratic variation} of $\X$ and
is denoted  by  $[\X]^\otimes$.
\end{Definition}


\medskip

The concepts  of scalar and  tensor quadratic variation
are however too strong for the applications: 
indeed several interesting examples
of Banach (or even Hilbert) space valued processes have no tensor quadratic variation.
A Banach space valued example is  the $C([-\tau,0])$-valued process $\X$
defined as the  frame (or window) of a standard Brownian
motion $W$: $\X(t) (x): = W({t+x}), \ x \in [-\tau,0]$.
It is neither a semimartingale,
nor a process with scalar quadratic variation process, see considerations
after Remarks 1.9 and 
Proposition 4.5
of \cite{DiGirolamiRusso11}. 
A second example, that indeed constitutes a main motivation for the present
 paper, 
is given by mild solutions of classical  stochastic evolution equations in infinite dimensions (including SPDEs): they have no scalar quadratic variation even if 
driven by a one-dimensional Brownian motion.

\smallskip
 The idea of Di Girolami and Russo was to introduce
 a suitable space $\chi$ continuously embedded into
the dual of the  projective tensor space
$B_1 \hat \otimes_\pi B_2,$ called {\bf Chi-subspace}.
 \\
  $\chi$ is a characteristics of their notion of quadratic variation, recalled in  Section \ref{sub4.2}.
When $\chi$ is the full space $(B_1 \hat \otimes_\pi B_2)^\ast$ the $\chi$-quadratic variation is called
{\bf global quadratic variation}. 
Following the approach of Di Girolami and Russo, see for instance
Definition 3.4 of \cite{DiGirolamiRusso11Fukushima},
 we make use of a the notion of $\chi$-covariation $[\X,\Y]_\chi$ when 
$\chi \subseteq (B_1\hat\otimes_\pi B_2)^*$ for two processes 
$\X$ and $\Y$ with values respectively in separable
 Hilbert spaces $B_1$ and $B_2$.  That notion is recalled
in Definition \ref{def:covariation}.

\smallskip

\cite{DiGirolamiRusso09} introduces a  (real valued) forward integral,
denoted by
$\int_{s}^{t}\tensor[_{B^*}]{\langle \Y({r}), d^{-}\X({r}) \rangle}{_B}$,
in the case when the integrator $\X$ takes  values in 
a Banach space $B$ and the  integrand $\Y$ is $B^*$-valued.
This appears as  a natural generalization of the first line 
of \eqref{DefRealIntCov}.
That notion is generalized in Definition  \ref{def:forward-integral}
 for operator-valued integrands; in that case, 
this produces a   Hilbert valued
 forward integral.

The It\^o formula for processes $\X$ having a
$\chi$-quadratic variation is given in Theorem 5.2 of 
\cite{DiGirolamiRusso11}. 
Let $F:[0,T] \times B \rightarrow \R$ or class $C^{1,2}$
such that $\partial_{xx}^2F \in C([0,T] \times B; \chi)$. 
Then, for every $t \in [s,T]$,
\begin{multline}\label{DGR-Ito}
F(t,\X({t}))=F(s, \X({s}))+\int_{s}^{t}\partial_{r}F(r,\X({r}))ds+
\int_{s}^{t}\tensor[_{B^*}]{\langle \partial_{x} 
F(r,\X({r})),d^{-}\X({r})\rangle}{_{B}} \\
+\frac{1}{2}\int_{s}^{t} \tensor[_{\chi}]{\langle \partial^2_{xx}
F(r,\X({r})),
d\widetilde{[\X,\X]}_{r}\rangle}{_{\chi^{\ast}}} \ a.s.
\end{multline}
We will introduce the notation $\widetilde{[\X,\X]}$ in Section
 \ref{SecStochInt}.

\paragraph*{Stochastic evolution equations in Hilbert spaces} \hspace{-0.43cm}  .

The class of stochastic evolution equations in infinite dimension 
that we consider can be seen as the abstract and unified formulation of 
several classes of stochastic partial differential differential equations (SPDEs) and stochastic functional differential equations.
They model a significant range of systems
 modeling phenomena arising in very different fields as physics, economics, 
physiology, population growth and migration etc.


The abstract formulation of the stochastic
 evolution equation introduced in Section \ref{sec:SPDEs} is characterized
 by an abstract generator of a $C_0$-semigroup $A$ and Lipschitz coefficients $b$ and $\sigma$.
It appears as
\begin{equation}\label{eq:SPDEIntro}
\left \{
\begin{array}{l}
\ud \X(t) = \left ( A\X(t)+ b(t,\X(t)) \right ) \ud t + \sigma(t,\X(t)) \ud
 \W_Q(t)\\[5pt]
\X(s)=x,
\end{array}
\right.
\end{equation}
where $\W$ is a $Q$-Wiener process with respect to some covariance
operator $Q$.
 As a particular case, when $Q$ is the identity operator, $\W_Q$ represents
a space-time white noise.


An SPDE is a partial differential equation  with random forcing terms or
 coefficients. As described for example in Part III of \cite{DaPratoZabczyk96},
 several families of SPDEs can be reformulated as stochastic evolution equations and then can be studied in the general abstract setting; of course for any of them the specification of the generator $A$ and of the functions $b$ and $\sigma$ are different. The reformulation can be done easier when the SPDE is in the form of deterministic PDE perturbed by a Gaussian noise involving an infinite dimensional Wiener process $\W_Q$ as a multiplicative factor and/or an additive term. Among the SPDEs that can be expressed in the formalism of stochastic evolution equations in 
infinite dimension we recall the stochastic heat and parabolic equations (even with boundary  noise), wave equations, reaction-diffusion equations, first order equations, Burgers equations, Navier-Stokes equations and Duncan-Mortensen-Zakai equations;
 however not all of them fulfill the assumpions of Section \ref{sec:SPDEs}.  In the same way the abstract formulation in the form of
 stochastic evolution equation in infinite dimension can be shown to include classes of stochastic delay differential equations and
 neutral differential equations, see again \cite{DaPratoZabczyk96} and the contained references.  We will recall some examples 
in Remark \ref{rm:examples-spde} of Section \ref{sec:SPDEs}.



There are several different possible ways to define  solutions of stochastic infinite dimensional evolutions equation and SPDEs:
 strong solutions (see e.g. \cite{DaPratoZabczyk92} Section 6.1), 
variational solutions (see e.g. \cite{PrevotRockner07}),
 martingale solutions, see \cite{MikuleviciusRozovskii99} and so on.

We will make use of the notion of \emph{mild solution} 
(see \cite{DaPratoZabczyk92} Chapter 7 or \cite{GawareckiMandrekar10} 
Chapter 3) where the solution of (\ref{eq:SPDEIntro})
 is defined (using, formally, a ``variations of parameters'' arguments) as the solution of the integral equation
\[
\X(t) = e^{(t-s)A}x + \int_s^t e^{(t-r)A} b(r,\X(r)) \ud r + \int_s^t e^{(t-r)A} \sigma(r,\X(r)) \ud \W_Q(r).
\]
This concept is widely used in the literature.

\paragraph*{Optimal control: dynamic programming and verification theorems}\hspace{-0.43cm}  .

As in the study of finite-dimensional stochastic (and non-stochastic)
 optimal control problem, the dynamic programming approach connects
 the study of the  minimization problem with the analysis of the 
related Hamilton-Jacobi-Bellman (HJB) equation: given a solution of 
HJB and a certain
number of hypotheses, the optimal control can be found in feedback form (i.e.
 as a function of the state) through a so called \emph{verification theorem}.
The idea is to identify a solution of the HJB equation with the value function
of the control problem.
When the state equation of
the optimal control problem is an infinite dimensional stochastic
 evolution equation, the related HJB is of second order and
 infinite dimensional.
The  simplest procedure  for establishing a verification theorem consists in considering 
the \emph{regular case} where the solution is assumed to have all the
 regularity needed to give meaning to all the terms appearing in the HJB
in the classical sense: it needs to be $C^1$ in the time variable and $C^2$ in the state variable.
Since in many interesting cases the value function does not have all the required regularity,
 several definitions of (more general) solutions have been introduced for the HJB equation.
 The various possibilities can be classified as follows.

\begin{itemize}
 \item[]
\emph{Strong solutions}: in this approach, first introduced in \cite{BarbuDaPrato83book}, the solution is defined as a proper limit of solutions of regularized problems. Verification results in this framework are given for example in \cite{Gozzi96} and \cite{GozziRouy96}, see \cite{Cerrai01, Cerrai01-40}, for the reaction-diffusion case.
 \item[]
\emph{Viscosity solutions}: in this case the solution is defined using test functions that locally ``touch'' the candidate solution. The viscosity solution approach was first adapted to the second order Hamilton Jacobi equation in Hilbert space in \cite{Lions83I,Lions83II,Lions83III} and then, for the ``unbounded'' case (i.e. including the unbounded operator $A$ appearing in \eqref{eq:SPDEIntro})
 in \cite{Swiech94}. As far as we know, differently from the finite-dimensional case 
there are no verification theorems available  for the infinite-dimensional case.
 \item[]
\emph{$L^2_{\mu}$ approach}: it was introduced in \cite{Ahmed03,GoldysGozzi06}, 
see also \cite{ChowMenaldi97, Ahmed01}. In this case the value function is
found in the space $L^2_{\mu}(H)$, where $\mu$ is an invariant measure for an associated uncontrolled process. The paper \cite{GoldysGozzi06} contains as well an excellent literature survey.  
\item[]
\emph{Backward approach}: it can be applied when the mild solution of the HJB can be represented using the solution of a forward-backward system and
  makes it possible  to find an optimal control in feedback form. It was introduced in \cite{Quenez97} and developed in  \cite{FuHuTe-quadr, FuhrmanTessitore02-ann, FuTe-ell}, see \cite{DeFuTe, FuhrmanMasieroTessitore11} 
for other particular cases.
\end{itemize}

\smallskip

The method we use in the present work to prove the verification theorem does not belong to any of the previous categories even if we use a strong solution approach to define the solution of the HJB. In the sequel of this introduction, we will be more precise.

\subsection*{The contributions of the work}

The novelty of the present paper arises at the 
three levels mentioned above: stochastic calculus,  infinite dimensional stochastic differential equations  and stochastic optimal control.
Indeed stochastic optimal control for infinite dimensional problems
is also a motivation to complete the theory of calculus via regularizations.

The stochastic calculus part starts (Sections \ref{SecStochInt}) with a
 natural extension (Definition \ref{def:forward-integral}) of the 
notion of forward integral in Hilbert (and even Banach) spaces introduced in 
\cite{DiGirolamiRusso09} and with the proof of its equivalence with
 the classical notion of integral when we integrate  a predictable
process w.r.t.
 a $Q$-Wiener process (Theorem \ref{th:integral-forrward=ito-brownian}) 
and w.r.t. a general local martingale
 (Theorem \ref{th:integral-forrward=ito-martingale}). 

In Section
 \ref{SecTensor}, we extend  the notion of Dirichlet process to 
infinite dimension.
Let $\X$ be an $H$-valued stochastic process. According to
the  literature, $\X$ can be naturally 
considered to be an {\it infinite dimensional Dirichlet process} if it is 
the sum of a local martingale  and a zero energy process.
A {\it zero energy process} (with some light sophistications) 
$\X$ is a process such that 
the expectation of the quantity in Definition \ref{def:per-nota}
converges to zero when $\varepsilon$ goes to zero. This happens for instance
in \cite{Denis},  even  though that decomposition also appears
in \cite{maroeck} Chapter VI Theorem 2.5, for processes $\X$ associated with 
an infinite-dimensional Dirichlet form.

Extending  F\"ollmer's notion of Dirichlet
process to infinite dimension, a process $\X$ taking values in 
a Hilbert space $H$, could be called
{\it Dirichlet} if it is the sum of a local martingale $\M$
plus a process $\A$ having a zero scalar quadratic variation.
However that natural notion is not suitable 
for an  efficient stochastic calculus for  infinite dimensional stochastic differential equations.

Similarly to the notion of $\chi$-finite quadratic variation process we 
introduce the notion of $\chi$-Dirichlet process as the sum of a 
local martingale $\M$ and a process $\A$ having a zero
 $\chi$-quadratic variation.

A completely new notion in the present paper 
is the one of  Hilbert valued {\it $\nu$-weak Dirichlet process}
which is
 again related to a Chi-subspace $\nu$ of 
the dual projective tensor  product  $H \hat \otimes_\pi H_1$
where $H_1$ is another Hilbert space, see Definition 
\ref{def:chi-weak-Dirichlet-process}. 
It is of course an extension of the notion of 
real-valued weak Dirichlet process, see Definition \ref{def:Dirweak}.
We illustrate that notion in the simple case when $H_1 = \R$,
$\nu = \nu_0\hat\otimes_\pi \mathbb{R} \equiv \nu_0$
  and $\nu_0$ is a Banach space continuously embedded in $H^*$: a
 process $\X$ is called  $\nu$-weak Dirichlet process if it is the sum
 of a local martingale $\M$ and a process $\A$ such that
 $[\A, N]_{\nu}=0$ for every real continuous local martingale $N$. 
This happens e.g. under the following assumptions.
\begin{itemize}
 \item[(i)] There is a family $(R(\epsilon), \epsilon > 0 )$ of non-negative random variables converging in probability, such
that $ Z(\epsilon) \le R(\epsilon), \epsilon > 0$, where
$$ Z(\epsilon) := \frac{1}{\epsilon} \int_0^T |\A(r+\epsilon) - \A(r)|_{\nu_0^*} |N(r+\epsilon) - N(r)| \ud r.$$

\item[(ii)] 
For all $h\in \nu_0$, $\lim_{\epsilon \to 0^+} \frac{1}{\epsilon}
 \int_0^t \,_{\nu_0}\left\langle \A(r+\epsilon) - \A(r), h\right\rangle_{\nu_0^*} 
(N(r+\epsilon) - N(r)) \ud r =0, \ \forall t \in [0,T].$
\end{itemize}
\begin{Remark} 
We remark that, when condition (i) is fulfilled,
the sequence $(Z(\epsilon))$ is {\it bounded} in
the F-space of random variables, 
 with metric defined by
$ d(X, Y ) = E [ \vert X - Y \vert_B  \wedge 1]$;
this distance governs the convergence in probability.
The notion of bounded subset of an F-space is given in Section II.1 of \cite{dunford}.
\end{Remark}


At the level of pure stochastic calculus, the most important result,
 is Theorem \ref{th:prop6}. 
It generalizes to the Hilbert values framework, 
Proposition 3.10  of \cite{GozziRusso06} which states
that given $f:[0,T] \times \R \rightarrow \R$ of class
$C^{0,1}$ and $X$ is a weak Dirichlet process with
finite quadratic variation then
$Y(t) = f(t, X(t))$ is a real weak Dirichlet process. 
This result is a Fukushima decomposition in the spirit of Dirichlet forms,
 which is the natural extension of Doob-Meyer
decomposition for semimartingales.
It can also be seen as a substitution-tool of It\^o's formula if $f$ is not smooth.
Besides Theorem  \ref{th:prop6},  an interesting general It\^o's formula 
 in the application to mild solutions of
 infinite dimensional stochastic differential equations \color{black} is  Theorem \ref{th:exlmIto}.
The  stochastic calculus theory
developed in Sections \ref{SecStochInt} and \ref{SecTensor},
 makes it possible  to prove that a
mild solution of an equation   
of  type \eqref{eq:SPDEIntro}
is a ${\chi}$-Dirichlet process and a ${\nu}$-weak-Dirichlet 
process; this is done in Corollary \ref{cor:X-barchi-Dirichlet}.
\medskip

As far as stochastic control is concerned, 
the main issue is the verification result stated in Theorem 
\ref{th:verification}. 
As we said, the method we used does not belong to any of the
 described families even if we define the solution of the HJB
 in line with  a strong solution approach.
Since the solution $v:[0,T] \times H \rightarrow \R$ 
 of the HJB equation is only of  class $C^{0,1}$
(with derivative in $C(H, D(A^*))$),
 we cannot apply a It\^o formula of class $C^{1,2}$.
The substitute of such a formula is given in 
 Theorem \ref{th:decompo-sol-HJB}, which is
based on the uniqueness character of the decomposition
of the real weak Dirichlet process $v(t,\X(t))$,
where $\X$ is a solution of the state equation $\X$.
The fact that  $v(t,\X(t))$ is weak Dirichlet 
follows  by    Theorem \ref{th:prop6}
because $\X$ is a $\nu$-weak Dirichlet process
for some suitable space $\nu$.
This is the 
first work that employs this method in infinite dimensions.
 A similar approach was used to deal with the finite dimensional case 
in \cite{GozziRusso06Stoch} but of course in the infinite-dimensional 
case the situation is much more complicated since
the state equation is not a semimartingale and so
  it indeed  requires the introduction of the concept of
 $\nu$-weak Dirichlet process. 

For the reasons listed below, Theorem \ref{th:verification} is more general 
than the results obtained with the  classical strong solutions approach, see e.g. \cite{Gozzi96, GozziRouy96}, and, in a context slightly different than ours, \cite{Cerrai01, Cerrai01-40}.
\begin{itemize}
 \item[(1)] The state equation is more general. 

(a)
 In equation (\ref{eq:state}) the coefficient $\sigma$ depends on 
 time and on the state while in classical strong solutions literature,
 it is  constant and equal
to identity.

(b) In classical strong solutions contributions,
 the coefficient $b$ appearing in 
equation (\ref{eq:state}) is of the particular form  $b(t,\X, a) = b_1(\X) + a$ so it
``separates'' the control and the state parts.

 \item[(2)] We only need  the Hamiltonian  to be well-defined 
and continuous without any particular differentiability, 
 differently to what happens 
in the classical strong solutions literature.
 \item[(3)] We use a milder definition of solution than in
 \cite{Gozzi96, GozziRouy96}; indeed we work with a bigger set 
of approximating functions: in particular (a), our domain $D(\mathscr{L}_0)$,
 does not require the functions and their derivatives to be uniformly
 bounded; (b),  the convergence of the derivatives
 $\partial_x v_n \to \partial_x v$ in (\ref{def:strong-sol})
 is not necessary and it is replaced by
  the weaker condition  (\ref{eq:b}).

\end{itemize}
However, we have to pay a price:
we  assume  that the gradient of the solution of the HJB $\partial_x v$, 
is continuous from $H$ to $D(A^*)$, instead of simply continuous
 from $H$ to $H$.

\smallskip

In comparison to the strong solutions approach,
the $L^2_\mu$ method used  in \cite{GoldysGozzi06}  permits to use \color{black} weaker assumptions
 on the data and enlarges the range of possible applications. 
However, the authors still require  $\sigma$ to be the identity, 
the Hamiltonian to be Lipschitz  and the
 coefficient $b$ to be in a ``separated'' form as in (1)(b) above.
 In the case treated by  \cite{Ahmed03}, the terms containing the control
 in the state variable is more general but the author assumes
 that $A$ and $Q$ have the same eigenvectors.
 So, in both cases, the assumptions on the state equation are
 for several aspects more demanding than ours.

 The backward approach, used e.g. 
in  \cite{ FuHuTe-quadr, FuhrmanTessitore02-ann, FuTe-ell}, 
allows to treat degenerate cases in which the transition semigroup has no 
smoothing properties. Still, in the verification results proved
 in this context, the Hamiltonian has to be differentiable, 
the dependence on the control in the state equation is assumed 
to be linear and its coefficient needs to have a precise relation 
with  $\sigma(t, \X(t))$. All those  hypotheses are stronger than ours.


\smallskip
One important feature of Theorem \ref{th:verification} is  
 that we do not need to assume any hypothesis to ensure 
the integrability of the target.
This is due to the fact that we apply 
the expectation operator only at the last moment.
We stress that, as far as we know, all the available verification theorems for optimal control problems driven by infinite dimensional stochastic evolution equations existing in the literature, obtained with any of the described method, require at least the $C^1$-regularity of the value function w.r.t. the state variable. Our method is not an exception in this sense but, as described below, we can avoid
 a series of assumptions needed with other approaches. It remains true that, in many interesting cases, the value function fails even to be $C^1$ w.r.t. the state variable.

\smallskip

A different approach to infinite dimensional stochastic optimal control
 problems is given by the use of maximum principle. Even recently, a series of
 interesting contributions on the subject appeared, see for example \cite{DuMeng-1, DuMeng-2, Fuhrman13, Fuhrman12,  Oksendal}. As in finite dimension (see e.g. \cite{YongZhou99}), if the Hamiltonian is not convex, the maximum principle approach only gives  necessary conditions but it does not ensure the sufficiency. Moreover, at this stage, maximum principle results for stochastic infinite dimensional problems need a strong regularity on the coefficients and they allow to study mainly  the case of  finite-dimensional noise. 





\bigskip

The scheme of the work is the following.
After some preliminaries in Section \ref{sec:preliminaries}, 
in Section \ref{SecStochInt} we introduce the definition of forward integral
 with values in Banach spaces and we discuss the relation 
with the Da Prato-Zabczyk integral in the Hilbert framework; 
for simplicity we do not treat the Banach case. Section \ref{SecTensor}, devoted to
stochastic calculus, is the core of the paper: we introduce the concepts
of  $\chi$-Dirichlet processes,
 $\nu$-weak-Dirichlet processes and we study their general properties.
 In Section \ref{sec:SPDEs}, the developed theory is applied to
 the case of mild solutions of stochastic PDEs and more in general of
  infinite dimensional stochastic differential equations, while Section \ref{sec:optimal-control}
 contains the application to stochastic optimal control problems 
in Hilbert spaces.

\section{Preliminaries and notations}

\label{sec:preliminaries}


\subsection{Basic functional analysis}
\label{SBN}
Let us consider two separable Banach spaces ${B_1}$ and ${B_2}$. 
We denote by $C(B_1; B_2)$ the set of the
continuous functions from  ${B_1}$ to ${B_2}$. 
This linear space is a topological vector space
if equipped with topology of the uniform convergence on
compact sets.



If  ${B_2}=\mathbb{R}$ we will often simply use the notation $C({B_1})$
 instead of $C({B_1};\mathbb{R})$. Similarly, given a real interval $I$,
typically $I = [0,T]$ or $I = [0,T)$, 
we use the notation $C(I \times {B_1}; {B_2})$ for the set of the
 continuous ${B_2}$-valued
 functions defined on $I\times {B_1}$ while we use the lighter notation 
$C(I \times {B_1})$ when ${B_2}=\mathbb{R}$.
 $C^1(I \times {B_1}$ will denote the space of 
Fr\'echet continuous differentiable functions 
$u: I \times B_1 \rightarrow \R$.
For a function $u: I \times B_1 \rightarrow \R, \ 
(t,x) \mapsto u(t,x) $, we denote by
$(t,x) \mapsto \partial _x u(t,x) $ [resp. $(t,x) \mapsto \partial^2_{xx}
 u(t,x) $]
(if it exists) the first [resp. second] Fr\'echet 
derivative  w.r.t. the variable $x \in {B_1}$).
Eventually a function $
(t,x) \mapsto u(t,x)
 \in C(I \times {B_1})$  [resp. $u  \in C^1(I \times {B_1})$] 
  will be said to belong to  $C^{0,1}(I \times {B_1})$
 [resp. $C^{1,2}(I \times {B_1})$]
 if $\partial _x  u$  exists and it is continuous,
i.e. it belongs to $C(I \times {B_1};{B_1}^*)$ 
 [resp.  $  \partial^2_{xx} u (t,x$] exists for any $(t,x) \in
[0,T] \times B_1$ and it is continuous, i.e.  it belongs
to $C(I \times {B_1}; {\mathcal Bi}({B_1},{B_1}))$,
where  ${\mathcal Bi}({B_1},{B_2}))$ is the linear topological
space of bilinear bounded forms on $B_1 \times B_2$.
\\
By convention all the continuous functions defined
on an interval $I$ are naturally extended 
by continuity to $\R$.

We denote 
by $\mathcal{L}({B_1}; {B_2})$ the space of linear bounded maps from 
${B_1}$ to ${B_2}$
and  by
$\Vert \cdot \Vert_{\mathcal{L}({B_1}; {B_2})}$ the corresponding norm.
We will often indicate in the sequel by a double bar, i.e.
 $\Vert \cdot \Vert$, the norm of an operator.

Often we will consider the case of two separable Hilbert spaces $U$ and $H$.
We denote $|\cdot|$ and $\left\langle \cdot , \cdot \right\rangle$ 
[resp. $|\cdot|_U$ and $\left\langle \cdot , \cdot \right\rangle_U$] the norm 
and the inner product on $H$ [resp. $U$].
\begin{Notation} \label{NotHilbert}
If $H$ is a Hilbert space, in order to argue more transparently, we  often distinguish between $H$ and its dual $H^*$ and
 with every element $h\in H$ we associate $h^*\in H^*$ through Riesz Theorem.
\end{Notation}

If $U= H$, we set $\shl(U) := \shl(U;U)$.
$\shl_2(U;H)$ will be the set of
 {\it Hilbert-Schmidt} operators from $U$ to $H$ and $\shl_1(H)$
 [resp. $\shl_1^+(H)$] will be the space of [resp. non-negative] 
{\it nuclear} operators on $H$.
For  details about the notions of Hilbert-Schmidt and nuclear operator, 
the reader may consult \cite{Ryan02}, Section 2.6 and \cite{DaPratoZabczyk92} Appendix C.
If $ T \in \shl_2(U;H)$ and $T^\ast: H\rightarrow U$ is the adjoint operator, 
then $T T^\ast \in \shl_1(H)$ and the Hilbert-Schmidt
norm of $T$ gives
$\Vert T \Vert^2_{\shl_2(U;H)} = \Vert T T ^\ast \Vert_{\shl_1(H)}.$
We recall that, for a generic element $T\in {\mathcal L}_1(H)$
 and given 
 a basis $\left \{ e_n \right \}$ of $H$, the sum
$
\sum_{n=1}^{\infty} \left \langle T e_n, e_n  \right \rangle
$
is absolutely convergent and independent of the chosen basis 
$\left \{ e_n \right \}$. It is called \emph{trace} of $T$
 and denoted by ${\mathrm Tr}(T)$. 
$\shl_1(H)$ is a Banach space and we denote by $\Vert \cdot \Vert_{\shl_1(H)}$
the corresponding norm.
If $T$ is non-negative then ${\mathrm Tr}(T) = \| T \|_{{\mathcal L}_1 (H)}$ and
 in general we have the inequalities
\begin{equation}
\label{eq:trace-h1}
|{\mathrm Tr}(T)| \leq \| T \|_{{\mathcal L}_1 (H)}, \quad
\sum_{n=1}^{\infty} |\left \langle T e_n, e_n  \right \rangle| \leq \| T
 \|_{{\mathcal L}_1 (H)},
\end{equation}
see Proposition C.1,
\cite{DaPratoZabczyk92}.
As a consequence,  if $T$ is a non-negative operator,  the  relation below
\begin{equation} \label{EHSQ}
\Vert T \Vert_{\shl_2(U;H)}^2 = {\mathrm Tr} (T T^\ast),
\end{equation}
will be very useful in the sequel.

\subsection{Reasonable norms on tensor products}

\label{SRN}

Consider two real Banach spaces $B_1$ and $B_2$. Denote, for $i=1,2$, with 
$|\cdot|_i$ the norm on $B_i$.
$B_1\otimes B_2$ stands for the \emph{algebraic tensor product} i.e.
 the set of the elements of the form $\sum_{i=1}^n x_i\otimes y_i$ where
 $x_i$ and  $y_i$ are respectively elements of $B_1$ and $B_2$. 
On $B_1\otimes B_2$ we identify all the expressions we need in order to ensure 
that the product $\otimes\colon B_1\times B_2 \to B_1\otimes B_2$ is bilinear.

 On $B_1\otimes B_2$ we introduce the projective norm $\pi$ defined, for all $u\in B_1\otimes B_2$, as
\[
\pi(u) := \inf \left \{ \sum_{i=1}^n  |x_i|_{B_1} |y_i|_{B_2} \; : \; 
u = \sum_{i=1}^{n}  x_i \otimes y_i  \right \}.
\]
The \emph{projective tensor product of $B_1$ and $B_2$}, $B_1\hat\otimes_\pi B_2$, is the Banach space obtained as completion of $B_1\otimes B_2$ for the norm $\pi$, see \cite{Ryan02} Section 2.1, or \cite{DiGirolamiRusso09}
 for further details. 

For $u \in B_1\otimes B_2$ of the form $u=\sum_{i=1}^{n} x_i \otimes y_i$
 we define
\[
\varepsilon(u) := \sup \left \{ \left | \sum_{i=1}^n 
 \Phi(x_i) \Psi(y_i) \right | \; : \; \Phi\in B_1^* 
, \; \Psi\in B_2^* 
, \; |\Phi|_{B_1^*} = |\Psi|_{B_2^*} =1 \right \}
\]
and denote  by  $B_1\hat\otimes_\epsilon B_2$ the completion of $B_1\otimes B_2$ for such a norm: it is the   
\emph{injective tensor product of $B_1$ and $B_2$}. We remind that $\epsilon(u)$ does not depend on the representation of $u$ 
and that, for any $u \in B_1\otimes B_2$, $\varepsilon(u) \leq \pi(u)$.

A norm $\alpha$ on $B_1\otimes B_2$ is said to be \emph{reasonable} if for any $u \in B_1\otimes B_2$, 
\begin{equation}
\label{eq:reasonable}
\varepsilon(u) \leq \alpha(u) \leq \pi(u).
\end{equation}

We denote  by $B_1\hat\otimes_\alpha B_2$ the completion of $B_1\otimes B_2$ w.r.t. the norm $\alpha$.
For any reasonable norm $\alpha$ on $B_1\otimes B_2$, for any $x\in B_1$ and $y\in B_2$ one has $\alpha(x\otimes y) = |x|_{B_1} |y|_{B_2}$.
 See \cite{Ryan02} Chapter 6.1 for details.

\begin{Lemma}
\label{lm:che-era-remark}
Let $B_1$ and $B_2$ be two real Banach spaces and $\alpha$ a reasonable norm on $B_1\otimes B_2$. We denote  $B := B_1\hat\otimes_\alpha B_2$.
 Choose $a^* \in B_1^*$ and  $b^* \in B_2^*$. One can associate to $a^*\otimes b^*$ the elements $i(a^* \otimes b^*)$ of $B^*$ acting as follows on a generic element $u = \sum_{i=1}^{n}  x_i \otimes y_i \in B_1\otimes B_2$:
\[
\left\langle i(a^*\otimes b^*), u \right \rangle = \sum_i^n \left\langle a^*, x_i \right\rangle \left\langle b^*, y_i \right\rangle.
\]
Then $i(a^*\otimes b^*)$ extends by continuity to the whole
 $B$ and  
\begin{equation} \label{TensEq}
 \vert i(a^*\otimes b^*) \vert_B =  |a^*|_{B_1^*}|b^*|_{B_2^*}.
\end{equation}
\end{Lemma}
\begin{proof}
We first prove the $\leq$  inequality in \eqref{TensEq}. 
Setting $\tilde a^* = \frac{a^*}{\vert a^* \vert}, \tilde b^* = \frac{b^*}{\vert b^* \vert}$ we write
\begin{equation}
\left\langle i(a^* \otimes b^*), u \right\rangle = \left\langle i(\tilde a^* \otimes \tilde b^*),
 u \right\rangle  \vert a^* \vert \vert b^* \vert
\le \varepsilon(u) |a^*|_{B_1^*}|b^*|_{B_2^*}.
\end{equation}
The latter inequality comes from the definition of injective tensor norm $\varepsilon$, 
considering first $\Phi = a^*, \Psi = B^*$. 
By (\ref{eq:reasonable}) 
$\left\langle i(a^* \otimes b^*), u \right\rangle \leq \alpha(u) |a^*|_{B_1^*}|b^*|_{B_2^*}$ and the 
$\leq$ inequality of  \eqref{TensEq}
 is proved.
\

Concerning the converse inequality, we have
\[
|a^*|_{B_1^*} = \sup_{|\phi|_{B_1}=1} \,_{B_1^*}\left\langle a^*, \phi \right\rangle_{B_1}
\]
and similarly
 for $b^*$. So, chosen $\delta>0$, there exist $\phi_1\in B_1$ and $\phi_2\in B_2$ with $|\phi_1|_{B_1}=|\phi_2|_{B_2}=1$ and
\[
|a^*|_{B_1^*} \leq \delta + \,_{B_1^*}\left\langle a^*, \phi_1 \right\rangle_{B_1}, \qquad
|b^*|_{B_2^*} \leq \delta + \,_{B_2^*}\left\langle b^*, \phi_2 \right\rangle_{B_2}.
\]
We set $u:= \phi_1\otimes \phi_2$. We obtain
\begin{multline}
|i(a^*\otimes b^*)|_{B^*} \geq \frac{\,_{B^*}\left\langle i(a^*\otimes b^*), u \right\rangle_{B}}{|u|_{B}}=
\frac{\,_{B^*}\left\langle i(a^*\otimes b^*), u \right\rangle_{B}}{|\phi_1|_{B_1}|\phi_2|_{B_2}}\\
 = 
\,_{B_1^*}\left\langle a^*, \phi_1 \right\rangle_{B_1}  \,_{B_2^*}\left\langle 
 b^*, \phi_2 \right\rangle_{B_2}
\geq (|a^*|_{B_1^*} - \delta) (|b^*|_{B_2^*} - \delta).
\end{multline}
Since $\delta>0$ is arbitrarily small we finally obtain
\[
|i(a^*\otimes b^*)|_{B^*} \geq |a^*|_{B_1^*} |b^*|_{B_2^*}.
\]
This gives the second inequality and concludes the proof.
\end{proof}

\begin{Notation} \label{Not1}
When $B_1=B_2$ and $x\in B_1$ we denote  by  $x\otimes^2$ the element $x\otimes x \in B_1 \otimes B_1$. 
\end{Notation}

The dual of the projective tensor product  $B_1\hat\otimes_\pi B_2$,
denoted by $(B_1\hat\otimes_\pi B_2)^*$, can be identified isomorphically
with the linear space of bounded bilinear forms on $B_1\times B_2$
 denoted  by   ${\mathcal Bi}(B_1,B_2)$. If $u \in (B_1\hat\otimes_\pi B_2)^*$
and $\psi_u $ is the associated form in ${\mathcal Bi}(B_1,B_2)$,
we have
 $$\vert u \vert_{(B_1\hat\otimes_\pi B_2)^*} = \sup_{\vert a \vert_{B_1} \le 1, 
\vert b \vert_{B_2} \le 1} \vert \psi_u(a,b) \vert. $$
See for this  \cite{Ryan02} Theorem 2.9  Section 2.2, page 22 and also
 the discussion after the proof of the theorem, page 23.


Every  element $u \in H\hat\otimes_\pi H$  is isometrically associated 
with an element $T_u$ in the space of nuclear operators  
${\mathcal L}_1 (H,H)$, defined, for $u$ of the form
  $\sum_{i=1}^{\infty} a_n\otimes b_n$, as follows:
 \[
 T_u(x) := \sum_{i=1}^{\infty} \left\langle x, a_n \right\rangle b_n,
 \]
see for instance \cite{Ryan02} Corollary 4.8 Section 4.1 page 76.

Since $T_u$ is nuclear, in particular 
(see Appendix C of \cite{DaPratoZabczyk92}), there
 exists a sequence of real numbers $(\lambda_n)$ and an orthonormal
 basis $(h_n)$ of $H$ such that $T_u$ can be written as
\begin{equation}
\label{eq:expressionTu}
 T_u(x)= \sum_{n=1}^{+\infty} \lambda_n \left \langle h_n, x \right\rangle h_n, \qquad \text{for all $x\in H$};
\end{equation}
in particular $T_u(h_n) = \lambda_n h_n$ for each $n$.
 Moreover $u$ can be written as
\begin{equation}
\label{eq:expressionu}
u= \sum_{n=1}^{+\infty} \lambda_n h_n \otimes h_n.
\end{equation}

To each element $\psi $ of $(H\hat\otimes_\pi H)^*$ we associate 
a bilinear continuous operator $B_\psi$ 
and a linear continuous operator 
$L_\psi: H \rightarrow H$ 
(see \cite{Ryan02} page 24, the discussion before Proposition 2.11 
Section 2.2) such that
\begin{equation}
\label{eq:expressionLB}
\left\langle L_\psi (x), y \right\rangle = B_\psi(x,y) = 
\psi(x\otimes y)\qquad \text{for all $x, y \in H$}.
\end{equation}

\begin{Proposition} \label{PPTTT}
Let  $u \in  H\hat\otimes_\pi H$  and $\psi \in (H\hat\otimes_\pi H)^*$
with associated maps
 $T_u \in  \mathcal{L}_1(H), L_\psi \in   \mathcal{L}(H)$.
Then
\[
{}_{(H\hat\otimes_\pi H)^*} 
\langle \psi,u \rangle_{H\hat\otimes_\pi H} = {\rm Tr} \left (L_\psi T_u \right ).
\]
\end{Proposition}
\begin{proof}
The claim follows from what we have recalled above. Indeed, using (\ref{eq:expressionu}) and (\ref{eq:expressionLB}) we have
\[
{}_{(H\hat\otimes_\pi H)^*} 
\langle \psi,u \rangle_{H\hat\otimes_\pi H} = \psi(u) = \psi\left ( \sum_{i=1}^{+\infty} \lambda_n h_n \otimes h_n \right ) = \sum_{n=1}^{+\infty} \left \langle L_\psi (\lambda_n h_n), h_n \right\rangle
\]
and the last expression is exactly ${\rm Tr} \left (L_\psi T_u \right )$ when
 we compute it using the basis $h_n$.
\end{proof}

\subsection{Probability and stochastic processes}

In the whole paper, there will be an underlying complete
probability space $\left( \Omega,\mathscr{F},\mathbb{P}\right)$.
Fix $T>0$ and $s\in [0,T)$. 
Let $\left \{ \mathscr{F}^s_t \right \}_{t\geq s}$ be a filtration  
satisfying the usual conditions.
Given a subset $\tilde\Omega \in \mathscr{F}$ we denote  by  $I_{\tilde\Omega}\colon \Omega \to \{0,1\}$ the characteristic function of the set $\tilde\Omega$, 
i.e. $I_{\tilde\Omega} (\omega) = 1$ if and only if $\omega \in \tilde\Omega$.

\smallskip

 Given a real Banach space $B$ we denote by $\mathscr{B}(B)$ the Borel $\sigma$-field on $B$. 

By default we assume that all the processes $\X \colon [s,T]\times \Omega \to B$ are measurable functions  
with respect to the product $\sigma$-algebra 
$\mathscr{B}([s,T]) \otimes \mathscr{F}$ with values in $(B, \mathscr{B}(B))$.
The dependence of a process on the variable $\omega\in\Omega$ is emphasized only if needed by the context. When we say that a process is 
continuous [resp. left continuous, right continuous, c\`adl\`ag, c\`agl\`ad ...] we mean that almost all its paths are continuous
 [resp. left-continuous, right-continuous, c\`adl\`ag, c\`agl\`ad...].

 Let $\mathscr{G}$ be a sub-$\sigma$-field of $\mathscr{B}([s,T])\otimes \mathscr{F}$.
We say that such a process  $\X: ([s,T]\times \Omega,\mathscr{G}) \to B$ is 
measurable with respect to $ \mathscr{G}$ if it is the measurable in the usual sense.
It is said 
{\it strongly (Bochner) measurable} (with respect to $ \mathscr{G}$) if it
is the limit of $\mathscr{G}$-measurable countable-valued functions.  
If $\X$ is  measurable and $\X$ 
is c\`adl\`ag, c\`agl\`ad  or if $B$ is separable then $\X$ is strongly measurable. 
The $\sigma$-field $\mathscr{G}$  will not be mentioned when it is clearly designated.
We denote  by  $\mathscr{P}$ the predictable $\sigma$-field on 
$[s,T] \times \Omega$.
The processes $\X$ measurable on $(\Omega \times [s,T], {\mathcal P})$
are also called {\it predictable} processes. All those processes will be considered
as strongly measurable, with respect to $\shp$.
 Each time we use expressions as
 ``adapted'', 
''predictable'' etc... we will always mean ``with respect to
 the filtration $\left \{ \mathscr{F}^s_t \right \}_{t\geq s}$''.
 
\begin{Notation} \label{Not0}
The \emph{blackboard bold} letters $\X$, $\Y$, $\M$... are used for Banach
 (or Hilbert)-space valued) processes,
 while notations $X$ (or $Y$, $M$...) are reserved for real valued processes.
\end{Notation}

\begin{Notation} \label{Not00}
We always assume the following convention: when needed all the 
Banach space  c\`adl\`ag processes (or functions) 
 indexed by $[s,T]$ are extended setting $\X(t)=\X(s-)$ for $t\leq s$ and $\X(t)=\X(T)$ for $t\geq T$.
\end{Notation}

\section{Stochastic integrals}

\label{SecStochInt}

We adopt the notations introduced in Section \ref{sec:preliminaries}.

\begin{Definition}
\label{def:forward-integral}
Let $\X\colon \Omega \times [s,T] \to \mathcal{L}(U,H)$ and $\Y\colon 
\Omega \times [s,T] \to U$ be two stochastic processes. Assume that $\Y$ 
is continuous and $\X$ is Bochner integrable.

If for almost every $t\in [s,T]$ the following limit  (in the norm of
 the space $H$) exists in probability
\[
\int_s^t \X(r) \ud ^- \Y(r):= \lim_{\epsilon \to 0^+} \int_s^t \X(r) 
\left (\frac{\Y(r+\epsilon) - \Y(r)}{\epsilon} \right ) \ud r
\]
and the process $t\mapsto \int_s^t \X(r) \ud ^- \Y(r)$ admits a continuous
 (in $H$) version, we say that $\X$ is forward integrable with respect to
 $\Y$. That  version of $\int_s^\cdot \X(r) \ud ^- \Y(r)$ is called
\emph{forward integral of $\X$ with respect to $\Y$}.
\end{Definition}

\begin{Remark}\label{R32}
\begin{enumerate}
\item
The definition  above is a natural generalization of that given 
in \cite{DiGirolamiRusso09} Definition 3.4; there the forward integral 
is a real valued process.
\item Previous integral definition extends to the case when
$U$ and $H$ are separable Banach spaces.
\end{enumerate}
\end{Remark}

\subsection{The case of $Q$-Wiener process}
\label{sub:Wiener}

Consider a positive and self-adjoint operator $Q \in \mathcal{L}(U)$. 
Even if not  necessary, we assume  $Q$ to be injective; this
 allows  us to avoid formal complications. However Theorem 
\ref{th:integral-forrward=ito-brownian} below holds without
this restriction.

 Define $U_0:=Q^{1/2} (U)$: $U_0$ is a Hilbert space for the inner product
 $\left \langle x,y \right\rangle_{U_0}:= \left \langle Q^{-1/2} x,
 Q^{-1/2} y \right\rangle_U$ and, clearly $Q^{1/2} \colon U \to U_0$ is an
 isometry, see e.g. \cite{DaPratoZabczyk92} Section 4.3 
for details.
We remind that, given $A \in \mathcal{L}_2(U_0, H)$, we have
$\| A \|^2_{\mathcal{L}_2(U_0, H)}  = Tr \left ( A Q^{1/2}(A Q^{1/2})^*
  \right )$.

Let $\W_Q=\{\W_Q(t):s\leq t\leq T\}$ be an $U$-valued 
$\mathscr{F}^s_t$-$Q$-Wiener process with $\W_Q(s)=0$, $\mathbb{P}$ a.s.
The definition and properties of $Q$-Wiener processes are presented for
 instance in \cite{GawareckiMandrekar10} Chapter 2.1.
If $\Y$ is predictable 
 with some integrability properties,
$\int_s^t \Y(r) \ud \W(r)$ denotes 
 the classical It\^o-type  integral with respect to $\W$,
defined e.g. in \cite{DaPratoZabczyk92}.
 In the sequel  such an integral will be 
 shown to be equal to the forward integral
so that  the forward integral happens to be an extension of the It\^o integral.
 In the next subsection we introduce the It\^o integral
  $\int_s^t \Y(r) \ud \M(r)$ with respect to a local martingale $\M$.  
If $H = \R$, previous integral  
 will also denoted $\int_s^t \langle \Y(r), \ud \M(r) \rangle_U $. 


\begin{Definition}
\label{def:suitable}
We say that a sequence of $\mathscr{F}^s_t$-stopping times
 $\tau_n \colon \Omega\to [0,+\infty]$ is \emph{suitable} if, 
 denoted with $\Omega_n$ the 
set 
\[
\Omega_n := \left \{ \omega \in \Omega \; : \; \tau_n(\omega) >T \right \},
\]
we have
\[
\Omega_n \subseteq \Omega_{n+1} \qquad \text{a.s. for all n}
\]
and
\[
\bigcup_{n\in \mathbb{N}} \Omega_n = \Omega \qquad a.s.
\]
\end{Definition}
In the sequel, we will use the terminology  ``stopping times''
 without mentioning the underlying  filtration $(\mathscr{F}^s_t)$. 

\begin{Theorem}
\label{th:integral-forrward=ito-brownian}
Let $\X\colon [s,T] \times \Omega \to \mathcal{L}_2(U_0,H)$ 
be a 
predictable process satisfying
\begin{equation}
\label{eq:condizione-integr}
\int_s^T \left \| \X(r) \right \|^2_{\mathcal{L}_2(U_0,H)} \ud r
 < +\infty \qquad a.s.
\end{equation}
Then, the forward integral 
\[
\int_s^\cdot \X(r) \ud^- \W_Q(r).
\]
exists and coincides with the classical It\^o integral (defined for example in \cite{DaPratoZabczyk92} Chapter 4)
\[
\int_s^\cdot \X(r) \ud \W_Q(r).
\]
\end{Theorem}
\begin{proof} 
We fix $ t \in [s,T].$
In the proof we follow the arguments related to
 the finite-dimensional case, see Theorem 2
 of \cite{RussoVallois07}.
As a first step we consider $\X$ with
\begin{equation}
\label{eq:firststep}
 \mathbb{E} \left( \int_s^T \left \| \X(r) \right \|^2_{\mathcal{L}_2(U_0,H)} \ud
 r \right) <+\infty.
\end{equation}
This fact ensures that  the hypotheses in the stochastic Fubini Theorem 4.18 
  of \cite{DaPratoZabczyk92} 
are satisfied.
We have
\[
\int_s^t \X(r) \frac{\W_Q(r+\epsilon) - \W_Q(r)}{\epsilon} \ud r =
 \int_s^t \X(r) \frac{1}{\epsilon} \left ( \int_r^{r+\epsilon} \ud \W_Q(\theta) 
\right) \ud r;
\]
applying the stochastic Fubini Theorem, the expression above is equal to
\[
\int_{s}^t \left ( \frac{1}{\epsilon}\int_{\theta-\epsilon}^{\theta} \X(r) \ud r 
\right) \ud \W_Q(\theta) + R_\epsilon(t)
\]
where $R_\epsilon(t)$ is a boundary term that converges to $0$ in probability,
 for any $t \in [s,T]$,
 so that we can ignore it. We can apply now the maximal inequality stated
 in \cite{Stein70}, Theorem 1: there exists a universal constant $C>0$
 such that, for every $f\in L^2([s,t];\mathbb{R})$,
\begin{equation}
\label{eq:Este}
\int_s^t \left ( \sup_{\epsilon \in (0,1]} \left \vert \frac{1}{\epsilon}
\int_{(r-\epsilon)
}^r f(\xi)  \ud \xi \right \vert \right )^2 \ud r \leq
  C \int_s^t f^2(r) \ud r.
\end{equation}
According to the vector valued version of the Lebesgue differentiation 
Theorem (see Theorem II.2.9 in \cite{DiestelUhl77}), the following quantity
\[
\frac{1}{\epsilon}\int_{(r-\epsilon)
}^r \X(\xi) \ud \xi
\]
converges $\ud \mathbb{P} \otimes \ud r$ a.e. to $\X(r)$. Consequently
 (\ref{eq:Este}) and dominated convergence theorem imply
\[
\mathbb{E} \int_s^t \left \| \left ( \frac{1}{\epsilon}
\int_{\theta-\epsilon}^{\theta} \X(r) \ud r \right ) - \X(\theta) 
\right \|_{\mathcal{L}_2(U_0,H)}^2 \ud \theta \xrightarrow{\epsilon\to 0} 0. 
\]

Finally, the convergence 
\begin{equation}
\label{eq:conv-first-step}
J_\epsilon:= \int_{s}^t \left ( \frac{1}{\epsilon}\int_{\theta-\epsilon}^{\theta} \X(r) \ud r \right ) \ud \W_Q(\theta)\\
\xrightarrow[L^2(\Omega,H)]{\epsilon\to 0} J:= \int_{s}^t \X(\theta) 
\ud \W_Q(\theta),
\end{equation}
justifies the claim.

If (\ref{eq:firststep}) is not satisfied we proceed
by localization. Denote again by  $J_\epsilon$ and $J$ the processes 
defined in (\ref{eq:conv-first-step}).  Call $\tau_n$ the  stopping times
 given by
\[
\tau_n:= \inf \left \{ t \in [s,T] \; : \; \int_s^t \left \| \X(r) \right 
\|^2_{\mathcal{L}_2(U_0,H)} \ud r  \geq n \right \}
\]
(and $+\infty$ if the set is void) and call $\Omega_n$ the sets
\[
\Omega_n := \left \{ \omega \in \Omega \; : \; \tau_n(\omega) > T \right \}.
\]
It is easy to see that the stopping times $\tau_n$ are suitable in the
 sense of Definition \ref{def:suitable}.

For each fixed $n$, the process $I_{[s, \tau_n]} \X$ verifies (\ref{eq:firststep}) and 
 from the first step
\[
J_n^\epsilon:= \int_{s}^t \left ( \frac{1}{\epsilon}\int_{\theta-\epsilon}^{\theta}
 I_{[s, \tau_n]}(r) \X(r) \ud r \right ) \ud \W(\theta)
\xrightarrow[L^2(\Omega,H)]{\epsilon\to 0} J_n:=
 \int_{s}^t I_{[s, \tau_n]}(\theta) \X(\theta) \ud \W(\theta).
\]
So
\[
I_{\Omega_n} J_\epsilon = I_{\Omega_n} J_n^\epsilon \xrightarrow
[L^2(\Omega,H)]{\epsilon\to 0} 
I_{\Omega_n} J_n = I_{\Omega_n} J.
\]
Consequently, for all $n$, $I_{\Omega_n} J_\epsilon$ converges to $I_{\Omega_n} J$ 
in probability and finally
$J_\epsilon$ converges to $J$ in probability as well.
 This fact concludes the proof.
\end{proof}

\subsection{The semimartingale case}

Consider now the case when the integrator is a more general local martingale. 
Let  $H$ and $U$ be again two separable Hilbert spaces; we adopt the notations
 introduced in Section \ref{sec:preliminaries}.

An $U$-valued measurable process $\M \colon [s,T] \times \Omega \to U$  is called martingale if, for all $t\in [s,T]$, $\M$ is $\mathscr{F}^s_t$-adapted
with $\mathbb{E} \left [ |\M(t)| \right ] <+\infty$ and $\mathbb{E}\left [ \M(t_2)  |\mathscr{F}^s_{t_1} \right ] = \M(t_1)$ for all $s\leq t_1 \leq t_2 \leq T$.
The concept of (conditional) expectation for $B$-valued processes, for a 
separable Banach space $B$, is recalled for instance in 
\cite{DaPratoZabczyk92} Section 1.3. 
All the considered martingales will be continuous.

We denote  by  $\mathcal{M}^2(s,T; H)$ the linear space of square integrable
 martingales equipped with the norm 
\[
|\M|_{\mathcal{M}^2(s,T; U)} := \left ( \mathbb{E} \sup_{t\in [s,T]} |\M(t)|^2
 \right )^{1/2}.
\]
It is a Banach space as stated in \cite{DaPratoZabczyk92},
 Proposition 3.9.

An $U$-valued measurable process $\M \colon [s,T] \times \Omega \to U$ 
 is called local martingale if there exists a non-decreasing sequence of stopping times
 $\tau_n\colon \Omega \to [s,T]\cup\{+\infty\}$ such that $\M(t\wedge \tau_n)$ for $t\in[s,T]$ is a martingale and $\mathbb{P} \left [ \lim_{n\to\infty} \tau_n = +\infty \right ]=1$. All the considered local martingales are continuous.

Given a continuous local martingale $\M \colon [s,T] \times \Omega \to U$,
the process $|\M|^2$ is a real local sub-martingale, see
Theorem 2.11 in \cite{KrylovRozovskii07}. 
The increasing predictable process, vanishing at zero, appearing in the 
 Doob-Meyer decomposition of $|\M|^2$
will be denoted by
 $[\M]^{\mathbb{R}, cl}(t), t\in [s,T]$.
It  is of course uniquely determined and continuous.

We remind some properties of the It\^o stochastic integral with respect 
to a local martingale $\M$. 
Call $\mathcal{I}_\M(s,T;H)$ the set
 of the processes $\X\colon [s,T]\times \Omega \to \mathcal{L}(U;H)$ that are 
 strongly measurable from $([s,T]\times \Omega, \mathscr{P})$ to $\mathcal{L}(U;H)$ and such that
\[
|\X|_{\mathcal{I}_\M(s,T;H)} := \left (\mathbb{E} \int_s^T \| \X(r) \|_{\mathcal{L}(U;H)}^2 \ud [\M]^{\mathbb{R}, cl}(r) \right )^{1/2} < +\infty.
\]
$\mathcal{I}_\M(s,T;H)$ endowed with the norm $|\cdot|_{\mathcal{I}_\M(s,T;H)}$
 is a Banach space.

The linear map
\[
\left \{
\begin{array}{l}
I\colon \mathcal{I}_\M(s,T;H) \to \mathcal{M}^2(s,T; H)\\
\X \mapsto \int_s^{\cdot} \X(r) \ud \M(r)
\end{array}
\right .
\]
is a contraction, see e.g. \cite{Metivier82} Section 20.4 (above Theorem 20.5). 
As illustrated in  \cite{KrylovRozovskii07} Section 2.2 (above Theorem 2.14),
the stochastic integral w.r.t. $\M$ extends to the integrands
  $\X$ which are 
measurable from $([s,T]\times \Omega, \mathscr{P})$ to $\mathcal{L}(U;H)$
 and such that
\begin{equation} \label{EChainRule}
\int_s^T \| \X(r) \|_{\mathcal{L}(U;H)}^2 \ud [\M]^{\mathbb{R}, cl}(r) < +\infty \qquad a.s.
\end{equation}
We denote by $\mathcal{J}^2(s,T; U,H)$ such a 
family of integrands w.r.t. $\M$.
 Actually, the integral can be even defined for a wider class of integrands, 
see e.g. \cite{MetivierPellaumail80}. For instance, according
to Section 4.7 of \cite{DaPratoZabczyk92},
  let 
\begin{equation} \label{AA}
 \M _t = \int_s^t  \A (r) \ud \W_Q(r), t \in [s,T],
\end{equation}
and $\A$ be an $\mathcal{L}(U,H)$-valued predictable process
such that $\int_s^T {\rm Tr} [\A(r) Q^{1/2} (\A(r) Q^{1/2})^*]  \ud r 
< \infty$ a.s.

If $\X$ is an $H$-valued (or $H^*$-valued using Riesz identification)
predictable process such that 
\begin{equation} \label{EChainRule11}
  \int_s^T \langle \X (r), \A(r) Q^{1/2} (\A  Q^{1/2})^* \X(r) \rangle_H \ud r 
< \infty, \,\, {\rm a.s.},
\end{equation}
then, as argued in Section 4.7 of   \cite{DaPratoZabczyk92},
\begin{equation} \label{AA1}
 N (t) = \int_s^t \langle \X (r), \ud \M(r) \rangle_H, t \in [s,T],
\end{equation}
is well-defined  and it equals 
$  N (t) = \int_s^t \langle \X (r), \A(r) \ud \W_Q(r) \rangle_H$ for 
$t \in [s,T]$.

We recall in the following proposition some significant properties of the stochastic integral with respect to local martingales.
\begin{Proposition} \label{PChainRule}
Let $\M$ be a continuous  $(\mathcal{F}_t^s)$-local martingale,
$\X$ verifying \eqref{EChainRule}. We set $\N(t) = \int_s^t \X(r) d\M(r)$.
\begin{enumerate}
\item[(i)] $\N$ is an  $(\mathcal{F}_t^s)$-local martingale.
\item[(ii)] Let $\K$ be an   $(\mathcal{F}_t^s)$-
predictable process
such that $\K \X$ fulfills \eqref{EChainRule}.
Then the It\^o-type stochastic integral 
$\int_s^t \K d \N$ for $t \in [s,T]$ is well-defined
and it equals $\int_s^t \K \X d \M$.
\item[(iii)] If $\M$ is a $Q$-Wiener process $W_Q$,
then, whenever $\X$ is such that 
\begin{equation} \label{EAAA}
  \int_s^T Tr \left [ \left ( \X(s) Q^{1/2} \right ) \left (\X(s) Q^{1/2} \right)^* \right ] \ud s < +\infty \ {\rm a.s.},
\end{equation}
then $\N(t) = \int_s^t \X(s) d\W_Q(s)$ is a local martingale
and 
\[
[N]^{\R, cl}(t) = \int_s^t \left ( \X(s) Q^{1/2} \right ) \left (\X(s) Q^{1/2} \right)^* \ud s.
\]
\item[(iv)] If  in item (iii), the expectation of the
quantity \eqref{EAAA} is finite, 
then $\N(t) = \int_s^t \X(s) d\W_Q(s)$ is square integrable continuous martingale. 
\item[(v)] If $\M$ is defined as in \eqref{AA} and $\X$ 
fulfills \eqref{EChainRule11}, then $\M$ is a real local martingale.
If moreover,  the expectation of \eqref{EChainRule11} is finite, then 
$N$, defined in (\ref{AA1}), is a square integrable martingale.
\end{enumerate}
\end{Proposition}
\begin{proof}
For (i) see \cite{KrylovRozovskii07} Theorem 2.14 page 14-15. For (ii) see \cite{MetivierPellaumail80}, proof of Proposition 2.2 Section 2.4. (iii) and (iv) are contained in \cite{DaPratoZabczyk92} Theorem 4.12 Section 4.4.
(v) is a consequence of (iii) and (iv) and of the considerations
before the statement of Proposition \ref{PChainRule}.
\end{proof}

\begin{Theorem}
\label{th:integral-forrward=ito-martingale}
Let us consider a continuous local martingale $\M\colon [s,T]\times{\Omega} 
\to U$ and a c\`agl\`ad process predictable $\mathcal{L}(U,H)$-valued process.
Then, the forward integral 
\[
\int_s^\cdot \X(r) \ud^- \M(r),
\]
defined in Definition \ref{def:forward-integral} exists and coincides with the It\^o integral 
\[
\int_s^\cdot \X(r) \ud \M(r).
\]
\end{Theorem}
\begin{Remark} \label{Radated} 
Any c\`agl\`ad 
adapted process is a.s.
bounded and therefore it belongs to ${J}^2(s,T; U,H)$.
\end{Remark}
\begin{proof}[Proof of Theorem \ref{th:integral-forrward=ito-martingale}]

Without loss of generality (replacing if necessary $\X$ with $(\X - \X(s))$) we can suppose that $\X(s) =0$.

The proof follows partially the lines of Theorem 
\ref{th:integral-forrward=ito-brownian}. Similarly we first localize the
 problem using the suitable sequence of stopping times defined by
\[
\tau_n:= \inf \left \{ t \in [s,T] \; : \; \left \| \X(t) \right \|^2_{\mathcal{L}(U,H)} + [\M]^{\mathbb{R}, cl}(t)  \geq n \right \}
\]
(and $+\infty$ if the set is void); the localized process belongs 
to $\mathcal{I}_\M(s,T;H)$ and satisfies the hypotheses of the stochastic
 Fubini theorem in the form given in \cite{Leon90}. Since the integral is
a contraction from $\mathcal{I}_\M(s,T;H)$ to $\mathcal{M}^2(s,T; H)$, it
 only remains to show that
\begin{equation}
\label{eq:EIM}
\mathbb{E} \int_s^t \left \| \left ( \frac{1}{\epsilon}\int_{\theta-\epsilon}^{\theta} \X(r) \ud r \right ) -  \X(\xi) \right \|^2_{\mathcal{L}(U,H) } \ud [\M]^{\mathbb{R}, cl} (\xi) \; \xrightarrow{\epsilon\to 0} 0, 
\end{equation}
when $\X$ belongs to $\mathcal{I}_\M(s,T;H)$. (\ref{eq:EIM}) holds, taking into account  the Lebesgue dominated convergence theorem, because $\X$ is left continuous and
both  $\X$ and $[M]^{\mathbb{R}, cl}$ are bounded.
\end{proof}

An easier but still important statement  concerns the integration 
with respect bounded variation processes.
\begin{Proposition} \label{ItoBV}
 Let us consider a continuous 
bounded variation process $\V \colon [s,T]\times{\Omega} 
\to U$
and  let $\X$ be a c\`agl\`ad measurable  process 
$[s,T] \times \Omega \rightarrow 
{\mathcal L}(U,H)$. 
Then the forward integral 
\[
\int_s^\cdot \X(r) \ud^- \V(r),
\]
defined in Definition \ref{def:forward-integral} exists and 
coincides with the Lebesgue-Bochner integral 
\[
\int_s^\cdot \X(r) \ud \V(r).
\]
\end{Proposition}
\begin{proof} 
The proof is similar to the one of Theorem
\ref{th:integral-forrward=ito-martingale}; one  proceeds
via Fubini theorem.
\end{proof}

\section{$\chi$-quadratic variation and $\chi$-Dirichlet processes}
\label{SecTensor}

\subsection{$\chi$-quadratic variation processes}
\label{sub4.2}

Denote by  $\mathscr{C}([s,T])$ the space of the real continuous processes equipped with the ucp (uniform convergence in probability) topology.
Consider two real Banach spaces $B_1$ and $B_2$
with the same notations as in Section \ref{sec:preliminaries}. 

 Following \cite{DiGirolamiRusso09,
 DiGirolamiRusso11} a {\bf Chi-subspace} (of $(B_1\hat\otimes_\pi B_2)^*$) is
 defined as any Banach subspace $(\chi, |\cdot|_\chi)$
which is continuously embedded into $(B_1\hat\otimes_\pi B_2)^*$:
in other words, there is some constant $C$ such that 
\[
|\cdot|_{(B_1\hat\otimes_\pi B_2)^*} \leq C |\cdot|_\chi.
\]

\begin{Lemma}
\label{lm:era52}
Let us consider a Banach space $\nu_1$ [resp. $\nu_2$] continuously embedded
 in $B_1^*$ [resp. $B_2^*$]. 
Then $\bar\chi := \nu_1 \hat\otimes_\pi \nu_2$ can be continuously embedded in
$(B_1\hat\otimes_\pi B_2)^*$.  In particular there exists a constant $C>0$
 such that,
for all $u\in \bar\chi$,
\begin{equation}
\label{eq:condizione-chi}
|u|_{(B_1\hat\otimes_\pi B_2)^*} \leq C |u|_{\bar\chi},
\end{equation}
after having identified an element of $\bar\chi$ with an element
of $(B_1\hat\otimes_\pi B_2)^*$, as indicated in Lemma \ref{lm:che-era-remark}.
In other words $\bar\chi$ is a Chi-subspace of $(B_1\hat\otimes_\pi B_2)^*$.
\end{Lemma}
\begin{Remark}\label{Rlm:era52}
In particular    $ B_1^* \hat\otimes_\pi B_2^*$ 
is a  Chi-subspace of    $(B_1\hat\otimes_\pi B_2)^*$.
\end{Remark}
\begin{proof}[Proof of Lemma \ref{lm:era52}]
To simplify the notations assume the norm of the injections $\nu_1 
\hookrightarrow B_1^*$
 and $\nu_2 \hookrightarrow B_2^*$ to be less or equal than $1$.
We remind that  $(B_1\hat\otimes_\pi B_2)^*$ is isometrically identified 
with the Banach space 
of the bilinear bounded forms from $B_1\times B_2$ to $\mathbb{R}$,
denoted by ${\mathcal Bi}(B_1,B_2)$.

Consider first an element $u \in \bar\chi$ of the form $u=\sum_{i=1}^n a_i^*\otimes b_i^*$ for some $a_i^*\in \nu_1$ and $b_i^* \in \nu_2$. $u$ can be identified with an element of ${\mathcal Bi}(B_1,B_2)$ acting as 
\[
u\left ( \phi, \psi \right ):= \sum_{i=1}^n \left\langle a_i^*, \phi \right\rangle \left\langle b_i^*, \psi \right\rangle.
\]



We can choose $a_i^*\in \nu_1$ and $b_i^*\in \nu_2$ such that $u=\sum_{i=1}^n a_i^*\otimes b_i^*$ and
\[
|u|_{\bar\chi} = \inf \left \{ \sum_{i=1}^n |x_i|_{\nu_1} |y_i|_{\nu_2} \; : \; u = \sum_{i=1}^{n} x_i \otimes y_i, \qquad x_i\in \nu_1,\; y_i \in \nu_2  \right \}
 > -\epsilon + \sum_{i=1}^n |a_i^*|_{\nu_1} |b_i^*|_{\nu_2}.
\]
Using such an expression for $u$ we have 
\[
\| u \|_{{\mathcal Bi}(B_1,B_2)} = \sup_{|\phi|_{B_1}, |\psi|_{B_2} \leq 1} \left | \sum_{i=1}^n 
\left\langle a_i^*, \phi \right\rangle \left\langle b_i^*, \psi
 \right\rangle \right | 
\leq \sum_{i=1}^n |a_i^*|_{B_1^*} |b_i^*|_{B_2^*}
 \leq \sum_{i=1}^n |a_i^*|_{\nu_1} |b_i^*|_{\nu_2} \leq \epsilon + |u|_{\bar\chi}.
\]
Since $\epsilon$ is arbitrary, we conclude
 that $\| u \|_{{\mathcal Bi}(B_1,B_2)} \leq  |u|_{\bar\chi}$.

Since this proves that the  mapping  that associates to
 $u\in \nu_1 \hat\otimes_\pi \nu_2$ its corresponding element in
 ${\mathcal Bi}(B_1,B_2)$, has norm $1$ on the dense subset 
$\nu_1 \hat\otimes_\pi \nu_2$, then the claim is proved.
\end{proof}

\begin{Remark} \label{RChiProduct}
Even though  the Chi-subspaces of tensor 
product type,  described in Lemma \ref{lm:era52} are natural,
 there are examples of Chi-subspace not of that form, 
see e.g. Section 2.6 in \cite{DiGirolamiRusso09}.
\end{Remark}

Let $\chi$ be a generic Chi-subspace. We introduce the following definition.
\begin{Definition}
\label{def:covariation} 
Given $\X$  [resp. $\Y$]  a  $B_1$-valued [resp. $B_2$-valued] process,
we say that $(\X, \Y)$ admits a $\chi$-covariation if the two
 following conditions are satisfied.
\begin{description}
\item[H1] For any sequence of positive real numbers $\epsilon_{n}\searrow 0$ 
there exists a subsequence $\epsilon_{n_{k}}$ such that 
\begin{equation} \label{FDefC}
\begin{split}
&\sup_{k}\int_{s}^{T} \frac{\left | (J\left(\X({r+\epsilon_{n_{k}}})-\X({r}))
\otimes(\Y({r+\epsilon_{n_{k}}})-\Y({r}))\right)\right |_{\chi^{\ast}} }{\epsilon_{n_{k}}} ds
\;< \infty\; a.s. 
\end{split}
\end{equation}

\item[H2]
If we denote by $[X,\Y]_\chi^{\epsilon}$ the application
\begin{equation}
\label{eq:def-chi-epsilon}
\left \{
\begin{array}{l}
[\X,\Y]_\chi^{\epsilon}:\chi\longrightarrow \mathscr{C}([s,T])\\[5pt]
\displaystyle
\phi \mapsto
\int_{s}^{\cdot} \tensor[_{\chi}]{\left\langle \phi,
\frac{J\left(  \left(\X({r+\epsilon})-\X({r})\right)\otimes \left(\Y({r+\epsilon})-\Y({r})\right)  \right)}{\epsilon} 
\right\rangle}{_{\chi^{\ast}}} dr, 
\end{array}
\right .
\end{equation}
where $ J: B_{1}\hat{\otimes}_{\pi}B_{2} \longrightarrow (B_{1}\hat{\otimes}_{\pi}B_{2})^{\ast\ast}$ is the canonical injection between a space and its bidual, 
the following two properties hold.
\begin{description}
\item{(i)} There exists an application, denoted by $[\X,\Y]_\chi$, defined on $\chi$ with values in $\mathscr{C}([s,T])$, 
satisfying
\begin{equation}
[\X,\Y]_\chi^{\epsilon}(\phi)\xrightarrow[\epsilon\longrightarrow 0_{+}]{ucp} [\X,\Y]_\chi(\phi), 
\end{equation} 
for every $\phi \in \chi\subset
(B_{1}\hat{\otimes}_{\pi}B_{2})^{\ast}$.
\item{(ii)} 
There exists a strongly measurable process
 $\widetilde{[\X,\Y]}_\chi:\Omega\times [s,T]\longrightarrow \chi^{\ast}$, 
such that
\begin{itemize}
\item for almost all 
$\omega \in \Omega$, $\widetilde{[\X,\Y]}_\chi(\omega,\cdot)$ is a (c\`adl\`ag) bounded variation process, 
\item 
$\widetilde{[\X,\Y]}_\chi(\cdot,t)(\phi)=[\X,\Y]_\chi(\phi)(\cdot,t)$ a.s. for all $\phi\in \chi$, $t\in [s,T]$.
\end{itemize}
\end{description}
\end{description}
\end{Definition}

\begin{Remark} \label{RTensBanach}
Since,  $(B_{1}\hat{\otimes}_{\pi}B_{2})^{\ast \ast}$ is continuously embedded
in $\chi^\ast$, then $J(a \otimes b)$ can be considered as 
an element of $\chi^\ast$.
Therefore we have
$$ \vert J(a \otimes b) \vert _{\chi^*} =
\sup_{\phi \in \chi, \vert \phi \vert_\chi \le 1} \langle J(a \otimes b), \phi \rangle 
= \sup_{\phi \in \chi, \vert \phi \vert_\chi \le 1}
 \vert \phi(a \otimes b) \vert. $$ 
We can apply this fact to the expression \eqref{FDefC} considering 
$ a = \X({r+\epsilon_{n_{k}}})-\X({r})$ and  $b = \Y({r+\epsilon_{n_{k}}})-\Y({r}) $.
\end{Remark}

\begin{Remark} \label{RDGR}  An easy consequence of Remark 3.10 and Lemma 3.18
in \cite{DiGirolamiRusso11} is the following. 
We set
\begin{equation} \label{Aepsilon}
A(\varepsilon) := \int_{s}^{T} \frac{\left | (J\left(\X({r+\epsilon})-\X({r}))
\otimes(\Y({r+\epsilon})-\Y({r}))\right)\right |_{\chi^{\ast}} }{\epsilon} dr.
\end{equation}
\begin{enumerate}
\item
If $\lim_{\epsilon \rightarrow 0} A(\epsilon) $ exists in probability
then Condition {\bf H1} of Definition \ref{def:covariation} is verified.
\item If $\lim_{\epsilon \rightarrow 0} A(\epsilon) = 0$ in probability
then $(\X, \Y)$  admits a $\chi$-covariation and 
$\widetilde{[\X,\Y]}$ vanishes. 
\end{enumerate}
\end{Remark}

\bigskip

If $(\X,\Y)$ admits a $\chi$-covariation we
 call $\chi$-covariation of $(\X,\Y)$ the $\chi^{\ast}$-valued process  $\widetilde{[\X,\Y]}_\chi$ defined for every
$\omega\in \Omega $ and $t\in [s,T]$ by $\phi \mapsto
\widetilde{[\X,\Y]}_\chi(\omega,t)(\phi)=[\X,\Y]_\chi(\phi)(\omega,t) $. 
By abuse of notation, $[\X,\Y]_\chi$ will also be often called 
$\chi$-covariation
and it will be confused with $\widetilde{[\X,\Y]}_\chi$.

\begin{Definition}
\label{def:global-cov}
If $\chi = (B_1\hat\otimes_\pi B_2)^*$ the $\chi$-covariation is called 
\emph{global covariation}. In this case we  omit the index $(B_1\hat\otimes_\pi B_2)^*$ using the notations $[\X, \Y]$ and $\widetilde{[\X, \Y]}$.
\end{Definition}
\begin{Remark}
\label{rm:RGlobTens}
The notions of scalar and tensor covariation have been 
defined in Definitions \ref{def:per-nota} and \ref{def:tensorcovariation}.
\begin{enumerate}
\item Suppose that  $\X$ and $\Y$ admits a scalar quadratic variation
and $(\X,\Y)$ has a tensor covariation,
 denoted   by  
$[\X, \Y]^{\otimes}$. 
Then $(\X,\Y)$  admits a global covariation $[\X, \Y]$.
  In particular, recalling that $B_1\hat\otimes_\pi B_2$ is embedded in
 $(B_1\hat\otimes_\pi B_2)^{**}$, we have 
$\widetilde{[\X, \Y]} = [\X, \Y]^{\otimes}$.
The proof is a slight adaptation of the one of
 Proposition 3.14 in \cite{DiGirolamiRusso11}.
In particular condition {\bf H1} holds using Cauchy-Schwarz inequality.
\item If $\X$ admits a scalar zero quadratic variation
then, by definition,
 the tensor covariation of $(\X,\X)$ also vanishes.
Consequently, by item (i) $\X$ also admits a global quadratic 
variation, which is also zero.
\end{enumerate}
\end{Remark}

\begin{Remark}
\label{rm:RChiGlob}
If $(\X,\Y)$ admits a global covariation then it admits a $\chi$-covariation for any Chi-subspace $\chi$. Moreover $[\X, \Y]_{\chi}(\phi)= [\X, \Y](\phi)$ for all $\phi\in \chi$.
\end{Remark}

We say that a process $\X$ admits a $\chi$-quadratic variation if $(\X, \X)$
 admits a $\chi$-covariation. The process
$\widetilde{[\X,\X]}_\chi$, often denoted by $\widetilde{[\X]}_\chi$,
is also called $\chi$-quadratic variation of $\X$.

\begin{Remark}
\label{rm:H1-casoglobal}
For the global covariation case (i.e. for $\chi = (B_1 \hat\otimes_\pi B_2)^*$) the condition \textbf{H1}  reduces to
\[
\sup_{k}\int_{s}^{T} \frac{1}{\epsilon_{n_{k}}}\left | (\X({r+\epsilon_{n_{k}}})-\X({r})) \right |_{B_1}  
\left | \Y({r+\epsilon_{n_{k}}})-\Y({r}))\right |_{B_2} ds \;< \infty\; a.s.
\]
In fact the embedding of $(B_1 \hat\otimes_\pi B_2)$ in its bi-dual is isometric and, for $x\in B_1$ and $y\in B_2$, $|x\otimes y|_{(B_1 \hat\otimes_\pi B_2)} = |x|_1 |y|_2$.
\end{Remark}
The product of a real finite quadratic variation process 
and a zero real quadratic variation process is again a zero
quadratic variation processes. Under some conditions this
can be generalized to the infinite dimensional case.

\begin{Proposition} \label{PFQVZ}
Let $i = 1,2$ and $\nu_i$ be a real Banach space continuously
embedded in the dual $B_i^*$ of a real Banach space $B_i$.
Let consider the Chi-subspace of the type 
$\chi_1 = \nu_1 \hat\otimes_\pi B_2^* $ and   
$\chi_2 = B_1^* \hat\otimes_\pi \nu_2 $,
$\hat \chi_i =  \nu_i  \hat\otimes_\pi \nu_i$, $ i =1,2$. 
 Let $\X$ [resp. $\Y$] be a process with values in $B_1$ [resp. $B_2$].
\begin{enumerate} 
\item Suppose that $\X$ admits a $\hat \chi_1$-quadratic variation 
and $\Y$  a zero scalar quadratic variation.
Then $[\X,\Y]_{\chi_1} = 0$. 
\item Similarly suppose that  $\Y$ admits a $\hat \chi_2$-quadratic variation 
and $\X$  a zero  scalar quadratic variation.
Then $[\X,\Y]_{\chi_2} = 0$. 
\end{enumerate}
\end{Proposition}
\begin{proof}
 We remark that Lemma \ref{lm:era52} imply that
$\chi_i$ and $\hat \chi_i$, $i = 1,2$ are indeed Chi-subspaces.
 By item 2. of Remark \ref{RDGR},
it is enough to show that $A(\varepsilon)$ defined in 
\eqref{Aepsilon} converge to zero, with $\chi = \chi_i, i = 1,2$.
By symmetry it is enough to show item 1.\\
We set $\chi = \chi_1$.
The Banach space $B_i$ is isometrically embedded in its bidual $B_i^{**},
i=1,2,$ so, since $\nu_1\subseteq B_1^*$ with continuous inclusion, we have $B_1 \subseteq B_1^{**} \subset \nu_1^*$ where the inclusion are continuous.

Moreover, since $\chi = \nu_1 \hat \otimes_\pi B_2^* \subseteq 
B_1^* \otimes_\pi B_2^* \subset (B_1 \hat \otimes_\pi B_2)^*$,
 with  continuous inclusions, taking into account Remark \ref{Rlm:era52},
 we have

$$    J(B_1 \hat \otimes_\pi B_2) \subset  
(B_1 \hat \otimes_\pi B_2)^{**} \subset \chi^*.
 $$
Let $a \in B_1$ and  $b \in B_2$.
We have
\begin{multline}
\label{eq:step-da-richiamare}
\vert J(a \otimes b) \vert_{\chi^*} =  \sup_{\vert \varphi \vert_{\nu_1} \le 1, 
\vert \psi \vert_{B_2^*}  \le 1}  
\vert {}_{\chi_1^*} \langle J(a \otimes b), \varphi \otimes \psi 
\rangle_{\chi_1}\vert  \\
= \sup_{\vert \varphi \vert_{\nu_1} \le 1}
\vert {}_{\nu_1} \langle \varphi, a \rangle_{\nu_1^*} \vert
 \sup_{\vert \psi \vert_{B_2^{*}} \le 1}
\vert {}_{B_2^*} \langle \psi, b \rangle_{B_2^{**}} \vert 
= \vert a \vert_{\nu_1^*}\vert b \vert_{B_2^{**}} = \vert a \vert_{\nu_1^*}\vert b \vert_{B_2}.
\end{multline}

Consequently, 
with $a =  \X(r+\varepsilon) - \X(r)$ and $b =\Y(r+\varepsilon) - \Y(r)$
for $r \in [s,T]$, 
 we have 
\begin{multline}
 A(\varepsilon) =
 \int_{s}^{T} \frac{\left | (J\left(\X({r+\epsilon})-\X({r}))
\otimes(\Y({r+\epsilon})-\Y({r}))\right)\right |_{\chi^{\ast}} }{\epsilon} dr
=  \int_s^T \vert \X(r+\varepsilon) - \X(r) \vert_{\nu_1^*}
\vert \Y(r+\varepsilon) - \Y(r) \vert_{B_2} \frac{dr}{\varepsilon}\\ 
\le \left(\int_s^T \vert \X(r+\varepsilon) - \X(r) \vert_{\nu_1^*}^2 
\frac{dr}{\varepsilon}
\int_s^T  \vert \Y(r+\varepsilon) - \Y(r) \vert_{B_2}^2 
\frac{dr}{\varepsilon} \right)^{1/2} \\
= \left( \int_{s}^{T} \frac{\left | (J\left(\X({r+\epsilon})-\X({r}))
\otimes(\X({r+\epsilon})-\X({r}))\right)\right |_{\hat \chi_1^{\ast}} }{\epsilon} dr
\right)^{1/2}  \left( \int_{s}^{T} \frac{\left |\Y({r+\epsilon})-\Y({r}))
\right |^2_{B_2}}{\epsilon} dr
\right)^{1/2}.
\end{multline}
The last equality is obtained using an argument similar to (\ref{eq:step-da-richiamare}).
%
The condition {\bf H1} related to the $\hat \chi_1$-quadratic variation of $\X$
and the zero scalar quadratic variation of  $\Y$, imply
that previous expression converges to zero.
\end{proof}

\medskip

When one of the processes is real the formalism of global 
covariation can be simplified as shown in the following proposition.

Here, and in the sequel, we consider the case of a real separable Hilbert space $H$ instead of a general real Banach space $B$.
According to our conventions, $|\cdot|$ represents both the norm in $H$ and the absolute value 
in $\mathbb{R}$.
\begin{Proposition}
\label{pr:1}
Let $H$ be a real separable Hilbert space. Let be $\X\colon [s,T] \times \Omega \to H$ a Bochner integrable process and $Y\colon [s,T] \times \Omega \to \mathbb{R}$ a real valued process. Suppose 
the following.
\begin{itemize}
\item[(a)] For any $\epsilon$,
$\frac{1}{\epsilon} \int_s^T |\X(r+\epsilon) - \X(r)|
 |Y(r+\epsilon) - Y(r)| \ud r$ is bounded by a r.v. $A(\epsilon)$
such that  $A(\epsilon)$ converges in probability when $\epsilon \to 0$.
\item[(b)] For every $h\in H$ the following limit 
\[
C(t)(h):= \lim_{\epsilon \to 0^+} \frac{1}{\epsilon} \int_s^t \left \langle h, \X(r+ \epsilon) - \X(r) \right\rangle (Y(r+ \epsilon) - Y(r)) \ud r
\]
exists ucp and there exists a continuous process $\tilde C\colon [s,T] \times \Omega \to H$ s.t.
\[
\left\langle \tilde C(t,\omega), h \right\rangle = C(t)(h)(\omega)\qquad \text{for $\mathbb{P}$-a.s. $\omega\in\Omega$},
\]
for all $t\in [s,T]$ and $h\in H$.
\end{itemize}
If we identify $H$ with $(H \hat\otimes_{\pi}\mathbb{R})^*$, then $\X$ 
and $Y$ admit a global covariation and  $\tilde C= \widetilde{[\X,Y]}$.
\end{Proposition}
\begin{proof}
Taking into account the identification of $H$ with $(H\hat\otimes_\pi \mathbb{R})^*$  the result is a consequence of Corollary 3.26 of \cite{DiGirolamiRusso11}. In particular condition \textbf{H1} follows from Remark \ref{rm:H1-casoglobal}.

\end{proof}

\subsection{Relations with the tensor covariation and
the classical tensor covariation}
The notions of tensor covariation recalled in Definition \ref{def:tensorcovariation} concerns general processes. In the specific case when
 $H_1$ and $H_2$ are two separable Hilbert spaces and $\M \colon [s,T]\times\Omega \to H_1$, $\N \colon [s,T]\times\Omega \to H_2$ are two continuous
 local martingales, another (classical) notion of tensor covariation is defined, see for instance in Section 23.1 of \cite{Metivier82}. This will be
 denoted by $[\M, \N ]^{cl}$. Recall that the notion introduced in Definition \ref{def:tensorcovariation} is denoted  by  $[\M, \N ]^{\otimes}$.

\begin{Remark}
\label{rm:propertyP}
We observe the following facts.
\begin{itemize}
\item[(i)] According to Chapter 22 and 23 in \cite{Metivier82}, given an
  $H_1$-valued [resp. $H_2$-valued] continuous local martingale $\M$ 
[resp. $\N$],
 $[\M, \N ]^{cl}$ is an $(H_1\hat\otimes_\pi H_2)$-valued process. Recall that $(H_1\hat\otimes_\pi H_2) \subseteq (H_1\hat\otimes_\pi H_2)^{**}$.
\item[(ii)] Taking into account Lemma \ref{lm:che-era-remark} we know that, given $h \in H_1$ and $k \in H_2$,  $h^* \otimes k^*$ can be considered as an element of $(H_1 \hat\otimes_\pi H_2)^*$. One has
\begin{equation}
\label{eq:propertyP}
[\M, \N ]^{cl}(t)(h^*\otimes k^*) = [ \left\langle \M,h \right\rangle \left\langle \N,k \right\rangle ]^{cl}(t),
\end{equation}
where $h^*$ [resp. $k^*$] is associated with $h$ [resp. $k$] via Riesz theorem. 
This property characterizes $[\M, \N ]^{cl}$, see e.g. \cite{DaPratoZabczyk92}, Section 3.4 after Proposition 3.11.
\item[(iii)] If $H_2=\mathbb{R}$ and $\mathbb{N}=N$ is a real continuous local martingale then, identifying $H_1\hat\otimes_\pi H_2$ with $H_1$, $[\M, \N ]^{cl}$ can be considered as  an  $H_1$-valued process. The characterization (\ref{eq:propertyP}) can be translated into
\begin{equation}
\label{eq:propertyP2}
[\M, \N ]^{cl}(t)(h^*) =
[ \left\langle \M,h \right\rangle, N ]^{cl}(t), \forall h \in H_1.
\end{equation}
By inspection, this allows us to see that
the classical covariation between $\M$ and $N$ can be expressed as
\begin{equation}
\label{eq:defNM}
[\M,N]^{cl}(t):= \M(t) N(t) -  \M(s) N(s) -
\int_s^t N(r) \ud \M(r) - \int_s^t \M(r) \ud N(r).
\end{equation}
\end{itemize}
\end{Remark}


In the sequel $H$ will denote a separable Hilbert space.
\begin{Remark}  
\label{rm:after-definition-chi-D}
The following properties hold.
\begin{enumerate}
\item If $\M$ is a continuous local martingale with values in $H$
 then $\M$ has a scalar quadratic variation, 
see Proposition 1.7 in \cite{DiGirolamiRusso11}.
 \item If $\M$ is a continuous local martingale with values in $H$ 
then $\M$ has a tensor quadratic variation. This fact is proved 
in Proposition 1.6 of \cite{DiGirolamiRusso11}. Using similar arguments 
one can see that if $\M_1$ [resp. $\M_2$] is a continuous local martingale
 with values in $H_1$ [resp. $H_2$] then 
$(\M_1,\M_2)$ admits a tensor covariation.
\end{enumerate}
\end{Remark}

\begin{Lemma}
\label{lm:L411}
Let 
$H$ be a  separable Hilbert space.
 Let $\M$ [resp. $\N$] be a continuous local martingale
with values in $H$.
 Then
$(\M,\N)$ admits a tensor covariation and 
\begin{equation}
\label{eq:eq413}
[\M, \N]^{\otimes}= [\M, \N]^{cl}. 
\end{equation}
In particular $(\M,\N)$ admits a global covariation 
and 
\begin{equation}
\label{eq:eq414}
\widetilde{[\M, \N]}= [\M, \N]^{cl}. 
\end{equation}
\end{Lemma}
\begin{proof} 
Thanks to Remark \ref{rm:after-definition-chi-D}
 $\M$ and $\N$ admit a  scalar quadratic variation
and $(\M, \N)$ a tensor covariation. 
By Remark \ref{rm:RGlobTens} they admit a global covariation. It is enough to show that they are equal as elements of $(H_1\hat\otimes_\pi H_2)^{**}$,
 so one needs to prove that
\begin{equation}
\label{eq:eq414bis}
[\M, \N]^{\otimes}(\phi)= [\M, \N]^{cl}(\phi),
\end{equation}
for every $\phi\in (H_1 \hat\otimes_\pi H_2)^*$. 

Given $h\in H_1$ and $k\in H_2$, we consider (via Lemma \ref{lm:che-era-remark}) $h^* \otimes k^*$ as an element of $(H_1\hat\otimes_\pi H_2)^*$.
According to Lemma \ref{lemmaTensor} below,
$H_1^*\hat\otimes_\pi H_2^*$ is sequentially dense in $(H_1\hat\otimes_\pi H_2)^*$
 in the weak-* topology.
 Therefore, taking into account item (ii) of Remark \ref{rm:propertyP}
 we only need to show that
\begin{equation}
\label{eq:eq415}
[\M, \N]^{\otimes}(h^*\otimes k^*)= [\langle \M,h \rangle \langle \N, 
k \rangle]^{cl},
\end{equation}
for every $h \in H_1, k \in H_2$.
By the usual properties of Bochner integral the left-hand side of (\ref{eq:eq415}) is the limit of 
\begin{multline}
\frac{1}{\epsilon} \int_s^\cdot (M({r+\epsilon}) - M(r)) \otimes (N({r+\epsilon}) - N(r)) (h^*\otimes k^*) \ud r\\
= \frac{1}{\epsilon} \int_s^\cdot \langle (M({r+\epsilon}) - M(r)),h \rangle\langle \otimes (N({r+\epsilon}) - N(r)),k \rangle \ud r.
\end{multline}
Since $\left\langle \M, h \right\rangle$ and $\left\langle \N, k \right\rangle$ are real local martingales, the covariation $[\langle \M,h \rangle \langle \N, k \rangle]$ exists and equals the classical covariation of local martingales because of Proposition 2.4(3) of \cite{RussoVallois93-Oslo}.
\end{proof}
\begin{Lemma}
 \label{lemmaTensor} 
Let $H_1, H_2$ be two separable Hilbert spaces. Then
$H_1^*\hat\otimes_\pi H_2^*$ is sequentially dense in $(H_1\hat\otimes_\pi H_2)^*$
 in the weak-* topology.
\end{Lemma}
\begin{proof}
Let $(e_i)$ and $(f_i)$ be respectively two orthonormal bases of $H_1$ and $H_2$. We denote  by $\mathcal{D}$ the linear span of finite linear combinations of $e_i \otimes f_i$. Let $T\in (H_1\hat\otimes_\pi H_2)^*$, which is a linear continuous functional on $H_1\hat\otimes_\pi H_2$. Using the identification
 of $(H_1\hat\otimes_\pi H_2)^*$ with ${\mathcal Bi}(H_1, H_2)$, for each $n\in \mathbb{N}$, we define the bilinear form
\[
T_n(a,b) := \sum_{i=1}^{n} \left\langle a, e_i \right\rangle_{H_1}  \left\langle b, f_i \right\rangle_{H_2} T(e_i, f_i).
\]
It defines an element of 
$H_1^*\hat\otimes_\pi H_2^* \subset  (H_1\hat\otimes_\pi H_2)^*$. It remains 
to show that 
\[
\tensor[_{(H_1\hat\otimes_\pi H_2)^*}]{\big\langle T_n, l \big\rangle}{_{H_1\hat\otimes_\pi H_2}} \xrightarrow{n\to\infty} \tensor[_{(H_1\hat\otimes_\pi H_2)^*}]{\big\langle T, l \big\rangle}{_{H_1\hat\otimes_\pi H_2}}, \quad \text{for all $l\in H_1\hat\otimes_\pi H_2$}.
\]
We show now the following.
\begin{itemize}
 \item[(i)] $T_n(a,b) \xrightarrow{n\to\infty} T(a,b)$ for all $a\in H_1$, $b\in H_2$.
 \item[(ii)] For a fixed $l\in H_1\hat\otimes_\pi H_2$, the sequence $T_n(l)$ is bounded.
\end{itemize}
Let us prove first $(i)$. Let $a\in H_1$ and $b\in H_2$. We write
\begin{equation}
\label{eq:TE1}
T_n(a,b) = T \left ( \sum_{i=1}^n \left\langle a, e_i\right\rangle_{H_1} e_i,  \quad \sum_{i=1}^n \left\langle b, f_i\right\rangle_{H_2} f_i \right ).
\end{equation}
Since 
\[
\sum_{i=1}^n \left\langle a, e_i\right\rangle e_i \xrightarrow{n\to +\infty} a \quad \text{in $H_1$},
\]
\[
\sum_{i=1}^n \left\langle b, f_i\right\rangle f_i \xrightarrow{n\to +\infty} b \quad\text{in $H_2$},
\]
and $T$ is a bounded bilinear form the point (i) follows.

Let us prove now $(ii)$. Let $\epsilon>0$ fixed and
 $l_0 \in 
\mathcal{D}$ 
such that $|l-l_0|_{H_1\hat\otimes_\pi H_2} \leq \epsilon$. Then 
\begin{equation}
\label{eq:TE2}
|T_n(l)| \leq |T_n(l-l_0)| + |T_n(l_0)| \leq |T_n|_{(H_1\hat\otimes_\pi H_2)^*} |l-l_0|_{H_1\hat\otimes_\pi H_2} + |T_n(l_0)|.
\end{equation}
So (\ref{eq:TE2})  is bounded by 
\begin{eqnarray*}
\sup_{\vert a \vert_{H_1}, \vert b\vert_{H_2} \le 1} \sum_{i=1}^n 
\vert \left\langle a, e_i\right\rangle e_i \vert_{H_1}
\vert \left\langle b, f_i\right\rangle f_i \vert_{H_2}
|T|_{(H_1\hat\otimes_\pi H_2)^*}  \epsilon + \sup_{n} |T_n(l_0)|
\le |T|_{(H_1\hat\otimes_\pi H_2)^*}  \epsilon + \sup_{n} |T_n(l_0)|,
\end{eqnarray*}
recalling that the sequence $(T_n(l_0))$ is bounded, since it is convergent.
Finally (ii) is also proved. \\
At this point (i) implies that
\[
\tensor[_{(H_1\hat\otimes_\pi H_2)^*}]{\big\langle T_n, l \big\rangle}{_{H_1\hat\otimes_\pi H_2}} \xrightarrow{n\to\infty} \tensor[_{(H_1\hat\otimes_\pi H_2)^*}]{\big\langle T, l \big\rangle}{_{H_1\hat\otimes_\pi H_2}}, \quad \text{for all $l\in \mathcal{D}$}.
\] 
Since $\shd$ is dense in $H_1\hat\otimes_\pi H_2$ , the conclusion follows by 
 Banach-Steinhaus theorem, see Theorem 18, Chapter II in
 \cite{DunfordSchwartz58}.
\end{proof}

We recall the following fact that concerns the 
 classical tensor covariation.
\begin{Lemma}
\label{lm:lemmaDPZ-covariation}
Let $\W_Q$ be a $Q$-Wiener process as in Subsection \ref{sub:Wiener}. Let $\Psi\colon ([s,T] \times \Omega, \mathscr{P}) \to \mathcal{L}_2(U_0,H)$ be a strongly measurable process satisfying condition (\ref{eq:condizione-integr}),
with $X = \Psi$.
 Consider the local martingale
\[
\M(t):= \int_s^t \Psi(r) \ud \W_Q(r).
\]
Then
\[
[\M, \M]^{cl}(t) = \int_s^t g(r) \ud r,
\]
where $g(r)$ is the element of $H\hat\otimes_\pi H$ associated with 
the nuclear operator
$G_g(r):=\Big ( \Psi(r) Q ^{1/2} \Big ) \Big ( \Psi(r) Q^{1/2} \Big )^*$.
\end{Lemma}
\begin{proof}
See \cite{DaPratoZabczyk92} Section 4.7.
\end{proof}


\bigskip


\begin{Lemma}
\label{lm:Prop3}
Let $\M\colon [s,T] \times \Omega \to H$ be a continuous local martingale and $\Z$ a  measurable process from $([s,T]\times \Omega, \mathscr{P})$ to $H$ and such that $\int_s^T \| \Z(r) \|^2 \ud [\M]^{\mathbb{R}, cl}(r) < +\infty$ $a.s$. Of course $\Z$  can be
Riesz-identified with an element of $\mathcal{J}^2(s,T; H^*,\mathbb{R})$.
 We define
\begin{equation}
\label{eq:defX-perlemma}
X(t) := \int_s^t \left\langle \Z(r) , \ud \M(r)\right\rangle.
\end{equation}
Then, $X$ is a real continuous local martingale and for every continuous real local martingale $N$, the (classical, one-dimensional) covariation process $[X,N]^{cl}$ is given by
\begin{equation}
\label{eq:[XN]-perlemma}
[X,N]^{cl}(t) = \int_s^t \left\langle \Z(r) , \ud [\M,N]^{cl}(r)\right\rangle;
\end{equation}
in particular the integral in the right-side is well-defined.
\end{Lemma}
\begin{proof}
The fact that $X$ is a local martingale is part of the result of Theorem 2.14 in \cite{KrylovRozovskii07}. For the other claim we can reduce, using a sequence of suitable stopping times as in the proof of Theorem \ref{th:integral-forrward=ito-martingale}, to the case in which $\Z$, $\M$ and $N$ are square integrable martingales.
Taking into account the characterization (\ref{eq:propertyP2}) and the discussion developed in  \cite{Meyer77}, page 456, (\ref{eq:[XN]-perlemma}) follows.
\end{proof}


\begin{Proposition}
\label{pr:thA}
If $\M\colon [s,T] \times \Omega \to H$ and $N\colon [s,T] \times \Omega \to \mathbb{R}$ are continuous local martingales. Then
$\M$ and $N$ admit a global covariation and $\widetilde{[\M,N]}=[\M,N]^{cl}$.
\end{Proposition}
\begin{proof}
We have to check the conditions stated in Proposition \ref{pr:1} 
 for $\tilde C$ equal to the right side of (\ref{eq:defNM}). Concerning (a),
 by Cauchy-Schwarz inequality we have
\[
\frac{1}{\epsilon} \int_s^T |N(r+\epsilon) - N(r)| |\M(r+\epsilon) - \M(r)| \ud r \leq [N,N]^{\epsilon, \mathbb{R}} [\M,\M]^{\epsilon, \mathbb{R}}.
\]
Since both $N$ and $\M$ are local martingales they admit a scalar quadratic variation (as recalled in Remark \ref{rm:after-definition-chi-D}), the result is established. Concerning (b), taking into account (\ref{eq:propertyP2}), we need to prove that for any $h\in H$
\begin{equation}
\label{eq:FIF}
\lim_{\epsilon \to 0} \frac{1}{\epsilon} \int_s^\cdot ( M^h(r+\epsilon) - M^h(r)) (N(r+\epsilon) - N(r)) \ud r = [\left\langle M, h \right\rangle, N]^{cl}
\end{equation}
ucp, where $M^h$ is the real local martingale $\langle \M, h \rangle$. (\ref{eq:FIF}) follows by Proposition 2.4(3) of \cite{RussoVallois93-Oslo}.
\end{proof}

\subsection{$\chi$-Dirichlet and $\nu$-weak Dirichlet processes}

We have now at our disposal all the elements we need to introduce the concept of $\chi$-Dirichlet process and $\nu$-weak Dirichlet process.
\begin{Definition}
\label{def:chi-Dirichlet-process}
Let $\chi\subseteq (H\hat\otimes_{\pi} H)^*$ be a Chi-subspace.
A continuous $H$-valued process $\X \colon ([s,T] \times \Omega,\mathscr{P}) \to H$ is called \emph{$\chi$-Dirichlet process} if there exists a decomposition $\X = \M +\A$ where
\begin{itemize}
 \item[(i)] $\M$ is a continuous local martingale,
 \item[(ii)] $\A$ is a continuous $\chi$-zero quadratic variation process,
 \item[(iii)] $\A(0)=0$.
\end{itemize}
\end{Definition}


\begin{Definition} \label{DNuDir}
Let $H$ and $H_1$ be two separable Hilbert spaces.
Let $\nu\subseteq (H\hat\otimes_{\pi} H_1)^*$ be a Chi-subspace. 
A continuous 
adapted $H$-valued process $\A \colon [s,T] \times \Omega \to H$ is said to be \emph{$\mathscr{F}^s_t$-$\nu$-martingale-orthogonal} if 
\[
[ \A, \N ]_\nu=0,
\]
for any $H_1$-valued continuous local martingale $\N$.
\end{Definition}

As we have done for the expressions ``stopping time'', 
''adapted'', ``predictable''... since we  always use the 
filtration $\mathscr{F}^s_t$, we simply write 
\emph{$\nu$-martingale-orthogonal} instead of $\mathscr{F}^s_t$-$\nu$-martingale-orthogonal.

\begin{Definition}
\label{def:chi-weak-Dirichlet-process}
Let $H$ and $H_1$ be two separable Hilbert spaces.
Let $\nu\subseteq (H\hat\otimes_{\pi} H_1)^*$ be a Chi-subspace. 
A continuous $H$-valued process $\X \colon [s,T] \times \Omega \to H$ is called \emph{$\nu$-weak-Dirichlet process} if it
is adapted and there exists a decomposition $\X = \M +\A$ where
\begin{itemize}
 \item[(i)] $\M$ is  an  $H$-valued continuous local martingale,
 \item[(ii)] $\A$ is an $\nu$-martingale-orthogonal process,
 \item[(iii)] $\A(0)=0$.
\end{itemize}
\end{Definition}
\begin{Remark} \label{R421}
The sum of two $\nu$-martingale-orthogonal processes is
again a $\nu$-martingale-orthogonal process.
\end{Remark}
\begin{Proposition} \label{P421}
\begin{enumerate}
\item Any process admitting a zero scalar quadratic variation
(for instance a bounded variation process) is
a $\nu$-martingale-orthogonal process.
\item Let $\Q$ be an equivalent probability to $\P$.
Any $\nu$-weak Dirichlet process under $\P$ 
is a  $\nu$-weak Dirichlet process under $\Q$. 
\end{enumerate}
\end{Proposition}
\begin{proof}
1. follows from Proposition \ref{PFQVZ} 1. setting
$\nu_1 = H_1^*$. In fact, any local martingale has 
a global quadratic variation because of Remark \ref{rm:RGlobTens} and Remark
 \ref{rm:after-definition-chi-D}. So it has a 
 ${\nu_1 \hat\otimes_\pi\nu_1 }$-quadratic variation by Remark  
\ref{Rlm:era52} and Remark \ref{rm:RChiGlob}.
\\
Concerning 2., by Theorem in Section 30.3, page 208 of \cite{Metivier82},
a local martingale under $\P$ is a local martingale under
$\Q$ plus a bounded variation process. 
The result
follows 
 by Proposition \ref{RBVQV} item 1. and Remark \ref{R421} since the $\nu$-covariation
 remains unchanged under
an equivalent probability measure.
\end{proof}

We recall that the decomposition of a real weak Dirichlet process is
 unique, see Remark 3.5 of \cite{GozziRusso06}. 
For the infinite dimensional case we now establish the uniqueness of 
the decomposition of a $\nu$-weak-Dirichlet process in two cases:
 when $H_1 = H$ and when $H_1=\mathbb{R}$.
\begin{Proposition} \label{PDenseUnique}
Let $\nu\subseteq (H\hat\otimes_{\pi} H)^*$ be a Chi-subspace. Suppose that
 $\nu$ is dense in $(H\hat\otimes_\pi H)^*$. Then any decomposition of
 a $\nu$-weak-Dirichlet process $\X$ is unique.
\end{Proposition}
\begin{proof}
Assume that $\X=\M^1+\A^1 = \M^2 + \A^2$ are two decompositions where $\M^1$ 
and $\M^2$ are continuous local martingales and $\A^1, \A^2$ are $\nu$-martingale-orthogonal processes. If we call $\M:= \M^1- \M^2$ and $\A:= \A^1-\A^2$ we have $0=\M+\A$. 

By Lemma \ref{lm:L411},
 $\M$ has a global quadratic variation. 
In particular it also has a $\nu$-quadratic variation and,
 thanks to the bilinearity  of the $\nu$-covariation,
\[
0 = [\M,0]_\nu = [\M,\M+\A]_\nu = [\M,\M]_\nu +  [\M,\A]_\nu = [\M,\M]_\nu + 0 =[\M,\M]_\nu.
\]
We prove now that $\M$ has also zero global quadratic variation. 
We have denoted  by  $\mathcal{C}([s,T])$ the space of the real continuous processes defined on $[s,T]$. We introduce, for $\epsilon >0$, the operators
\begin{equation}
\label{eq:peremailaFrancesco}
\left \{
\begin{array}{l}
\displaystyle
[\M,\M]^\epsilon \colon  (H\hat\otimes_\pi H)^* \to \mathcal{C}([s,T])\\[5pt]
\displaystyle
([\M,\M]^\epsilon (\phi))(t) := \frac{1}{\epsilon} \int_s^t \tensor[_{(H\hat\otimes_\pi H)}]{\left\langle (\M_{r+\epsilon} - \M_r)\otimes^2, \phi \right\rangle}{_{(H\hat\otimes_\pi H)^*}} \ud r.
\end{array}
\right .
\end{equation}
Observe the following
\begin{itemize}
\item[(a)] $[\M,\M]^\epsilon$ are linear and bounded operators.
\item[(b)] For $\phi\in (H\hat\otimes_\pi H)^*$ the limit $[\M,\M](\phi):= \lim_{\epsilon\to 0} [\M,\M]^\epsilon (\phi)$ exists.
\item[(c)] If $\phi\in \nu$ we have $[\M,\M](\phi)=0$.
\end{itemize}
Thanks to (a) and (b) and Banach-Steinhaus theorem (see Theorem 17, Chapter II in \cite{DunfordSchwartz58}) we know that $[\M,\M]$ is linear and bounded. Thanks to (c) and the fact that the inclusion 
$\nu \subseteq  (H\hat\otimes_\pi H)^*$ is dense, it follows $[\M,\M]=0$. 
By Lemma \ref{lm:L411} $[\M, \M]$ coincides with the classical quadratic variation $[\M, \M]^{cl}$ and it is characterized by
\[
0 = [\M, \M]^{cl} (h^*,k^*) = [\langle \M , h \rangle, \langle \M , k \rangle]^{cl},
\]
by Remark 
\ref{rm:propertyP} (ii).
Since $\M(0)=0$ and therefore $\langle \M , h \rangle(0)=0$ it follows that $\langle \M , h \rangle\equiv 0$ for any $h\in H$. Finally $\M\equiv 0$,
which concludes the proof.
\end{proof}

\begin{Proposition}
\label{pr:421}
Let $H$ be a separable Hilbert space.
Let $\nu\subseteq H^*\equiv (H\hat\otimes_\pi \mathbb{R})^*$ be a Banach space with continuous and dense inclusion. Then any decomposition of a $\nu$-weak-Dirichlet process $\X$ with values in $H$ is unique.
\end{Proposition}

\begin{Remark} \label{L421}
Taking into account the identification of $H^*$ with $(H\hat\otimes_\pi \mathbb{R})^*$ it is possible to consider $\nu$ as a dense subset of $(H\hat\otimes_\pi \mathbb{R})^*$ which is a Chi-subspace.
\end{Remark}
\begin{proof}[Proof of Proposition \ref{pr:421}]
We denote again the inner product on $H$ by 
$\langle \cdot, \cdot \rangle$.
We show that the unique decomposition of the $0$ process is trivial. Assume that $0=\X = \M + \A$. Since $\nu$ is dense in $H^*$ it is possible to choose an orthonormal basis $(e_i^*)$ in $\nu$. We introduce $M^i := \langle \M, e_i \rangle$, they are continuous real local martingales and then, thanks to the properties of $\A$ we have
\[
0= [\X, M^i]_\nu = [\M, M^i]_\nu + [\A, M^i]_\nu = [\M, M^i]_\nu.
\]
By Remark \ref{rm:RChiGlob}, Proposition \ref{pr:thA}
and Remark \ref{rm:propertyP} (iii) we know that
\[
\left\langle [\M, M^i]_\nu , e_i \right\rangle = \left\langle [\M, M^i]^{cl} , e_i \right\rangle
= [M^i, M^i]^{cl};
\]
so $M^i=0$ for all $i$ and then $\M=0$. This concludes the proof.
\end{proof}


\begin{Proposition}
\label{pr:dirichlet-weak-dirichlet}
Let $H$ and $H_1$ be two separable Hilbert spaces.
Let  $\chi = \chi_0 \hat\otimes_\pi \chi_0$
for some $\chi_0$ Banach space continuously embedded in $H^*$.
Define $\nu = \chi_0\hat\otimes_\pi H_1^*$.
 Then an $H$-valued continuous zero $\chi$-quadratic variation process
 $\A$ is a $\nu$-martingale-orthogonal process. 
\end{Proposition}

\begin{proof}
Taking into account Lemma \ref{lm:era52},
$\chi$ is  a Chi-subspace of $(H\hat\otimes_\pi H)^*$ and
 $\nu$ is a Chi-subspace of $(H\hat\otimes_\pi H_1)^*$.
Let $\N$ be a continuous local martingale with values in $H_1$. 
We need to show that $[\A, \N]_\nu=0$. 
We consider the random maps
$T^\epsilon \colon  \nu \times \Omega \to \mathcal{C}([s,T])$ defined by
\[
T^\epsilon(\phi) := [\A, \N]_\nu^{\epsilon} (\phi) = \frac{1}{\epsilon} 
\int_s^\cdot \!\!\! \,_{\nu^*}\!\!\left \langle (\A(r+\epsilon) - \A(r)) 
\otimes (\N(r+\epsilon) - \N(r)), \phi \right\rangle_{\nu} \! \ud r,
\]
for $\phi \in \nu$.

\emph{Step 1}:

Suppose that $\phi = h^*\otimes k^*$ for $h^* \in\chi_0$ and $k \in H_1$. Then 
\begin{multline}
T^\epsilon(\phi)(t) = \frac{1}{\epsilon} \int_s^t \,_{\chi_0^*}\left \langle (\A(r+\epsilon) - \A(r)), h^* \right \rangle_{\chi_0} \, \left\langle (\N(r+\epsilon) - \N(r)), k \right\rangle_{H_1} \ud r\\
\leq 
\left [ \frac{1}{\epsilon} \int_s^t \,_{\chi_0^*}\left \langle (\A(r+\epsilon) - \A(r)), h^* \right \rangle^2_{\chi_0} \ud r 
  \frac{1}{\epsilon} \int_s^t \left\langle (\N(r+\epsilon) - \N(r)), k \right\rangle^2_{H_1} \ud r\right ]^{1/2}\\
=
\left [\frac{1}{\epsilon} \int_s^t \,_{\chi^*}\left \langle (\A(r+\epsilon) - \A(r))\otimes^2, h^*\otimes h^* \right \rangle_{\chi} \ud r \right ]^{\frac{1}{2}} \\
 \times \left [\frac{1}{\epsilon} \int_s^t \left\langle (\N(r+\epsilon) - \N(r)), k \right\rangle^2_{H_1} \ud r\right  ]^{\frac{1}{2}},
\end{multline}
that converges ucp to
\[
\left ( [\A,\A](t)(h^* \otimes h^*) [\N,\N]^{\rm cl}(t)(k^*\otimes k^*)
 \right )^{1/2} = 0,
\]
since the quadratic quadratic variation of a local martingale
is the classical one and taking into account item (ii) of
Remark \ref{rm:propertyP}.

\emph{Step 2}:

We denote  by  $\mathcal{D}$ the linear combinations of elements of the form
 $h^* \otimes k^* $ for $h^* \in \chi_0$ and $k \in H_1$. We remark that
 $\mathcal{D}$ is dense in $\nu$. From the convergence found in \emph{Step 1}, it follows that, for every $\phi\in \mathcal{D}$, ucp we have
\[
T^{\epsilon} (\phi) \xrightarrow{\epsilon \to 0} 0.
\]
\emph{Step 3}:

We consider a generic $\phi\in\nu$. By Lemma \ref{lm:che-era-remark},
for $t \in [s,T]$ it follows
\begin{multline}
\label{eq:dacitare-primo-passo}
|T^{\epsilon}(\phi)(t)| \leq |\phi|_{\nu} \int_s^t \frac{\left | (\A(r+\epsilon) - \A(r)) \otimes (\N(r+\epsilon) - \N(r)) \right |_{\nu^*}} {\epsilon} \ud r \\
= |\phi|_{\nu} \frac{1}{\epsilon}\int_s^t \left | (\N(r+\epsilon) - \N(r)) \right |_{H_1} \left| (\A(r+\epsilon) - \A(r)) \right |_{\chi_0^*} \ud r\\
\leq 
|\phi|_{\nu} \left ( \frac{1}{\epsilon}\int_s^t \left | (\N(r+\epsilon) - \N(r)) \right |^2_{H_1} \ud r \; 
\frac{1}{\epsilon}\int_s^t
\left| (\A(r+\epsilon) - \A(r)) \right |^2_{\chi_0^*} \ud r \right )^{\frac12}\\
=
|\phi|_{\nu} \Bigg ( \frac{1}{\epsilon}\int_s^t \left | (\N(r+\epsilon) - \N(r)) \right |^2_{H_1} \ud r  \times 
\frac{1}{\epsilon}\int_s^t
\left| (\A(r+\epsilon) - \A(r))\otimes^2 \right |_{\chi^*} \ud r \Bigg )^{\frac12}.
\end{multline}
To prove that $[\A, \N]_\nu=0$ we check the corresponding conditions 
\textbf{H1} and \textbf{H2} of the Definition \ref{def:covariation}.
 By Lemma \ref{lm:L411} we know that $\N$ admits a global 
quadratic variation i.e. a $(H_1\otimes H_1)^*$-quadratic variation.
 By condition \textbf{H1} of the Definition \ref{def:covariation}  related to
 $(H_1\otimes H_1)^*$-quadratic variation for the process 
$\N$ and the $\chi$-quadratic variation of $\A$, for any sequence $(\epsilon_n)$
converging to zero, there is a subsequence   $(\epsilon_{n_k})$
such that  the sequence $T^{\epsilon_{n_k}}(\phi)$ is bounded  for any $\phi$ 
in the $\mathcal{C}[s,T]$ metric a.s.
 condition \textbf{H1} of the $\nu$-covariation. By Banach-Steinhaus for $F$-spaces (Theorem 17, Chapter II in \cite{DunfordSchwartz58}) it follows that $T^\epsilon (\phi)\xrightarrow{\epsilon\to 0} 0$ ucp for all $\phi \in \nu$ and so condition \textbf{H2} and the final result follows.
\end{proof}

\begin{Corollary}
\label{Cor:codi}
Assume that the hypotheses of Proposition \ref{pr:dirichlet-weak-dirichlet} 
are satisfied. If $\X$ is a $\chi$-Dirichlet process then we have the following.
\begin{itemize}
 \item[(i)] $\X$ is a $\nu$-weak-Dirichlet process.
 \item[(ii)] $\X$ is a $\chi$-weak Dirichlet process.
\item[(iii)] $\X$ is a $\chi$-finite-quadratic-variation process.
\end{itemize}
\end{Corollary}
\begin{proof}
(i)   follows by Proposition \ref{pr:dirichlet-weak-dirichlet}.
\\
As far as (ii) is concerned, let   $\X = \M + \A$ be 
 a  $\chi$-Dirichlet process decomposition, where
 $\M$ is a local martingale.
 Setting $H_1 = H$,
then 
 $\chi$ is included
in $\nu$, so  Proposition   \ref{pr:dirichlet-weak-dirichlet}
implies that $\A$ is a $\chi$-orthogonal process 
and so (ii) follows. \\
We prove now (iii).  
By Lemma \ref{lm:L411} and Remark 
\ref{rm:RChiGlob} $\M$ admits a $\chi$-quadratic variation.  By
 the bilinearity of the $\chi$-covariation, it is enough to show
that $[\M,\A]_\chi = 0$. This follows from item (ii).
\end{proof}

\begin{Proposition}
\label{pr:DiGirolamiRusso37}
Let $B_1$ and $B_2$ be two Banach spaces and $\chi$ a Chi-subspace of $(B_1\hat\otimes_{\pi} B_2)^*$. Let $\X$ and $\Y$ be two stochastic processes with values respectively in $B_1$ and $B_2$ such that $(\X,\Y)$ admits a
 $\chi$-covariation. Let $G$ be a continuous measurable process $G: [s,T] \times \Omega \to \mathcal{K}$ where $\mathcal{K}$ is a closed separable subspace of $\chi$. 
Then for every $t\in [s,T]$
\begin{equation}   		\label{eq SDFR}  
\int_{s}^{t} \tensor[_\chi]{\langle G(\cdot,r), [\X,\Y]^{\epsilon} (\cdot, r)\rangle}{_{\chi^*}} \ud r 
\xrightarrow[\epsilon \longrightarrow 0]{} 
\int_{s}^{t} \tensor[_\chi]{\langle G(\cdot,r),d\widetilde{[\X,\Y]}(\cdot,r)\rangle}{_{\chi^*}}
\end{equation}
in probability.
\end{Proposition}
\begin{proof}
See \cite{DiGirolamiRusso11Fukushima} Proposition 3.7.
\end{proof}

We state below the most important result related to
 the stochastic calculus part of the paper. It generalizes the 
finite dimensional result contained in \cite{GozziRusso06} Theorem 4.14. The definition of real weak Dirichlet process is recalled in Definition \ref{def:Dirweak}.

\begin{Theorem}
\label{th:prop6}
Let $\nu_0$ be a Banach subspace continuously embedded in $H^*$. Define $\nu:= \nu_0\hat\otimes_\pi\mathbb{R}$ and $\chi:=\nu_0\hat\otimes_\pi\nu_0$. Let
 $F\colon [s,T] \times H \to \mathbb{R}$  be a $C^{0,1}$-function.  Denote with $\partial_x F$ the Frechet derivative of $F$ w.r.t. $x$ and assume that the mapping $(t,x) \mapsto \partial_xF(t,x)$ is continuous from $[s,T]\times H$ to $\nu_0$.  Let $ \X(t) = \M(t) + \A(t)$ for $t\in [s,T]$ be an $\nu$-weak-Dirichlet process with finite $\chi$-quadratic variation. Then $Y(t):= F(t, \X(t))$ is a (real) weak Dirichlet process with local martingale part
\[
R(t) = F(s, \X(s)) + \int_s^t \left\langle \partial_xF(r,\X(r)), \ud \M(r) \right\rangle.
\]
\end{Theorem}
\begin{Remark}
\label{rm:dire-a-francesco}
Indeed the condition $\X$ having a $\chi$-quadratic variation may be replaced with the weaker  condition of $(\X,M)$ having a $\nu_0\hat\otimes_\pi \mathbb{R}$-covariation for any real continuous local martingale $M$.
\end{Remark}
\begin{proof}[Proof of Theorem \ref{th:prop6}]
By definition $\X$ can be written as the sum of a continuous local martingale $\M$ and a $\nu$-martingale-orthogonal process $\A$. 

Let $N$ be a real-valued local martingale. Taking into account 
Lemma \ref{lm:Prop3} and that the covariation of two  real local martingales
defined in \eqref{DefRealIntCov},
coincides with the classical covariation,
 it is enough to prove that 
\[
[F(\cdot, \X(\cdot)), N](t) = \int_s^t \left\langle \partial_xF(r,\X(r)), \ud [\M,N]^{cl}(r) \right\rangle, \qquad \text{for all } t\in [s,T].
\] Let $t \in [s,T]$. We evaluate the $\epsilon$-approximation of
 the covariation, i.e.
\[
\frac{1}{\epsilon} \int_s^t \left ( F(r+\epsilon, \X(r+\epsilon)) - F(r, \X(r)) \right ) 
\left ( N(r+\epsilon) - N(r) \right ) \ud r.
\]
It equals 
\[
I_1(t,\epsilon) + I_2(t,\epsilon),
\]
where
\[
I_1(t,\epsilon) =  \int_s^t \left ( F(r+\epsilon, \X(r+\epsilon)) - F(r+\epsilon, \X(r)) \right ) 
 \frac{ \left ( N(r+\epsilon) - N(r) \right )}{\epsilon} \ud r
\]
and
\[
I_2(t,\epsilon) =  \int_s^t \left ( F(r+\epsilon, \X(r)) - F(r, \X(r)) \right ) 
\frac{\left( N(r+\epsilon) - N(r) \right )}{\epsilon} \ud r.
\]
We prove now that
\begin{equation}
\label{eq:11}
I_1(t,\epsilon) \xrightarrow{\epsilon\to 0} \int_s^t \left\langle \partial_xF(r, \X(r)) , \ud [\M,N]^{cl}(r) \right\rangle
\end{equation}
in probability; in fact
\[
I_1(t,\epsilon) = I_{11}(t,\epsilon) + I_{12}(t,\epsilon)
\]
where 
\[
I_{11}(t,\epsilon) := \int_s^t \frac{1}{\epsilon} \left\langle \partial_xF(r, \X(r)), \X(r+\epsilon) - \X(r) \right\rangle (N(r+\epsilon) - N(r)) \ud r,
\]
\begin{multline}
\nonumber
I_{12}(t,\epsilon) := \int_0^1 \int_s^t \frac{1}{\epsilon} \left\langle \partial_xF(r+\epsilon, a\X(r) + (1-a) \X(r+\epsilon)) - \partial_xF(r, \X(r)), \right .\\ 
\left . \X(r+\epsilon) - \X(r) \right\rangle (N(r+\epsilon) - N(r)) \ud r \ud a.
\end{multline}
Now we apply Proposition \ref{pr:DiGirolamiRusso37} with $B_1= H$, $B_2 = \mathbb{R}$, $\X=\M$, $\Y=N$, $\chi=\nu$ 
so that 
\begin{equation}
\label{eq:EQ11bis}
I_{11}(t,\epsilon) \xrightarrow{\epsilon\to 0} \int_s^t \tensor[_{\nu}]{\left\langle \partial_xF(r, \X(r)) , \ud \widetilde{[\X,N]}(r) \right\rangle}{_{\nu^*}}.
\end{equation}
Recalling that  $\X=\M+\A$, we remark that
 $[\X,N]_\nu$ exists and the $\nu^*$-valued process $\widetilde{[\X,N]}_\nu$
 equals 
\[
\widetilde{[\M,N]}_\nu + \widetilde{[\A,N]}_\nu = \widetilde{[\M,N]}_\nu,
\]
since $\A$ is a $\nu$-martingale orthogonal process. 
Taking into account  the formalism of Proposition \ref{pr:1},
Remark \ref{rm:RChiGlob} and
Proposition \ref{pr:thA} if $\Phi \in H \equiv H^{*}$, we have
\begin{eqnarray*}
  \tensor[_{\nu}] {\left\langle \Phi, \widetilde{[\M,N]_\nu} 
\right\rangle}{_{\nu^*}}
&=&  \tensor[_{\nu_0}]{ \left\langle \Phi, \widetilde{[\M,N]_\nu}
 \right\rangle}{_{\nu_0^*}} =
 \tensor[_{H^*}]{ \left\langle \Phi, \widetilde{[\M,N]}
 \right\rangle}{_{H^{**}}} \\
&=&
 \tensor[_{H^*}]{ \left\langle \Phi, [\M,N]^{cl}
 \right\rangle}{_{H}}.
\end{eqnarray*}
Consequently, it is not difficult to show that
the right-hand side of (\ref{eq:EQ11bis}) gives
\[
\int_s^t {}_{H^*}\left \langle \partial_x F(r, \X(r)) , \ud [\M,N]^{cl}(r)
 \right\rangle_{H} .
\]


For a fixed $\omega\in\Omega$ we consider the function $\partial_xF$ restricted to $[s,T]\times K$ where $K$ is the (compact) subset of $H$ obtained as convex hull of $\{ a \X(r_1) + (1-a) \X(r_2) \; : \; r_1, r_2 \in [s,T] \}$. $\partial_xF$ restricted to $[s,T]\times K$ is uniformly continuous with values in $\nu_0$.
Consequently, for  $\omega$-a.s. 
\begin{equation}
\label{eq:per-f-2}
|I_{12} (t,\epsilon)| \leq \int_s^T \delta\left (\partial_xF_{|[s,T]\times K} ; \epsilon +
\sup_{|r-t|\leq \epsilon} |\X(r) - \X(t)|_{\nu_0^*} \right )
\times  | \X(r+\epsilon) - \X(r)|_{\nu_0^*}
\frac{1}{\epsilon}|N(r+\epsilon) - N(r)| \ud r,
\end{equation}
where, for a uniformly continuous function 
$g:[s,T] \times K \rightarrow \nu_0$, $\delta(g ; \epsilon)$ is 
the modulus of continuity $\delta(g; \epsilon):= \sup_{|x-y|\leq \epsilon}
 |g(x) - g(y)|_{\nu_0}$. 
In previous formula we have identified $H$ with $H^{**}$ so that
$\vert x \vert_H \le \vert x \vert_{\nu_0^*}, \forall x \in H$.
So 
\eqref{eq:per-f-2} is lower than
\begin{multline} \label{E54}
 \delta\left (\partial_xF_{|[s,T]\times K} ; \epsilon +
\sup_{|s-t|\leq \epsilon} |\X(s) - \X(t)|_{\nu_0^*} \right )
\times \left ( \int_s^T \frac{1}{\epsilon}|N(r+\epsilon) - N(r)|^2 \ud r 
\int_s^T \frac{1}{\epsilon}|( \X(r+\epsilon) - \X(r) )|_{\nu_0^*}^2 \ud r \right )^{1/2}\\
=\delta\left (\partial_xF_{|[s,T]\times K} ; \epsilon +
\sup_{|s-t|\leq \epsilon} |\X(s) - \X(t)|_{\nu_0^*} \right )
\times \left ( \int_s^T \frac{1}{\epsilon}|N(r+\epsilon) - N(r)|^2 \ud r 
\int_s^T \frac{1}{\epsilon}|( \X(r+\epsilon) - \X(r) ) \otimes^2|_{\chi^*} \ud r \right )^{1/2},
\end{multline}
when $\varepsilon \rightarrow 0$, where we have used Lemma 
\ref{lm:che-era-remark},  for $\alpha = \pi$
with the usual identification.
The right-hand side of \eqref{E54}, of course converges to zero, 
 since $\X$ [resp. $N$] is a $\chi$-finite quadratic variation process [resp. a real finite quadratic variation process]
and $X$ is also continuous as a $\nu_0^*$-valued process.

To conclude the proof of the proposition we only need to show that $I_2(t,\epsilon)\xrightarrow{\epsilon\to 0}0$. This is relatively simple since 
\[
I_2(t,\epsilon) = \frac{1}{\epsilon} \int_s^t \Gamma(u,\epsilon) \ud 
N(u) + R(t,\epsilon),
\]
where $R(t, \epsilon)$ is a boundary term s.t. $R(t, \epsilon)\xrightarrow{\epsilon\to 0}0$ in probability and
\[
\Gamma(u,\epsilon) = \frac{1}{\epsilon} \int_{(u-\epsilon)_+}^{u} 
\left ( F(r+\epsilon, \X(r)) - F(r, \X(r)) \right ) \ud r.
\]
Since 
\[
\int_s^T (\Gamma(u,\epsilon))^2 \ud [N](u) \to 0,
\]
in probability, Problem 2.27, chapter 3 of \cite{KaratzasShreve88} implies that $I_2(\cdot, \epsilon) \to 0$ ucp. The result finally follows.
\end{proof}


\section{The case of stochastic  infinite dimensional stochastic differential equations }
\label{sec:SPDEs}

This section concerns applications of the stochastic calculus via regularization 
to mild solutions of infinite dimensional stochastic differential equations.

Assume, as in Subsection \ref{sub:Wiener} that $H$ and $U$ are real separable Hilbert spaces, $Q \in \mathcal{L}(U)$, $U_0:=Q^{1/2} (U)$. Assume that  $\W_Q=\{\W_Q(t):s\leq t\leq T\}$  is an $U$-valued $\mathscr{F}^t_s$-$Q$-Wiener process (with $\W_Q(s)=0$, $\mathbb{P}$ a.s.) and denote by  $\mathcal{L}_2(U_0, H)$ the Hilbert space of the Hilbert-Schmidt operators from $U_0$ to $H$. We adopt the conventions of
the mentioned subsection.

We denote   by  $A\colon D(A) \subseteq H \to H$ the generator of the
 $C_0$-semigroup $e^{tA}$ (for $t\geq 0$) on $H$.
The reader may consult   for instance \cite{BDDM07} Part II, Chapter 1 
for basic properties of $C_0$-semigroups.
$A^*$ denotes the adjoint of $A$, $D(A)$ and $D(A^*)$ are
 Banach (even Hilbert) spaces
 when endowed with the graph norm.


Let $b$ be a predictable
 process with values in $H$ and $\sigma$ 
be a
predictable process  with values in
 $\mathcal{L}_2(U_0, H)$ such that
\begin{equation}
\label{eq:hp-per-existence-mild}
\mathbb{P} \left [ \int_0^T |b(t)| + \|\sigma(t)\|^2_{\mathcal{L}_2(U_0, H)} \ud t < +\infty \right ] =1.
\end{equation}
We introduce the process
\begin{equation}
\label{eq:SPDE-mild}
\X(t) = e^{(t-s)A} x  + \int_s^t e^{(t-r)A} b(r) \ud r + \int_s^t e^{(t-r)A} \sigma(r) \ud \W_Q(r).
\end{equation}

\begin{Remark} 
\label{rm:examples-spde}
\begin{enumerate}
\item A mild solution to an equation of  type \eqref{eq:SPDEIntro} is a particular case of \eqref{eq:SPDE-mild}. Indeed, once existence and uniqueness of the solution $\X(\cdot)$ of \eqref{eq:SPDE-mild} is proved, we can simply take $b(r) = b(r, \X(r))$ and $\sigma(r) = \sigma(r, \X(r))$.
\item
Typical examples of SPDEs that can be rewritten in the form \eqref{eq:SPDEIntro} (see e.g. \cite{DaPratoZabczyk96} Part III) are for example stochastic heat (and more general parabolic) equations of the form $dy(s,\xi ) = \left[\Delta_{\xi }y(s,\xi ) +  b(s,y(s,\xi )) \right] \ud s + \sigma(s,y(s,\xi ))  dW_Q(s)(\xi)$ with zero Dirichlet or Neumann boundary conditions (and a suitable initial datum). There are also infinite-dimensional reformulations for heat equations with time-dependent boundary terms and/or boundary noise. Here we use the abstract and general formulation with a generic generator of a $C_0$-semigroup $A$, so we do not use regularizing properties that are typical of the heat semigroup. Other classes of SPDE can be rewritten in the same setting.
\item
A different example arises for instance from stochastic delay differential equations. The following simple example (but more general setting can be studied, see also Chapter 10 in \cite{DaPratoZabczyk96}), is
\begin{equation}
\ud y(s)=\left( a_{0}y(s)+ \int_{-R}^0 a_1(r) y(s+r) \ud r \right) \ud s + \sigma \ud W_{0}(s), 
\end{equation}
(coupled with the initial datum) where $R>0$ is a positive real constant, $a_0$ and $\sigma$ are real numbers, $a_1(r)$ an element of $L^2(-R,0)$ and $W_0$ a real Brownian motion.  It is reformulated for instance 
in the infinite dimension abstract
 setting  in \cite{GozziMarinelliSavin09}.
\end{enumerate}
\end{Remark}

\bigskip

We define 
\begin{equation}
\label{eq:defY}
\Y(t):= \X(t) - \int_s^t b(r) \ud r - \int_s^t \sigma (r) \ud \W_Q(r) -x.
\end{equation}

\begin{Lemma}
\label{lm:conOndrejat}
Let $b$  [resp. $\sigma$] be a 
predictable process with values 
 in $H$  
[resp. with values  in $\mathcal{L}_2(U_0, H)$] such that 
(\ref{eq:hp-per-existence-mild}) is satisfied. Let $\X(t)$ be defined by (\ref{eq:SPDE-mild}) and $\Y$ defined by (\ref{eq:defY}). If $z\in D(A^*)$ we have
\begin{equation}
\left\langle \Y(t) , z \right\rangle = \int_s^t \left\langle \X(r), A^*z \right\rangle \ud r.
\end{equation}
\end{Lemma}
\begin{proof}
See \cite{Ondrejat04} Theorem 12.
\end{proof}

We want now to prove that $\Y$ has zero-$\chi$-quadratic variation for a suitable space $\chi$. We will see that the space
\begin{equation}
\label{eq:def-chi}
\bar\chi:= D(A^*)\hat\otimes_\pi D(A^*).
\end{equation}
does the job. We set $\bar\nu_0:= D(A^*)$ which is clearly continuously
embedded into $H^*$.

By Lemma \ref{lm:era52}, $\bar \chi$ is a Chi-subspace of 
 $(H\hat\otimes_\pi H)^*$.

\begin{Proposition}
\label{pr:Y-zero-chi-quad-var}
The process $\Y$ has zero $\bar\chi$-quadratic variation.
\end{Proposition}
\begin{proof}
Observe that, thank to Lemma 3.18 in \cite{DiGirolamiRusso11} it will be enough
 to show that 
\[
I(\epsilon):= \frac{1}{\epsilon}\int_s^T |(\Y(r+\epsilon) - \Y(r))\otimes^2|_{\bar\chi^*}
 \ud r \xrightarrow{\epsilon \to 0}0, \qquad \text{in probability}.
\]
In fact, identifying $\bar\chi^*$ with $B(\bar\nu_0, \bar\nu_0 ; \mathbb{R})$,
we get
\begin{multline}
I(\epsilon) = \frac{1}{\epsilon} \int_s^T \sup_{|\phi|_{\bar{\nu}_0},\,
 |\psi|_{\bar{\nu}_0} \leq 1} \left | \left\langle (\Y(r+\epsilon) - \Y(r)), \phi \right \rangle \left\langle (\Y(r+\epsilon) - \Y(r)), \psi \right \rangle \right | \ud r \\
\leq 
\frac{1}{\epsilon} \int_s^T \sup_{|\phi|_{\bar{\nu_0}}, \, |\psi|_{\bar{\nu}_0} \leq 1} \left \{ \left | \int_r^{r+\epsilon} \left\langle (\X(\xi), A^*\phi \right \rangle \ud \xi \right |  
\left | \int_r^{r+\epsilon} \left\langle (\X(\xi), A^*\psi \right \rangle \ud \xi \right | \right \} \ud r,
\end{multline}
where we have used Lemma \ref{lm:conOndrejat}. This  is smaller than
\[
\frac{1}{\epsilon} \int_s^T  \left | \int_r^{r+\epsilon} |\X(\xi)| \ud \xi \right |^2 \ud r
\leq \epsilon \sup_{\xi\in [s,T]} |\X(\xi)|^2,
\]
which converges to zero almost surely.
\end{proof}

\begin{Corollary}
\label{cor:X-barchi-Dirichlet}
The process $\X$ is a $\bar \chi$-Dirichlet process.
 Moreover it
 is also a $\bar\chi$ finite quadratic variation process
 and a $\bar\nu_0\hat\otimes_\pi \mathbb{R}$-weak-Dirichlet process.
\end{Corollary}
\begin{proof}
For $ t \in [s,T]$, we have $\X(t) = \M(t) + \A(t)$, where 
\begin{eqnarray*}
\M(t) = x + \int_s^t \sigma (r) \ud \W_Q(r), \\
\A(t) = \V(t) + \Y(t), \\
\V(t) = \int_s^t b(r) dr.
\end{eqnarray*}
$\M$ is a local martingale by Proposition \ref{PChainRule} (i)
and $\V$ is a bounded variation process. 
By  Proposition \ref{PFQVZ} and Remark \ref{rm:RChiGlob},
we get
$$ [\V,\V]_{\bar{\chi}} =  [\V,\Y]_{\bar{\chi}} = [\Y,\V]_{\bar{\chi}} = 0.$$ 
By Proposition \ref{pr:Y-zero-chi-quad-var} and the bilinearity of 
the $\bar{\chi}$-covariation, it yields that $\A$ has a zero
$\bar{\chi}$-quadratic variation and so $\X$ is a $\bar{\chi}$-Dirichlet process.
 The second part of the statement is a consequence of Corollary  
\ref{Cor:codi}.
\end{proof}

In the sequel we will denote by $UC([s,T]\times H; D(A^*))$ the 
$F$-space of the functions \\ $G: [s,T] \times H \rightarrow  D(A^*)$ 
 which are uniformly continuous on each closed ball,
 equipped with the topology of the uniform 
 convergence on closed balls.

The theorem below  generalizes for some aspects the It\^o formula
of \cite{DiGirolamiRusso11}, i.e. their Theorem 5.2, 
to the case when the second derivatives do not necessarily 
belong to the Chi-subspace $\chi$.
\begin{Theorem}
\label{th:exlmIto}
 Let $F\colon [s,T] \times H \to \mathbb{R}$ of class $C^{1,2}$ 
which belongs to $ UC([s,T]\times H; D(A^*))$.

Let $\X$ be an  $H$-valued process process
admitting a $\bar\chi$-finite quadratic variation.
We suppose the following.
\begin{itemize}
\item[(i)] There exists a (c\`adl\`ag) bounded variation process $C\colon [s,T] \times \Omega \to (H\hat\otimes_\pi H)$ such that, for all $t$ in $[s,T]$ and $\phi\in \bar\chi$, 
\[
C(t,\cdot)(\phi) = [\X, \X]_{\bar\chi}(\phi)(t, \cdot) \qquad a.s.
\]
\item[(ii)] For every continuous function $\Gamma\colon [s,T] \times H \to D(A^*)$ the integral
\begin{equation} \label{ForwI}
\int_s^t \left\langle \Gamma(r,\X(r)), \ud ^- \X(r) \right\rangle
\end{equation}
exists.
\end{itemize}
Then
\begin{multline}
\label{eq:Ito}
F(t,\X(t)) = F(s, \X(s)) + \int_s^t \left\langle \partial_rF(r, \X(r)), \ud^- \X(r) \right\rangle \\
+\frac12 \int_s^t  {}_{(H \hat\otimes_{\pi} H)^*}  
\left \langle \partial_{xx}^2 F(r, \X(r)) , \ud C(s) 
\right\rangle_{H \hat\otimes_{\pi} H} 
  + \int_s^t \partial_r F(r, \X(r)) \ud r.
\end{multline}
\end{Theorem}

Before the proof  of the theorem we make some comments.

\begin{Remark} \label{RRR}
A consequence of assumption (i) of Theorem \ref{th:exlmIto} is the existence
of a $\mathbb{P}$-null set $O$ such that, for every $t\in [s,T]$, $\omega \not\in O$,
\[
\widetilde{[\X, \X]}_{\bar\chi}(\phi)(t,\omega) = C(t,\omega)(\phi),
\]
for every $\phi\in D(A^*)\hat\otimes_{\pi} D(A^*)$. In other words the
 $\bar\chi$-quadratic variation of $\X$ coincides with $C$.
\end{Remark}

\begin{Remark}
\label{rm:per-lemmaIto}
The conditions (i) and (ii) of Theorem \ref{th:exlmIto} are verified
 if for instance $\X = \M + \V + \S $, where  $\M$ is a local martingale,
  $\V$ is an $H$-valued bounded variation process, 
and  $\S$ is a process verifying
\[
\left\langle \S, h \right\rangle(t) = \int_s^t \left\langle \Z (r) ,
 A^* h \right\rangle \ud r, \qquad \text{for all } h\in D(A^*),
\]
for some  measurable process $\Z$ with
 $\int_s^T |\Z(r)|^2 \ud r <+\infty$ a.s.

Indeed, by Lemma \ref{lm:L411}, $\M$ admits a global quadratic
quadratic variation which can be identified with $[\M, \M]^{cl}$.
\\
On the other hand $\A = \V + \S$ has a zero $\bar \chi$-quadratic variation,
by Proposition \ref{PFQVZ} and the  bilinearity character of the 
$\bar \chi$-covariation.
$\X$  is therefore a $\bar \chi$-Dirichlet process.
By Corollary \ref{Cor:codi} and again the bilinearity of the
$\bar \chi$-covariation, we obtain that
$\X$ has a finite $\bar \chi$-quadratic variation.
Taking also into account Lemma \ref{lm:L411}
and Remark \ref{rm:RChiGlob}, we get 
 $ \widetilde{[\X,\X]}_{\bar \chi}(\Phi)(\cdot) = 
\langle [\M, \M]^{cl}, \Phi \rangle$ if $\Phi \in \bar \chi$.
Consequently, we can set $C =  [\M, \M]^{cl}$ and condition (i) 
is verified.
\\
  To prove (ii) consider a continuous
 function $\Gamma\colon [s,T] \times H \to D(A^*)$. The integral
of   $(\Gamma(r,\X(r)))$ w.r.t. the semimartingale  $\M + \V$ where
 $\M(t) = x + \int_s^t \sigma (r) \ud W_Q(r) $  exists and  equals
  the classical It\^o integral, by
   Proposition \ref{th:integral-forrward=ito-martingale} and
Proposition \ref{ItoBV}. 
Therefore, 
 we only have to prove that
\[
\int_s^t \left\langle \Gamma(r,\X(r)), \ud^- \S (r) \right\rangle,
 \qquad t\in [s,T]
\]
exists. For every $t\in [s,T]$ the $\epsilon$-approximation of such an
 integral gives, up to a remainder boundary term $C(\epsilon,t)$ which 
converges in probability to zero,
\begin{multline}
\label{eq:era38}
\frac{1}{\epsilon}\int_s^t \left\langle \Gamma(r,\X(r)), \S(r+\epsilon) - 
\S(r) \right\rangle \ud r 
=  \frac{1}{\epsilon}\int_s^t \int_{r}^{r+\epsilon} \left\langle \Z(u), A^*\Gamma(r,\X(r)) \right\rangle \ud u \ud r \\
=  \frac{1}{\epsilon}\int_s^t \int_{u-\epsilon}^{u} \left\langle \Z(u), A^*\Gamma(r,\X(r)) \right\rangle \ud r \ud u \xrightarrow{\epsilon \to 0} \int_s^t \left\langle \Z(u), A^*\Gamma(u,\X(u)) \right\rangle \ud u,
\end{multline}
in probability by classical Lebesgue integration theory. The right-hand side of (\ref{eq:era38}) has obviously a continuous modification so
\eqref{ForwI} exists by definition and
condition (ii) is fulfilled. \\
In particular we have proved that 
\begin{eqnarray*}
\int_s^t \left\langle \Gamma(r,\X(r)), \ud^- \X (r) \right\rangle
&=& \int_s^t \left\langle \Gamma(r,\X(r)), \ud \M (r) \right\rangle
+ \int_s^t \left\langle \Gamma(r,\X(r)), \ud \V (r) \right\rangle  \\ &+&
\int_s^t \left\langle \Z(u), A^*\Gamma(u,\X(u)) \right\rangle \ud u 
\end{eqnarray*}
\end{Remark}

\begin{proof}[Proof of Theorem \ref{th:exlmIto}]
$\;$

\emph{Step 1.}

\smallskip

\noindent
Let $\{e_i^*\}_{i\in \mathbb{N}}$ be an orthonormal basis of $H^*$ made of elements of $D(A^*)\subseteq H^*$. 
This is always possible since $D(A^*)\subseteq H^*$ densely
embedded, via a Gram-Schmidt orthogonalization
procedure.\\
For $N\geq 1$ we denote by $P_N\colon H \to H$ the orthogonal projection on the span of the vectors $\{e_1, .. , e_N\}$. $P_\infty\colon H \to H$ will simply denote the identity.

\smallskip

\noindent
Let us for a moment omit the time  dependence on $F$,
which is supposed to be of class $C^2$ from $H$ to $\R$.
We define $F_N\colon H \to \mathbb{R}$ as $F_N(x):= F(P_N (x))$. We have
\begin{equation} \label{57bis}
\partial_xF_N(x) = P_N \partial_xF(P_N X(x))
\end{equation}
and 
\[
\partial_{xx}^2F_N(x) = (P_N\otimes P_N) \partial_{xx}^2 F (P_N X(x)),
\]
where the latter  equality has to be understood as
\begin{multline}
\label{eq:3}
\,_{(H\hat\otimes_{\pi} H)^*}\! \left\langle \partial_{xx}^2 F_N (x) , h_1\otimes h_2 \right\rangle_{(H\hat\otimes_{\pi} H)} = 
\,_{(H\hat\otimes_{\pi} H)^*}\!\left\langle \partial_{xx}^2 F (P_N(x)) , (P_N(h_1))\otimes (P_N(h_2)) \right\rangle_{(H\hat\otimes_{\pi} H)},
\end{multline}
for all $h_1, h_2 \in H$. $\partial_{xx}^2 F_N (x)$ is an element of $(H\hat\otimes_{\pi} H)^*$ but it belongs to $(D(A^*)\hat\otimes_{\pi} D(A^*))$ as well; indeed it can be written as
\[
\sum_{i,j=1}^N \,_{(H\hat\otimes_{\pi} H)^*}\!\left\langle \partial_{xx}^2F(P_N (x)), e_i\otimes e_j \right\rangle_{(H\hat\otimes_{\pi} H)} \, \left ( e_i^* \otimes  e_j^* \right )
\]
and $e_i^*\otimes e_j^*$ are in fact elements of 
$D(A^*)\hat\otimes_{\pi} D(A^*)$. 

\medskip
We come back now again to the time dependence notation $F(t,x)$.
We can apply the It\^o formula proved in \cite{DiGirolamiRusso11}, 
Theorem 5.2, and with the help of Assumption (i), we find
\begin{multline}
\label{eq:Ito-N}
F_N(t,\X(t)) = F_N(s, \X(s)) + \int_s^t \left\langle \partial_xF_N(r, \X(r)), \ud^- \X(r) \right\rangle \\
+\frac12 \int_s^t \left\langle \partial_{xx}^2 F_N(r, \X(r)) , \ud C(s) \right\rangle + \int_s^t \partial_r F_N(r, \X(r)) \ud r.
\end{multline}

\bigskip

\emph{Step 2.}
We consider, for fixed $\epsilon >0$, the map
\[
\left \{
\begin{array}{l}
T_\epsilon \colon UC([s,T]\times H; D(A^*)) \to L^0(\Omega)\\[5pt]
T_\epsilon \colon  \displaystyle G \mapsto \int_s^t \left\langle G(r, \X(r)), \frac{\X(r+\epsilon) - \X(r)}{\epsilon} \right\rangle\ud r,
\end{array}
\right .
\]
where the set    $L^0(\Omega)$ of all real random variables
 is equipped with the topology of the convergence in probability. 
Assumption (ii) implies that $\lim_{\epsilon\to 0} T_\epsilon G$ exists for every $G$.
By Banach-Steinhaus for $F$-spaces (see Theorem 17, Chapter II in \cite{DunfordSchwartz58}) it follows that the map
\[
\left \{
\begin{array}{l}
UC([s,T]\times H; D(A^*)) \to L^0(\Omega)\\[5pt]
\displaystyle
G \mapsto \int_s^t \left\langle G(r, \X(r)), \ud ^-\X(r) \right\rangle
\end{array}
\right .
\]
is linear and continuous.

\bigskip

\emph{Step 3.}

\smallskip

\noindent
If $K\subseteq H$ is a compact set then the set 
\[
P(K) := \left \{ P_N(y) \; : \; y\in K, \; N\in  \mathbb{N}\cup {+\infty} \right \}
\]
is compact as well. Indeed, consider $\left \{ P_{N_l} (y_l) \right \}_{l \ge 1}$
be  a sequence in $P(K)$. We look for a subsequence convergence to an
element of $ P(K)$.

 Since $K$ is compact we can assume,  without loss of generality, that $y_l$ converges, for ${l\to + \infty}$, to some $y\in K$. If $\{ N_l \}$ assumes only a finite number of values then (passing if necessary to a subsequence) $N_l \equiv \bar N$ for some $\bar N \in \mathbb{N}\cup {+\infty}$ and then $P_{N_l} (y_l) \xrightarrow{l\to+\infty} P_{\bar N} (y)$. Otherwise we can assume (passing if necessary to a subsequence) that $N_l \xrightarrow{l\to+\infty} +\infty$ and then 
it is not difficult to prove that $P_{N_l} (y_l) \xrightarrow{l\to+\infty} y$, which belongs to $P(K)$ since
$y = P_\infty y$.

In particular, being $\partial_xF$ continuous,
\[
\mathcal{D}:= \left \{ \partial_xF (P_N (x)) \; : \; x\in K, \; 
N\in  \mathbb{N}\cup \{+\infty\} \right \}
\]
is compact in $D(A^*)$. Since the sequence of maps $\{ P_N \}$ is
 uniformly continuous, it follows that
\begin{equation}
\label{eq:unif-conv-1}
\sup_{x \in \mathcal{D}} \vert (P_N - I) (x) \vert  \xrightarrow{N\to\infty} 0. 
\end{equation}

\bigskip

\emph{Step 4.}

\smallskip

We show now that 
\begin{equation}
\label{eq:first-term}
\lim_{N\to\infty} \int_s^t \left\langle \partial_xF_N(r, \X(r)), \ud^- \X(r) \right\rangle = \int_s^t \left\langle \partial_xF(r, \X(r)), \ud^- \X(r) \right\rangle.
\end{equation}
holds in probability for every $t\in [s,T]$.
Let $K$ be a compact subset of $H$. In fact 
\[
\sup_{t \in [0,T], x\in K} |\partial_xF(t, P_N (x)) - \partial_xF(t,x)| 
 \xrightarrow{N\to\infty}  0,
\]
since $\partial_xF$ is continuous. On the other hand
\[
\sup_{t \in [0,T], x\in K} |(P_N - I)(\partial_xF(t, P_N x))|  
\xrightarrow{N\to\infty}  0,
\]
because of (\ref{eq:unif-conv-1}). Consequently, by \eqref{57bis}, 
\begin{equation} \label{eq:second}
\partial_x F_N  \to \partial_x F, 
\end{equation}
uniformly on each compact, with values in $H$.
This yields that 
$\omega$-a.s. 
\[
\partial_x F_N (r, \X (r)) \to \partial_x F (r, \X (r)),
\]
uniformly on each compact. By Step 2, then (\ref{eq:first-term}) follows.

%

\bigskip

\emph{Step 5.}

\smallskip

\noindent
Finally, we prove that
\begin{equation}
\label{eq:second-term}
\lim_{N\to\infty} \frac12 \int_s^t \left\langle \partial_{xx}^2 F_N(r, \X(r)) , \ud C(s) \right\rangle = \frac12 \int_s^t \left\langle \partial_{xx}^2 F(r, \X(r)) , \ud C(s) \right\rangle.
\end{equation}

For a fixed $\omega \in\Omega$ we define $K(\omega)$ the compact set as
\[
K(\omega):= \left \{ \X(t)(\omega) \; : \; t\in[s,T] \right \}.
\]
We write
\begin{multline}
\left | \int_s^t \left\langle \partial_{xx}^2F_N(r,\X(x)) - \partial_{xx}^2 F (r,\X(x)) , \ud C(r) \right\rangle \right | (\omega)\\
\leq
\sup_{\substack { y\in K(w)\\ t\in [s,T]}} \left \| \partial_{xx}^2F_N(t,y) -
 \partial_{xx}^2F(t,y)  \right \|_{(H\hat\otimes_{\pi} H)^*} \int_s^T \ud |C(r)| (\omega).
\end{multline}

Using arguments similar to those used in proving (\ref{eq:second})
 one can see that 
\[
\partial^2_{xx} F_N \xrightarrow{N\to\infty} \partial^2_{xx} F
\]
uniformly on each compact. Consequently 
\[
\sup_{r\in [s,T]} | (\partial^2_{xx} F_N -
 \partial^2_{xx} F) (r, \X(r)) |_{(H \hat\otimes_\pi H)^*} \xrightarrow{N\to\infty} 0.
\]
Since $C$ has bounded variation, finally (\ref{eq:second-term}) holds.


\bigskip

\emph{Step 6.}

\smallskip

Since $F_N$ [resp. $\partial_r F_N$] converges uniformly on each compact to 
$F$ [resp. $\partial_r F$], when $N\to\infty$, then 
$$ \int_s^t \partial_r F_N(r, \X(r)) \ud r   \xrightarrow{N\to\infty}  
\int_s^t \partial_r F (r, \X(r)) \ud r.$$
Taking the limit when $N\to\infty$ in
 (\ref{eq:Ito-N}), finally provides (\ref{eq:Ito}).
\end{proof}
Next result can be considered a It\^o formula for {\it mild
type processes}, essentially coming out from 
mild solutions of infinite dimensional stochastic differential equations. An interesting  contribution
in this direction, but in a different spirit appears in \cite{RoecknerMild}.
\begin{Corollary}
\label{lm:Ito-pre}
Assume that $b$ is a 
predictable process with values in $H$ and $\sigma$ 
is
a predictable process  with values in $\mathcal{L}_2(U_0, H)$ satisfying
 (\ref{eq:hp-per-existence-mild}). Define $\X$ as in (\ref{eq:SPDE-mild}). Let $x$ be an element of $H$. Assume that $f\in C^{1,2}([0,T] \times H)$
 with $\partial_xf \in UC([0,T] \times H, D(A^*))$. Then,  $\mathbb{P}-a.s.$,
\begin{multline}
\label{eq:first-Dinkin-pre}
f(t,\X(t)) =  f(s,x)  + \int_s^t \partial_sf(r,\X(r)) \ud r  +  \int_s^t \left\langle A^* \partial_x f(r,\X(r)), \X(r) \right\rangle \ud r
 +  \int_s^t \left\langle \partial_x f(r,\X(r)), b(r) \right\rangle \ud r\\
+ \frac{1}{2}  \int_s^t \text{Tr} \left [\left ( \sigma (r) {Q}^{1/2} \right )
\left ( \sigma (r) Q^{1/2}  \right)^* \partial_{xx}^2 f(r,\X(r)) \right ] \ud r
+ \int_s^t \left\langle \partial_x f(r,\X(r)), \sigma(r) \ud \W_Q(r) \right\rangle.
\end{multline}
\end{Corollary}
\begin{Remark} \label{Rott}
We remark that in \eqref{eq:first-Dinkin-pre},
the partial derivative $\partial^2_{xx} f(r, x)$
for any $r \in [s,T]$ and $x \in H$ stands
in fact for its associated linear bounded operator
in the sense of \eqref{eq:expressionLB}.
From now on we will  make this 
natural identification.
\end{Remark}

\begin{proof}

It is a consequence of Theorem \ref{th:exlmIto} taking into account
 Remark \ref{rm:per-lemmaIto}:
we have  $\M(t) = x + \int_s^t \sigma (r) \ud \W_Q(r), t \in [0,T]$,
$\V(t)  = \int_s^t b(r) dr$, $\S = \Y$ with  
$\Z(r) = \X(r)$. 
According to that Remark, in   Theorem  \ref{th:exlmIto}
 we set $C = [\M,\M]^{cl}$.
We also use the chain rule
for It\^o's integrals in Hilbert spaces, see the considerations
before  Proposition
\ref{PChainRule},  together with Lemma \ref{lm:lemmaDPZ-covariation}.
The fourth integral in the right-hand side of 
\eqref{eq:first-Dinkin-pre} appears from the second integral
in \eqref{eq:Ito} using Proposition \ref{PPTTT}
and again  Lemma \ref{lm:lemmaDPZ-covariation}.
\end{proof}

\section{The optimal control problem}
\label{sec:optimal-control}

In this section we illustrate the utility of the tools of stochastic
 calculus via regularization in the study of optimal control problems
 driven  by stochastic PDEs or more in general by infinite dimensional stochastic differential equations. We will prove a decomposition result for the strong 
solutions of the Hamilton-Jacobi-Bellman equation related to the optimal
 control problem and we use that decomposition to derive a verification theorem.

$U$, $U_0$, $\W_Q=\{\W_Q(t):t\leq s\leq T\}$, $\mathcal{L}_2(U_0, H)$
 and $A\colon D(A) \subseteq H \to H$ were
 defined as in Section \ref{sec:SPDEs}.

\subsection{The setting of the problem}

 We consider $\Lambda$ a Polish space (i.e. a complete and separable metric space) that will be our control space. In other words we will try to minimize our functional on a class of $\Lambda$-valued processes. 
We formulate the following standard assumptions that will ensure existence and uniqueness for the solution of the state equation. 
\begin{Hypothesis}
\label{hp:onbandsigma}
$b\colon [0,T] \times  H \times \Lambda \to H$ is a continuous function and satisfies, for some $C>0$,
\[
\begin{array}{l}
|b(s,x,a) - b(s,y,a)| \leq C |x-y|, \\[3pt]
|b(s,x,a)| \leq C (1+|x|),
\end{array}
\]
for all $x,y \in H$, $s\in [0,T]$, $a\in\Lambda$. $\sigma\colon [0,T]\times H \to \mathcal{L}_2(U_0, H)$ is continuous and, for some $C>0$,
\[
\begin{array}{l}
\|\sigma(s,x) - \sigma(s,y)\|_{\mathcal{L}_2(U_0, H)} \leq C |x-y|, \\[3pt]
\|\sigma(s,x)\|_{\mathcal{L}_2(U_0, H)} \leq C (1+|x|),
\end{array}
\]
for all $x,y \in H$, $s\in [0,T]$.
\end{Hypothesis}

\medskip
Let us fix for the moment a predictable process $a = a(\cdot): [s,T] \times \Omega 
\rightarrow \Lambda$,
 where the dot refers to the 
time variable. $a$ will indicate in the sequel an admissible control
in a sense to be specified.

We consider the state equation
\begin{equation}
\label{eq:state}
\left \{
\begin{array}{l}
\ud \X(t) = \left ( A\X(t)+ b(t,\X(t),a(t)) \right ) \ud t + \sigma(t,\X(t)) \ud \W_Q(t)\\[5pt]
\X(s)=x.
\end{array}
\right.
\end{equation}


The solution of (\ref{eq:state}) is understood in the mild sense,
 so an $H$-valued 
adapted  strongly measurable process $\X(\cdot)$ is a solution if
\[
\mathbb{P} \left \{ \int_s^T \left (|\X(r)| + | b(r,\X(r), a(r))| + \| \sigma(r,\X(r))\|_{\mathcal{L}_2(U_0, H)}^2\right) \ud r <+\infty  \right \} = 1
\]
and
\begin{equation}
\label{eq:state-mild}
\X(t) =  e^{(t-s)A}x + \int_s^t e^{(t-r)A} b(r,\X(r),a(r)) \ud r
 + \int_s^t e^{(t-r)A} \sigma(r,\X(r)) \ud \W_Q(r)
\end{equation}
 $\mathbb{P}$-a.s. for every  $ t \in [s,T]$.

 Thanks to Theorem 3.3 of \cite{GawareckiMandrekar10}, 
 given Hypothesis \ref{hp:onbandsigma}, 
there exists a unique solution
 $\X(\cdot; s,x, a(\cdot))$ of (\ref{eq:state}), which
 admits a continuous modification. So for us $\X$
can always  be considered as a continuous  process.

\begin{Remark} 
\label{rm:whitenoise}

\begin{itemize}
\item[(i)] When the operator $Q$ is trace class, the norm $\mathcal{L}_2(U_0, H)$ that appears in Hypothesis \ref{hp:onbandsigma} (assumptions on $\sigma$) can be controlled using the norm $\mathcal{L}(U, H)$ so the Lipschitz and the growing conditions imposed on $\sigma$ follows from similar hypotheses involving the norm of $\mathcal{L}(U, H)$. 
 
\item[(ii)] When the operator $Q$ is the identity, $\W_Q$ is naturally associated to a space time white noise, see also \cite{Walsh86} and \cite{quer}. In this case the norm $\mathcal{L}(U, H)$ does not control anymore the norm $\mathcal{L}_2(U_0, H)$ but still something can be done, in certain cases, to weaken the Lipschitz and the growing conditions imposed on $\sigma$,  asking something more on the semigroup. For example (see  \cite{GawareckiMandrekar10} Section 3.10), when $\sigma$ only depends on $\X$, if we only ask for linear growth and Lipschitz continuity w.r.t. the norm of $\mathcal{L}(U, H)$, we have existence, uniqueness and regularity results for the stochastic evolution equation similar to the ones we use in this work if the following condition is satisfied:
 $\| S(t) \sigma (x) \|_{\mathcal{L}_2(U, H)} \leq K(t) C(1+|x|)$ and $\| S(t) \sigma (x) - S(t) \sigma (y)  \|_{\mathcal{L}_2(U, H)} \leq K(t) C(|x-y|)$ where $K$ is a real function s.t. $\int_0^1 K^2(r) r^{-2\alpha} \ud r <+\infty$ for some $\alpha \in (0,1/2)$.
 
 A simple case in which such a condition is verified (see \cite{DaPratoZabczyk96} Section 11.2) is given by a one-dimensional heat equation on the segment $(0,1)$ of the following form:
 \[
 dX (t, \xi) = \left [  \frac{\partial^2 X}{\partial \xi^2}(t,\xi) + f(\xi, X(t,\xi))  \right ] dt + b (\xi, X(t,\xi)) dW(t,\xi),
 \]
 with zero Dirichlet boundary conditions and some square integrable initial datum where $W$ is cylindrical Wiener process on $U=L^2(0,1)$. 

\item[(iii)] In the framework described by item (i), 
 all our results do not extend automatically. Indeed when $Q$ is the identity and $\sigma$ has linear growth w.r.t. the norm of $\mathcal{L}(U, H)$,
 the stochastic integral appearing in the definition $\Y$ given in (\ref{eq:defY}) is not defined in $H$ but only in a larger space. This means in particular that the decomposition of $\X$ described in Theorem \ref{th:exlmIto} is nomore a decomposition in $H$-valued addenda. To include the ``purely cylindrical'' case we  need to describe the evolution of the solution of the stochastic equation in a larger Hilbert space and introduce there consistent notions. This will
the object of a paper  in preparation.
\end{itemize}
\end{Remark}

Setting
$b := b(\cdot, \X(\cdot; s,x, a(\cdot)), a(\cdot))$, 
$\sigma := \sigma(\cdot, \X(\cdot; s,x, a(\cdot)))$
then $\X$ fulfills
 (\ref{eq:hp-per-existence-mild})
and it is of type \eqref{eq:SPDE-mild}.
The following corollary  is just a particular case of Corollary 
\ref{lm:Ito-pre}, which is reformulated  here for the reader convenience.
\begin{Corollary}
\label{lm:Ito}
Let $\Lambda$ be a Polish space and
assume that  $b$ and $\sigma$  satisfy the Hypothesis \ref{hp:onbandsigma}. 
Let
$a:[s,T] \times \Omega \rightarrow \Lambda$ be predictable 
and $x\in H$.
Let $\X(\cdot)$ denote $\X(\cdot; s,x, a(\cdot))$.
Assume that $f\in C^{1,2}([0,T] \times H)$ \\
with $\partial_xf \in 
UC([0,T] \times H,D(A^*))$.
Then,  $\mathbb{P}-a.s.$,
\begin{multline}
\label{eq:first-Dinkin}
f(t,\X(t)) =  f(s,x)  + \int_s^t \partial_r 
f(r,\X(r)) \ud r  +  \int_s^t \left\langle A^* \partial_x f(r,\X(r)), 
\X(r) \right\rangle \ud r
 +  \int_s^t \left\langle \partial_x f(r,\X(r)), b(r, \X(r), a(r)) \right\rangle \ud r\\
+ \frac{1}{2}  \int_s^t \text{Tr} \left [\left ( \sigma (r, \X(r)) {Q}^{1/2} \right )
\left ( \sigma (r, \X(r)) Q^{1/2}  \right)^* \partial_{xx}^2 f(r,\X(r)) \right ] \ud r
+ \int_s^t \left\langle \partial_x f(r,\X(r)), \sigma(r, \X(r)) \ud \W_Q(r) \right\rangle.
\end{multline}
\end{Corollary}
Let $l\colon [0,T] \times H \times \Lambda \to \mathbb{R}$ 
be a measurable function and $g\colon H\to\mathbb{R}$ a continuous
 function. $l$ is called the running cost and $g$ the terminal cost. \\
We introduce now the class $ \mathcal{U}_s$ of admissible controls.
It is constituted by $a:[s,T] \times \Omega \rightarrow \Lambda$
such that  for $\omega$ a.s.
$(r, \omega) \mapsto l(r, \X(r,s,x, a(\cdot)), a(r)) + 
g(\X(T, s, x, a(\cdot))) $ is $\ud r \otimes \ud \P$- is 
quasi-integrable. This means that, either 
its positive or negative part are integrable. \\
We 
want to determine a {\em minimum} over all $a(\cdot) \in \mathcal{U}_s$, of the cost functional
\begin{equation}
\label{eq:functional}
J(s,x;a(\cdot))=\mathbb{E}\bigg[ \int_{s}^{T} l(r, \X(r;s,x,a(\cdot)), a(r)) \ud r + g (\X(T;s,x,a(\cdot))) \bigg].
\end{equation}

The value function of this problem is defined as
\begin{equation}
\label{eq:def-valuefunction}
V(s,x) = \inf_{a(\cdot)\in \mathcal{U}_s} J(s,x;a(\cdot)).
\end{equation}

\begin{Definition}
If $a^*(\cdot)\in \mathcal{U}_s$ minimizes (\ref{eq:functional}) among the controls in $\mathcal{U}_s$, i.e. if $J(s,x;a^*(\cdot)) = V(s,x)$, we say that the control $a^*(\cdot)$ is \emph{optimal} at $(s,x)$.
In this case the pair $(a^*(\cdot), \X^*(\cdot))$, where $\X^*(\cdot):= \X(\cdot; s,x, a^*(\cdot))$, is called an \emph{optimal couple} (or optimal pair) at $(s,x)$.
\end{Definition}

\subsection{The HJB equation}

The HJB equation associated with the minimization problem above is
\begin{equation}
\label{eq:HJB}
\left\{
\begin{array}{l}
\partial_s v + \left \langle  A^* \partial_x v, x \right\rangle + \frac{1}{2} Tr \left [ \, \sigma(s,x) \sigma^*(s,x) \partial_{xx}^2 v \right ] \\[3pt]
\qquad\qquad  + \inf_{a\in \Lambda }\Big\{  \left \langle \partial_x v, b(s,x,a) \right\rangle + l(s,x,a) \Big\}=0,\\ [8pt]
v(T,x)=g(x).
\end{array}
\right.
\end{equation}
In the above equation $\partial_xv$ [resp. $\partial^2_{xx} v$] is the 
[resp. second] Fr\'echet 
derivatives of $v$ w.r.t. the $x$ variable; it is identified with elements of $H$
 [resp. with a symmetric bounded operator on $H$].
 $\partial_s v$ is the derivative w.r.t. the time variable.
For $(t,x,p,a)\in [0,T]\times H \times H\times \Lambda$, the term 
\begin{equation}
\label{eq:def-CV-Hamiltonian}
F_{CV}(t,x,p,a):=  \left \langle p, b(t,x,a) \right\rangle + l(t,x,a),
\end{equation}
is called the \emph{current value Hamiltonian} of the system and its infimum over $a\in \Lambda$ 
\begin{equation}
\label{eq:def-Hamiltonian}
F(t,x,p):= \inf_{a\in \Lambda }\left\{  \left \langle p, b(t,x,a) \right\rangle + l(t,x,a) \right\}
\end{equation}
is called the \emph{Hamiltonian}. Using this notation the HJB equation, (\ref{eq:HJB}) 
can be rewritten as
\begin{equation}
\label{eq:HJB-with-F}
\left\{
\begin{array}{l}
\partial_s v+  \left \langle A^* \partial_x v, x \right\rangle + \frac{1}{2} Tr \left [ \sigma(s,x) \sigma^*(s,x) \partial^2_xv \right ] + F(s,x,\partial_xv)=0,\\[6pt]
v(T,x)=g(x).
\end{array}
\right.
\end{equation}

\begin{Hypothesis}
\label{hp:Hamiltonian}
The value function is always finite and the Hamiltonian $F(t,x,p)$
 is well-defined and finite  for all $\left( t,x,p\right)
 \in \left[ 0,T\right] \times H \times H$. Moreover it is supposed
to be continuous.
\end{Hypothesis}

We introduce the operator $\mathscr{L}_0$ on $C([0,T]\times H)$ defined as
\begin{equation}
\label{eq:def-L0}
\left \{
\begin{array}{l}
D(\mathscr{L}_0):= \left \{ \varphi \in C^{1,2}([0,T]\times H) \; : \; \partial_x\varphi \in C([0,T]\times H ; D(A^*)) \right \} \\[8pt]
\mathscr{L}_0 (\varphi)(s,x) := \partial_s \varphi(s,x)\\[4pt]
\qquad\qquad\qquad\qquad+  \left \langle A^* \partial_x \varphi (s,x), x \right\rangle + \frac{1}{2} Tr \left [  \sigma(s,x) \sigma^*(s,x) \partial_{xx}^2 \varphi(s,x) \right ].
\end{array}
\right .
\end{equation}

The HJB equation (\ref{eq:HJB-with-F}) 
can be rewritten as 
\[
\left \{
\begin{array}{l}
\mathscr{L}_0 (v) (s,x) + F(s,x, \partial_xv(s,x)) =0\\[5pt]
v(T,x) = g(x).
\end{array}
\right .
\]

\subsection{Strict and strong solutions}

For some $h\in C([0,T]\times H)$ and $g\in C(H)$ we consider the following Cauchy problem
\begin{equation}
\label{eq:simil-HJB-con-h}
\left \{
\begin{array}{l}
\mathscr{L}_0 (v) (s,x) = h(s,x)\\[5pt] 
v(T,x) = g(x).
\end{array}
\right .
\end{equation}

\begin{Definition}
We say that $v \in C([0,T]\times H)$ is a \emph{strict solution} of (\ref{eq:simil-HJB-con-h}) if $v\in D(\mathscr{L}_0)$ and (\ref{eq:simil-HJB-con-h}) is satisfied.
\end{Definition}

\begin{Definition}
\label{def:strong-sol}
Given $h\in C([0,T]\times H)$ and $g\in C(H)$ we say that
  $v \in C^{0,1}([0,T]\times H)$ 
 with  $\partial_x v \in UC([0,T)\times H; D(A^*))$ 
 is a \emph{strong solution} of (\ref{eq:simil-HJB-con-h}) if there exist three sequences: $ \{v_n \} \subseteq D(\mathscr{L}_0)$, $\{h_n\} \subseteq C([0,T]\times H)$ and $ \{ g_n \} \subseteq C(H)$ fulfilling the following.
\begin{enumerate}
\item[(i)] For any $n\in\mathbb{N}$, $v_n$ is a strict solution of the problem
\begin{equation}
\label{eq:approximating}
\left \{
\begin{array}{l}
\mathscr{L}_0 (v_n) (s,x) = h_n(s,x)\\[5pt]
v_n(T,x) = g_n(x).
\end{array}
\right .
\end{equation}
\item[(ii)] The following convergences hold:
\[
\left \{
\begin{array}{ll}
v_n\to v & \text{in} \;\; C([0,T]\times H)\\
h_n\to h & \text{in} \;\; C([0,T]\times H)\\
g_n\to g & \text{in} \;\; C(H).
\end{array}
\right .
\]
\end{enumerate}



\end{Definition}

\subsection{Decomposition for solutions of the HJB equation}

\begin{Theorem}
\label{th:decompo-sol-HJB}
Consider $h\in C([0,T]\times H)$ and $g\in C(H)$. 
Assume that Hypothesis \ref{hp:onbandsigma} is satisfied.
 Suppose that 
 $v \in C^{0,1}([0,T] \times H)$ 
 with $\partial_x v \in UC([0,T)\times H; D(A^*))$  is 
a strong solution of (\ref{eq:simil-HJB-con-h}).
 Let $\X(\cdot):=\X(\cdot;t,x,a(\cdot))$ be the solution of (\ref{eq:state}) starting at time $s$ at some $x\in H$ and driven by some control $a(\cdot)\in \mathcal{U}_s$.
Assume that $b$ is of the form 
\begin{equation}\label{eq:b}
b(t,x,a) = b_g(t,x,a) + b_i(t,x,a),
\end{equation}
where $b_g$ and $b_i$ satisfy the following conditions.
\begin{itemize}
 \item[(i)] $\sigma(t,\X(t))^{-1} b_g(t,\X(t),a(t))$ is bounded (being $\sigma(t,\X(t))^{-1}$ the pseudo-inverse of $\sigma$);
 \item[(ii)] $b_i$ satisfies 
\begin{equation}
\label{eq:conv-ucp-b}
\lim_{n\to\infty} \int_s^\cdot \left \langle \partial_x v_n (r,\X(r)) - \partial_x v (r,\X(r)), b_i(r,\X(r), a(r)) \right\rangle \ud r =0 \qquad ucp,
\end{equation}
on $ [s,T_0]$ for each $s < T_0 <T$. 
\end{itemize}
Then
\begin{multline} \label{E86}
v(t, \X(t)) - v(s, \X(s)) =  v(t, \X(t)) - v(s, x) =
  \int_s^t h(r, \X(r)) \ud r \\
+ \int_s^t \left\langle \partial_x v(r, \X(r)), b(r,\X(r), a(r)) \right\rangle \ud r
+ \int_s^t \left\langle \partial_x v(r, \X(r)), \sigma (r,\X(r)) \ud \W_Q(r) \right\rangle.
\end{multline}
\end{Theorem}
\begin{Example} \label{EThSol-HJB}
Hypothesis $(i)$ and $(ii)$ of Theorem \ref{th:decompo-sol-HJB} are satisfied if the approximating sequence $v_n$ converges to $v$ in a stronger way. For example if $v$ is a strong solution of the HJB in the sense of Definition \ref{def:strong-sol} and, moreover, $\partial_x v_n$ converges to $\partial_xv$ in $C([0,T]\times H)$, then the convergence in point (ii) can be easily checked.
 The convergence of the spatial partial derivative
 is the  typical assumption     required in the standard strong solutions 
literature.
\end{Example}
\begin{Example} \label{EThSol-HJB1}
The assumptions of Theorem \ref{th:decompo-sol-HJB} are fulfilled
 if the following assumption is satisfied.
\[
\sigma(t,\X(t))^{-1} b(t,\X(t),a(t)) \text{ is bounded}
\]
for all choice of admissible controls $a(\cdot)$.
In this case we apply Theorem  \ref{th:decompo-sol-HJB}  
with $b_i = 0$  and $b=b_g$.
\end{Example}
\begin{proof}[Proof of Theorem \ref{th:decompo-sol-HJB}]
We fix $T_0$ in $(s,T)$.
We  denote  by  $v_n$ the sequence of
smooth solutions of the approximating problems prescribed 
by Definition \ref{def:strong-sol}, which converges to $v$.
Thanks to Corollary \ref{lm:Ito}, every $v_n$ verifies
for $t \in [s,T_0]$,
\begin{eqnarray}
\label{eq:v-Ito-pre}
v_n(t,\X(t)) &=&  v_n(s,x)  + \int_s^t \partial_r {v_n}(r,\X(r)) \ud r \nonumber\\ 
 &+&   \int_s^t \left\langle A^* \partial_x v_n(r,\X(r)), \X(r) \right\rangle \ud r
 +  \int_s^t \left\langle \partial_x v_n(r,\X(r)), b(r, \X(r), a(r)) 
\right\rangle \ud r  \nonumber\\
&& \\
&+& \frac{1}{2}  \int_s^t \text{Tr} \left [\left ( \sigma (r, \X(r)) {Q}^{1/2}
 \right )
\left ( \sigma (r, \X(r)) Q^{1/2}  \right)^* \partial_{xx}^2 
v_n(r,\X(r)) \right ] \ud r  \nonumber\\
&+& \int_s^t \left\langle \partial_x v_n(r,\X(r)), \sigma(r, \X(r)) \ud \W_Q(r)
 \right\rangle, \ t \in [s,T]. \qquad \mathbb{P}-{\rm a.s.} \nonumber
\end{eqnarray}
Using Girsanov's Theorem (see \cite{DaPratoZabczyk92} Theorem 10.14) we can observe that
\[
\beta_Q(t) := W_Q(t) + \int_s^t \sigma(r,\X(r))^{-1} b_g(r,\X(r), a(r)) \ud r,
\]
is a $Q$-Wiener process w.r.t. a probability $\mathbb{Q}$ equivalent to $\mathbb{P}$ on the whole interval $[s,T]$.
 We can rewrite (\ref{eq:v-Ito-pre}) as
\begin{multline}
\label{eq:v-Ito}
v_n(t,\X(t)) =  v_n(s,x)  + \int_s^t \partial_r {v_n}(r,\X(r)) \ud r\\  +  \int_s^t \left\langle A^* \partial_x v_n(r,\X(r)), \X(r) \right\rangle \ud r
 +  \int_s^t \left\langle \partial_x v_n(r,\X(r)), b_i(r, \X(r), a(r)) \right\rangle \ud r,\\
+ \frac{1}{2}  \int_s^t \text{Tr} \left [\left ( \sigma (r, \X(r)) {Q}^{1/2} \right )
\left ( \sigma (r, \X(r)) Q^{1/2}  \right)^* \partial_{xx}^2 v_n(r,\X(r)) \right ] \ud r\\
+ \int_s^t \left\langle \partial_x v_n(r,\X(r)), \sigma(r, \X(r)) \ud \beta_Q(r) \right\rangle. \qquad \mathbb{P}-a.s.
\end{multline}
Since $v_n$ is a strict solution of (\ref{eq:approximating}), the expression
above gives
\begin{multline}
v_n(t,\X(t)) =  v_n(s,x)  + \int_s^t {h_n}(r,\X(r)) \ud r\\ +  \int_s^t \left\langle \partial_x v_n(r,\X(r)), b_i(r, \X(r), a(r)) \right\rangle \ud r
+ \int_s^t \left\langle \partial_x v_n(r,\X(r)), \sigma(r, \X(r)) \ud \beta_Q(r)\right\rangle.
\end{multline}
Since we wish to take the limit for $n\to\infty$, we define
\begin{multline}
M_n(t) := v_n(t,\X(t)) -  v_n(s,x)  - \int_s^t {h_n}(r,\X(r)) \ud r -  \int_s^t \left\langle \partial_x v_n(r,\X(r)), b_i(r, \X(r), a(r)) \right\rangle \ud r.
\end{multline}
$\{ M_n \}_{n\in\mathbb{N}}$ is a sequence of real $\mathbb{Q}$-local
 martingales converging ucp, thanks to the definition of strong solution
 and Hypothesis (\ref{eq:conv-ucp-b}), to 
\begin{multline} \label{E72}
M(t) := v(t,\X(t)) -  v(s,x)  - \int_s^t {h}(r,\X(r)) \ud r -  \int_s^t \left\langle \partial_x v(r,\X(r)), b_i(r, \X(r), a(r)) \right\rangle \ud r.
\end{multline} 
Since the space of real continuous local martingales equipped with 
the ucp topology is closed (see e.g. Proposition 4.4 of \cite{GozziRusso06}) 
then 
$M$ is a continuous $\mathbb{Q}$-local martingale.

\medskip
We have now gathered all the ingredients to conclude the proof.
As in Section  \ref{sec:SPDEs}, we set $\bar \nu_0 = D(A^*)$, 
$\nu = \bar \nu_0 \hat\otimes_\pi \R, \bar \chi =
\bar \nu_0 \hat \otimes_\pi \bar \nu_0.$

Corollary \ref{cor:X-barchi-Dirichlet} ensures that $\X(\cdot)$ is 
a $\nu$-weak Dirichlet process admitting a 
$\bar \chi$-quadratic variation 
with decomposition $\M + \A$
 where $\M$  is 
 the local martingale  (with respect to $\P$) defined by
 $\M(t) = x + \int_s^t \sigma (r, \X(r)) \ud \W_Q(r)$
and $\A$ is a $\nu$-martingale-orthogonal process.
Now 
$$\X(t) = \tilde \M(t) 
 + \V(t) + \A(t),
 t \in [s,T_0],$$
where $\tilde \M(t) =  x + \int_s^t \sigma (r, \X(r)) \ud \beta_Q(r)$
and  $\V(t) = - \int_s^t b_g(r, \X(r),a(r)) dr$, $t \in [s,T_0]$
is a bounded variation process. So by Proposition 
\ref{PChainRule} (i), $\tilde \M$ is a $\Q$-local martingale
and
by Proposition \ref{P421} 1.,
$\V$ is   a $\Q-\nu$-martingale orthogonal process.
By Remark \ref{R421} $\V + \A$ is a $\Q-\nu$-martingale orthogonal process
and $\X$ is a $\nu$-weak Dirichlet process with 
local martingale part $\tilde M$, with respect to $\Q$.
Still under $\Q$,  
Theorem \ref{th:prop6} 
 ensures that the process $v(\cdot, \X(\cdot))$ is a real
 weak Dirichlet process whose local martingale part being equal to
\[
N(t) = \int_s^t \left\langle \partial_x v(r,\X(r)), \sigma(r, \X(r)) \ud \beta_Q(r)\right\rangle.
\]
On the other hand, with respect to $\mathbb Q$,
\eqref{E72} implies that
\begin{multline}
v(t,\X(t)) = \bigg [  v(s,x)  + \int_s^t {h}(r,\X(r)) \ud r
+  \int_s^t \left\langle \partial_x v(r,\X(r)), b_i(r, \X(r), a(r)) \right\rangle \ud r \bigg ] + N(t),
\end{multline}
is a decomposition of $v(\cdot,\X(\cdot))$ as $\mathbb Q$-
semimartingale, which is also in particular, a  $\mathbb Q$-weak
 Dirichlet process. By Proposition \ref{rm:ex418}
 such a decomposition is unique and so
 \begin{multline}
M(t) = N(t) = \int_s^t \left\langle \partial_x v(r,\X(r)), \sigma(r, \X(r)) \ud \beta_Q(r)\right\rangle\\ 
= \int_s^t \left\langle \partial_x v(r,\X(r)), b_g(r, \X(r),a(r)) \ud r\right\rangle
+ \int_s^t \left\langle \partial_x v(r,\X(r)), \sigma(r, \X(r))
 \ud \W_Q(r)\right\rangle.
\end{multline}
This shows \eqref{E86} for $t \in [s,T_0]$.
Letting $T_0$ go to $T$ allows to 
 conclude the proof of Theorem \ref{th:decompo-sol-HJB}.
\end{proof}


\subsection{Verification Theorem}
\begin{Theorem}
\label{th:verification}
Assume that Hypotheses \ref{hp:onbandsigma} and \ref{hp:Hamiltonian} 
are satisfied.
Let  $v \in C^{0,1}([0,T]\times H)$ 
with $\partial_x v \in UC([0,T]\times H; D(A^*))$  be a
 strong solution of (\ref{eq:HJB}). Assume that for all initial data
 $(s,x)\in [0,T]\times H$ and every control $a(\cdot)\in \mathcal{U}_s$ $b$
 can be written as $b(t,x,a) = b_g(t,x,a) + b_i(t,x,a)$ with $b_i$ and $b_g$
 satisfying hypotheses (i) and (ii) of Theorem \ref{th:decompo-sol-HJB}.
 Let $v$ such that
 $\partial_xv$
 has most polynomial growth in the $x$ variable. Then
we have the following.
\begin{enumerate}
\item[(i)]  $v\leq V$ on $[0,T]\times H$.
\item[(ii)]  Suppose that, for some $s\in [0,T)$, there exists a 
predictable process $a(\cdot) = a^*(\cdot)  \in \mathcal{U}_s$ 
such that, denoting $\X\left( \cdot;s,x,a^*(\cdot)\right)$ simply
 by $\X^*(\cdot)$,
we have
\begin{equation}
\label{eq:condsuff}
F\left( t, \X^*\left( t\right) ,\partial_xv\left( t,\X^*\left( t\right)
\right) \right) =F_{CV}\left( t,\X^*\left( t\right) ,\partial_xv\left(
t,\X^*\left( t\right) \right) ;a^*\left( t\right) \right),
\end{equation}
  $dt \otimes \ud \P$ a.e. 
Then  $a^*(\cdot)$
 is optimal at $\left( s,x\right) $; moreover 
$v\left( s,x\right) =V\left(s,x\right)$.
\end{enumerate}
\end{Theorem}
\begin{proof}
We choose a control $a(\cdot) 
\in \mathcal{U}_s$  and call $\X$ the related trajectory. Thanks to Theorem
 \ref{th:decompo-sol-HJB} we can write
 \begin{multline}
\label{ex66bis}
g(\X(T)) = v(T, \X(T)) = v(s, x) - \int_s^T F (r, \X(r), \partial_xv(r,\X(r))) \ud r \\
+ \int_s^T \left\langle \partial_x v(r, \X(r)), b(r,\X(r), a(r)) \right\rangle \ud r
+ \int_s^T \left\langle \partial_x v(r, \X(r)), \sigma (r,\X(r)) \ud \W_Q(r) \right\rangle.
\end{multline}
Since both sides of \eqref{ex66bis} are a. s. finite, we can add 
 $\int_s^T l(r, \X(r) , a(r)) \ud r$ to them, obtaining 
\begin{multline}
\label{eq:ex66}
g(\X(T)) + \int_s^T l(r, \X(r) , a(r)) \ud r 
 = v(s, x) + \int_s^T \left\langle \partial_x v(r, \X(r)), \sigma (r,\X(r)) \ud \W_Q(r) \right\rangle \\
+ \int_s^T \left(- F (r, \X(r), \partial_xv(r,\X(r))) +
 F_{CV} (r, \X(r), \partial_xv(r,\X(r)))\right) \ud r.
\end{multline}
Observe now that, by definition of $F$ and $F_{CV}$ we know that 
\[
- F (r, \X(r), \partial_xv(r,\X(r))) + F_{CV} (r, \X(r), \partial_xv(r,\X(r)))
\]
is always positive. So its expectation always exists even if it 
could be $+\infty$, but not $-\infty$ on an event of positive probability.
This shows a posteriori that 
$ \int_s^T l(r, \X(r) , a(r)) \ud r $ cannot be $-\infty$
on a set of positive probability. \\
By \cite{DaPratoZabczyk92} Theorem 7.4, all the momenta of 
$\sup_{r \in [s,T]} \vert \X(r) \vert $ are 
finite.
On the other hand, $\sigma$ is Lipschitz-continuous,
  $v(s,x)$ is deterministic and, 
 since $\partial_x v$ has polynomial growth, 
  then
$$ \mathbb{E} \int_s^T \left\langle \partial_x v(r, \X(r)), \left ( \sigma (r,\X(r)) Q^{1/2} \right ) \left ( \sigma (r,\X(r)) Q^{1/2} \right )^* \partial_x v(r, \X(r)) \right\rangle \ud r
$$
is finite. Consequently, by Proposition \ref{PChainRule} (v) 
\[
 \int_s^\cdot \left\langle \partial_x v(r, \X(r)), \sigma (r,\X(r)) \ud
 \W_Q(r) \right\rangle
\]
is  a true martingale vanishing at $s$.  Consequently,
 its expectation is zero. 
 So the expectation of the right-hand side of (\ref{eq:ex66})
 exists even if it could be $+\infty$; consequently the same holds for the left-hand side.\\
By definition of $J$, we have
\begin{multline}
\label{eq:finale}
J(s,x,a(\cdot)) = \mathbb{E} \bigg [ g(\X(T)) + \int_s^T l(r, \X(r) , a(r)) \ud r \bigg ] = v(s, x) \\
+ \mathbb{E} \int_s^T  \Big ( - F (r, \X(r), \partial_xv(r,\X(r))) + F_{CV}(r, \X(r), \partial_xv(r,\X(r)), a(r)) \Big ) \ud r.
\end{multline}
So minimizing $J(s,x,a(\cdot))$ over $a(\cdot)$ is equivalent to minimize
\begin{equation}
\label{eq:ultima-2}
\mathbb{E} \int_s^T  \Big ( - F (r, \X(r), \partial_xv(r,\X(r))) + F_{CV}(r, \X(r), \partial_xv(r,\X(r)), a(r)) \Big )\ud r,
\end{equation}
which is a non-negative quantity.
As mentioned above, the integrand of such an expression is always nonnegative and then a lower bound for (\ref{eq:ultima-2}) is $0$. If the conditions of point (ii) are satisfied such a bound is attained by the control $a^*(\cdot)$, that in this way is proved to be optimal.

Concerning the proof of (i), since the integrand in (\ref{eq:ultima-2}) is nonnegative,  (\ref{eq:finale}) gives
\[
J(s,x,a(\cdot)) \geq  v(s, x).
\]
Taking the inf over $a(\cdot)$ we get $V(s,x) \geq v(s,x)$, which concludes the proof.
\end{proof}

\begin{Remark} \label{RFeedback}
\begin{enumerate}
\item The first part of the proof does not make use that
$a$ belongs to ${\mathcal U}_s$, but only that
$r \mapsto l(r,\X(\cdot,s,x,a(\cdot)), a(\cdot))$ 
is a.s. strictly bigger then $-\infty$.  Under  that only assumption, 
$a(\cdot)$ is forced to  be admissible, i.e. to belong to
${\mathcal U}_s$.
\item Let $v$ be a strong solution of HJB equation.
Observe that the condition (\ref{eq:condsuff}) can be rewritten as
\[
a^*(t) \in \arg\min_{a\in \Lambda} \Big [ F_{CV} \left( t,\X^*\left( t\right), 
\partial_xv\left(
t,\X^*\left( t\right) \right) ;a \right) \Big ]
.\]
Suppose that for any $(t,y) \in [0,T] \times H $,
$\phi(t,y) =  \arg\min_{a \in \Lambda} \big  ( F_{CV}\left(t,y, 
\partial_xv(t,y);a\right) \big )$
is measurable and single-valued.
Suppose moreover that 
\begin{equation}  \label{EFeedAdm}
\int_s^T l(r,\X^*(r), a^*(r)) dr > - \infty \ {\rm a.s.}
\end{equation}
Suppose that the equation 
\begin{equation}
\label{eq:stateFeed}
\left \{
\begin{array}{l}
\ud \X(t) = \left ( A\X(t)+ b(t,\X(t),\phi(t,\X(t) \right ) \ud t + \sigma(t,\X(t)) \ud \W_Q(t)\\[5pt]
\X(s)=x.
\end{array}
\right.
\end{equation}
admits a unique mild solution $\X^*$.
Now  \eqref{EFeedAdm} and Remark \ref{RFeedback}
imply that $a(\cdot)^*$ is admissible.
Then $\X^*$ is the optimal trajectory of the state variable and $a^*(t) =
 \phi(t, \X^*(t)), 
t \in [0,T]$ is the optimal control. The function $\phi$ is the \emph{optimal feedback} of the system since it gives the optimal control as a function of the state.
\end{enumerate}
\end{Remark}

\begin{Remark} 
Observe that, using exactly the same arguments we used in this section one could treat the (slightly) more general case in which $b$ has the form:
\[
b(t,x,a)= b_0(t,x) + b_g(t,x,a) + b_i(t,x,a).
\]
where $b_g$ and $b_i$ satisfy condition of Theorem \ref{th:decompo-sol-HJB}
and $b_0: [0,T] \times H \rightarrow H$ is continuous. In this case the addendum $b_0$ can be included in the expression of $\mathscr{L}_0$ that becomes the following
\begin{equation}
\label{eq:def-L0-b}
\left \{
\begin{array}{l}
D(\mathscr{L}_0^{b_0}):= \left \{ \varphi \in C^{1,2}([0,T]\times H) \; : \; \partial_x\varphi \in C([0,T]\times H ; D(A^*)) \right \} \\[6pt]
\mathscr{L}_0^{b_0} (\varphi) := \partial_s \varphi+  \left \langle A^* \partial_x \varphi, x \right\rangle + \left \langle \partial_x \varphi, b_0(t,x) \right\rangle + \frac{1}{2} Tr \left [ \sigma(s,x) \sigma^*(s,x) \partial_{xx}^2 \varphi \right ].
\end{array}
\right .
\end{equation}
Consequently in the definition of regular solution the operator $\mathscr{L}_0^{b_0}$ appears instead $\mathscr{L}_0$.
\end{Remark}

\begin{Remark} \label{RFinal}
In the definition of strong solution given
in Definition \ref{def:strong-sol}
one could substitute the assumption
$v  \in C^{0,1}([0,T] \times H)$ 
with $v  \in C^{0,1}([0,T) \times H)  \cap  C^{0}([0,T] \times H)$
such that $\partial_x v$ has polynomial growth. 
This, together with the fact that $\X$ has all moments
(see Theorem 7.4, Chapter 7 of \cite{DaPratoZabczyk92}),
would permit to pass to the limit $T_0 \rightarrow T$
in the last step of the proof of Theorem \ref{th:decompo-sol-HJB}.

\end{Remark}

\bigskip
{\bf ACKNOWLEDGEMENTS:} The research was partially 
supported by the ANR Project MASTERIE 2010 BLAN-0121-01.
It was partially written during the stay of the
second named author in the Bernoulli Center (EPFL Lausanne)
and at Bielefeld University, SFB 701 (Mathematik).\\
The work of the first named author was partially supported by the
 \emph{Post-Doc Research Grant}
of \emph{Unicredit \& Universities} and his research has been developed in the framework of the center of excellence LABEX MME-DII (ANR-11-LABX-0023-01).


\begin{small}

\bibliographystyle{plain}

\bibliography{biblio92}

\end{small}

\end{document}